\documentclass[a4paper,reqno,11pt]{amsart} 

\usepackage{graphicx} % Required for inserting images
\usepackage{float}
\usepackage[alphabetic, abbrev, msc-links, nobysame]{amsrefs}

\DeclareUnicodeCharacter{02BC}{'}
\DeclareUnicodeCharacter{211D}{\ensuremath{\mathbb{R}}}
\usepackage[foot]{amsaddr}
\usepackage{physics}
\usepackage{amsmath,amsthm}
\usepackage[all,cmtip]{xy}
\usepackage{amssymb}
\usepackage{geometry}
\usepackage{fancyhdr}
\usepackage{lipsum}
\usepackage{mathtools}
\mathtoolsset{showonlyrefs} % number only equations that are referenced with \eqref
\usepackage{relsize}
\usepackage{physics}
\allowdisplaybreaks
\usepackage{braket}
\usepackage{bm}
\usepackage{esint}
\usepackage{enumitem}
\usepackage{mathrsfs}
\usepackage{xcolor}
\usepackage{bbold}
\usepackage{hyperref}
\usepackage[alphabetic, abbrev, msc-links, nobysame]{amsrefs}

\numberwithin{equation}{section}
\newtheorem{thm}{Theorem}[section]

\newtheorem{cor}[thm]{Corollary}
\newtheorem{definition}[thm]{Definition}

\newtheorem{lemma}[thm]{Lemma}

\newtheorem{remark}[thm]{Remark}
\newtheorem{prop}[thm]{Proposition}
\newtheorem{corollary}[thm]{Corollary}
\newtheorem{example}[thm]{Example}

\def\bi{
	\begin{itemize}}
\def\ei{
	\end{itemize}}
\def\bem{
	\begin{emuerate}}
\def\eem{
	\end{enumerate}}
\def\beq{
	\begin{eqnarray*}}
\def\eeq{
	\end{eqnarray*}}

\def\i{itemize}
\def\N{\mathbb{N}} %naturals
\def\Z{\mathbb{Z}} %integers
\def\R{\mathbb{R}} %reals
\def\C{\mathbb{C}} %complexes
\def\P{\mathbb{P}} %probability
\def\E{\mathbb{E}} %expectation
\def\F{\mathcal{F}} %Filtration sigma algebra
\def\Ru{\mathcal{R}} %Update operator

\def\mc{\mathcal}
\def\var{\text{var}}

\def\1{\mathbb{1}}
\DeclarePairedDelimiter{\Parentheses}{(}{)} \newcommand{\pr}{\Parentheses*}

\DeclarePairedDelimiter{\Bracket}{[}{]}
\newcommand{\bracket}{\Bracket*}

\title{From Individual-Based Models to General Stochastic Reaction Diffusion Equations}
\date{}
\author[A. Gonz\'{a}lez Casanova]{Adri\'{a}n Gonz\'{a}lez Casanova\textsuperscript{*}} \email{agonz591@asu.edu}
\author[J. Yang]{Johnny (Yifan) Yang\textsuperscript{*}}\email{jyang436@asu.edu}

\address{\textsuperscript{*}School of Mathematical and Statistical Sciences,
Arizona State University, Tempe, AZ 85287-1804, USA}
\begin{document}

\begin{abstract}
In this article, we construct, for a broad class of one-dimensional stochastic reaction–diffusion equations, a sequence of spatially-structured individual-based models whose rescaled empirical processes converge to the solution of the equation.

The purpose of this article is to show that, whenever the rescaled one-step mean and variance of a reasonable individual-based model converge to those of an SDE with drift \(b(x)x(1-x)+\theta^+(1-x)-\theta^-x\) and noise coefficient \( \sqrt{\sigma(x)x(1-x)} \), under certain moment conditions there also exists a corresponding multi-island model, composed of independent and identically distributed copies of the individual-based model on each island, with selection and mutation scaled to the diffusive time scale, coupled through migration, which converges to an SPDE with a Laplacian plus the same drift and noise coefficients, provided this SPDE is unique in law.
\end{abstract}
\maketitle

\tableofcontents
\section{Introduction}

Cannings models~\cite{cannings1974latent} provide a flexible generalization of the Wright–Fisher model. These models have played a central role in the study of genealogies, particularly through their connection to Kingman~\cite{kingman1982coalescent} and $\Lambda$-coalescents~\cites{pitman1999coalescents,sagitov1999general,mohle2001classification}, and admit a rich variety of scaling limits. A key feature is that, although individuals choose parents  from the previous generation, these choices are no longer independent, leading to an increased variance in offspring distributions while maintaining neutrality.

The present work is motivated by combining two perspectives. On the one hand, we seek to incorporate various selection mechanisms and Gillespie’s principle~\cites{gillespie1974natural,gillespie1975natural} that offspring variance can depend on the current state of the population. On the other hand, we aim to retain the structural flexibility of Cannings models, allowing for correlations in the reproduction mechanism. This leads naturally to a new class of individual-based models in which the correlation structure of reproduction depends on the current frequency configuration.

Our first objective is to understand the scaling limits of such one dimensional models. One expects that, under suitable assumptions, these systems converge to stochastic differential equations (SDEs) whose drift and diffusion coefficients reflect selection, mutation, and frequency-dependent variance effects. The resulting diffusions are known in the literature, because they are the moment duals of Branching Coalescing Pairwise Branching processes \cite{GPP}, when the drift and the noise coefficient are polynomials with suitable signs.

The classical stepping stone model~\cite{kimura1964stepping} provides a spatial version of the Wright--Fisher model and may be viewed as a discrete counterpart of the stochastic Fisher--Kolmogorov--Petrovskii--Piskunov (FKPP) equation. In this model, colonies are arranged in space, individuals reproduce locally, and migration between neighboring colonies generates spatial interaction. Under diffusive rescaling of space and time, the corresponding type-frequency process converges to the stochastic FKPP equation. For voter models this was proved by Mueller and Tribe~\cite{mueller1995stochastic}, with long range interactions, and by Durrett and Fan~\cite{durrett2016genealogies}, with many individuals in each colony and interactions between neighboring colonies; without selection, the limit is the continuum stepping stone model of~\cite{shiga1988stepping}. The Wright--Fisher version with discrete generations is the case $\sigma\equiv1$ of Theorem~\ref{thm:canning_convergence}.

The main contribution of this paper is to lift this analysis to spatially structured populations. Inspired by the stepping stone model \cite{kimura1964stepping}, we consider systems of colonies in which each colony evolves according to an individual-based mechanism of the type described above, and migration couples neighboring colonies. We call those models the Cannings-Gillespie (CG) stepping stone models. Our goal is to identify conditions under which these spatial models converge, under diffusive scaling of space and time, to stochastic reaction–diffusion equations.

We establish a general convergence criterion that links the behavior of the underlying one dimensional model to the limiting SPDE. More precisely, we show that if the rescaled one-step mean and variance of the single-colony dynamics converge to those of an SDE with drift $b$ and diffusion coefficient $\sigma$, then under certain moment conditions (see Theorem~\ref{thm:main}), the associated spatial model, with selection and mutation scaled to the diffusive time scale, converges to an SPDE with Laplacian, drift $b$, and noise coefficient $\sqrt{\sigma}$, provided this SPDE is unique in law. This result can be viewed as a spatial extension of classical convergence criteria such as Möhle’s lemma~\cites{mohle2000total,mohle2001forward}, now formulated in the context of SPDEs.

As an application, we introduce and analyze the CG stepping stone models. Their scaling limits include every finite-invariant-interval equation covered by the uniqueness result of Athreya and Tribe~\cite{athreya2000uniqueness}, including equations covered after an affine change of the state variable; see Corollary~\ref{cor:AT conditions}. Thus every equation in this Athreya--Tribe class admits an individual-based spatial approximation.  Mueller stated in~\cite{mueller2009some}*{p.~115} that he ``firmly believe[s] that many SPDEs should be viewed as limits of particle systems.'' Our work offers this perspective for the Athreya--Tribe (AT) class. 

Indeed, for every SPDE in the AT class,  we give two individual-based models with different reproduction mechanisms. In the first, each colony has $N$ individuals, but only some of them  reach reproductive age, and their number depends on the type frequency (Section~\ref{sec:ex-pop}). In the second construction, every individual reaches reproductive age, but family sizes are random and unequal: the parental weights follow a Dirichlet law, and their variance depends on the frequency (Section~\ref{sec:ex-dir}).  Figures~\ref{fig:active-parent-generations} and~\ref{fig:dirichlet-cg-eq11} show the two mechanisms for the Gillespie FKPP equation (see \eqref{eq:gillespie-fkpp}).

\paragraph{Related work.}
Mueller and Tribe~\cite{mueller1995stochastic} showed that rescaled long range voter processes converge to the heat equation with Wright--Fisher noise $\sqrt{u(1-u)}$, with an additional drift proportional to $u(1-u)$ when the voter model is biased, and that rescaled long range contact processes converge to the KPP equation with branching noise $\sqrt{u}$. Durrett and Fan~\cite{durrett2016genealogies} placed $M_n$ individuals at each site of $L_n^{-1}\Z$, found the scalings under which this biased voter model converges to the stochastic FKPP equation, and used the limit to study genealogies in expanding populations. Fan~\cite{fan2021stochastic} extended these limits to metric graphs. Kliem~\cite{kliem2011convergence} obtained a class of drifts by perturbing the voter model, and Etheridge, V\'eber and Yu~\cite{etheridge2020rescaling} derived the stochastic FKPP equation in one dimension from the spatial $\Lambda$-Fleming--Viot process with selection. In all of these limits the noise is $\sqrt{\gamma u(1-u)}$ or $\sqrt{\gamma u}$ and its dependence on $u$ is fixed by the reproduction mechanism. In our limits the noise is $\sqrt{\sigma(u)u(1-u)}$, where $\sigma$ is a continuous function determined by the reproduction inside a colony.

Our results (Theorem~\ref{thm:main}) assume that the limiting equation is unique in law. For the Wright--Fisher SPDE with selection this follows from a duality of Shiga~\cite{shiga1988stepping}. Athreya and Tribe~\cite{athreya2000uniqueness} extended moment duality to a class of analytic coefficients that is stable under small perturbations, and noted that equations in this class are expected to arise from particle systems whose offspring mean and variance depend on the local density~\cite{athreya2000uniqueness}*{Section~1.1}. Corollary~\ref{cor:AT conditions} gives such approximations for every equation in their class with a finite invariant interval.

The rest of the paper is organized as follows. Section~\ref{sec:setup-main} introduces the one-colony paintbox  CG model and its spatial version, and states the main result, Theorem~\ref{thm:canning_convergence}, together with the general convergence result, Theorem~\ref{thm:main}. Section~\ref{sec:examples} gives two ways to build one-colony models that satisfy the hypotheses of Theorem~\ref{thm:canning_convergence}. Section~\ref{sec:proof-canning} derives  Theorem~\ref{thm:canning_convergence} from Theorem~\ref{thm:main}. The remaining sections prove Theorem~\ref{thm:main}. Section~\ref{sec:prelim-estimates} collects background on the limiting SPDE and on the discrete heat kernel. Section~\ref{sec:discrete-evolution} gives the discrete mild and weak formulations and the space--time moment bounds, which yield tightness in Section~\ref{sec:tightness}. Section~\ref{sec:limit-points} identifies every subsequential limit through the martingale problem and completes the proof. Appendix~\ref{app:heat-kernel-proofs} contains the proofs of the auxiliary heat-kernel and interpolation estimates.

\paragraph{Motivational example:}
Our first main example is what we call the \emph{Gillespie stepping stone model}. In this model, reproduction within each colony follows a zero-dimensional Gillespie-type population genetic mechanism, while migration occurs between neighboring colonies. We ask for the scaling limit when colonies become dense in space and time is accelerated diffusively. We show that the rescaled frequency process converges to a stochastic partial differential equation that we call the \emph{Gillespie FKPP equation},
\begin{align}
	\label{eq:gillespie-fkpp}\partial_t u = \frac{\alpha}{2}\Delta u-u(1-u)+\varepsilon\sqrt{u(1-u)(1-\kappa u)}\dot{W}.
\end{align}
Here, $\varepsilon>0$ controls the noise amplitude. This equation is a natural extension of the stochastic FKPP equation, but with a noise coefficient that reflects the frequency-dependent population size and hence the non-classical demographic variance. Its moment dual is a system of branching--coalescing Brownian motions with pairwise branching, which is an interesting object in its own right; it is a spatial version of the ancestral selection/efficiency graph of~\cite{casanova2020wright}. The following example gives a simple individual-based model corresponding to~\eqref{eq:gillespie-fkpp}. We will discuss its spatial version in more detail in Section~\ref{sec:ex-pop}; see Remark~\ref{rem:resource-example}. 

\begin{example}[A resource-limited stepping stone model]\label{ex:gillespie-resource}
	Fix a colony of $N$ individuals with a total amount of available resource equal to $R:=(1-\kappa)N\ge1$. Suppose there are two competing types. Type $A$ individuals require only $1-\kappa$ units of resource to reach reproductive age, while type $a$ individuals require one unit, where $\kappa\in[0,1)$. If the current frequency of type $A$ is $x\in[0,1]$, then the  number of individuals that can reach reproductive age is
	\begin{align*}
		N_R(x):=\max\left\{m\in\N:\; m x(1-\kappa)+m(1-x)\le R\right\}=\left\lfloor\frac{R}{1-\kappa x}\right\rfloor\le N.
	\end{align*}
	Thus the effective population size is no longer constant, but depends on the current type frequency.

	These $N_R(x)$ individuals are chosen uniformly at random among the $N$ individuals. Let $\xi^R_g$ be the type-$A$ frequency at generation $g$ and $W^R_g$ the type-$A$ frequency among its $N_R(\xi^R_g)$ chosen individuals. Each of the $N$ individuals of generation $g+1$ independently has type $A$ with probability $p_R(W^R_g)$, so
	\begin{align*}
		\xi^R_{g+1}=\frac{1}{N}\mathrm{Bin}\bigl(N,p_R(W^R_g)\bigr),
		\qquad
		p_R(x):=\frac{x}{x+(1+\beta_R)(1-x)},
	\end{align*}
	where $\beta_R\ge0$ is the selective advantage of type $a$.
	The key feature is that, uniformly in $x$ and for bounded $\beta_R$, the one-step variance is $x(1-x)/N_R(x)+O(\beta_R/R)$ (exactly $x(1-x)/N_R(x)$ when $\beta_R=0$), so it depends on the frequency through $N_R(x)$, while the mean change is $p_R(x)-x+O(\beta_R/R)$, whose leading term does not. In particular, taking $\beta_R=1/R$ and speeding up time by $R$, one obtains as $R\to\infty$ the diffusion (by a standard diffusion approximation, see e.g.~\cite{ethier2009markov}*{Chapter~7, Theorem~4.1})
	\begin{align*}
		dX_t=-X_t(1-X_t)\,dt+\sqrt{X_t(1-X_t)(1-\kappa X_t)}\,dB_t,
	\end{align*}
	whose noise coefficient is no longer the classical Wright--Fisher term $X_t(1-X_t)$ but carries the factor $1-\kappa X_t$. This is the Wright--Fisher diffusion with efficiency of~\cite{casanova2020wright}*{(1.2)}. It is also the limit of the model of~\cite{casanova2020wright} in which an offspring that would exceed the resource is rejected~\cite{alsmeyer2026two}*{Theorem~2.2}.
\end{example}

\begin{figure}[H]
	\centering
	\includegraphics[width=0.8\textwidth]{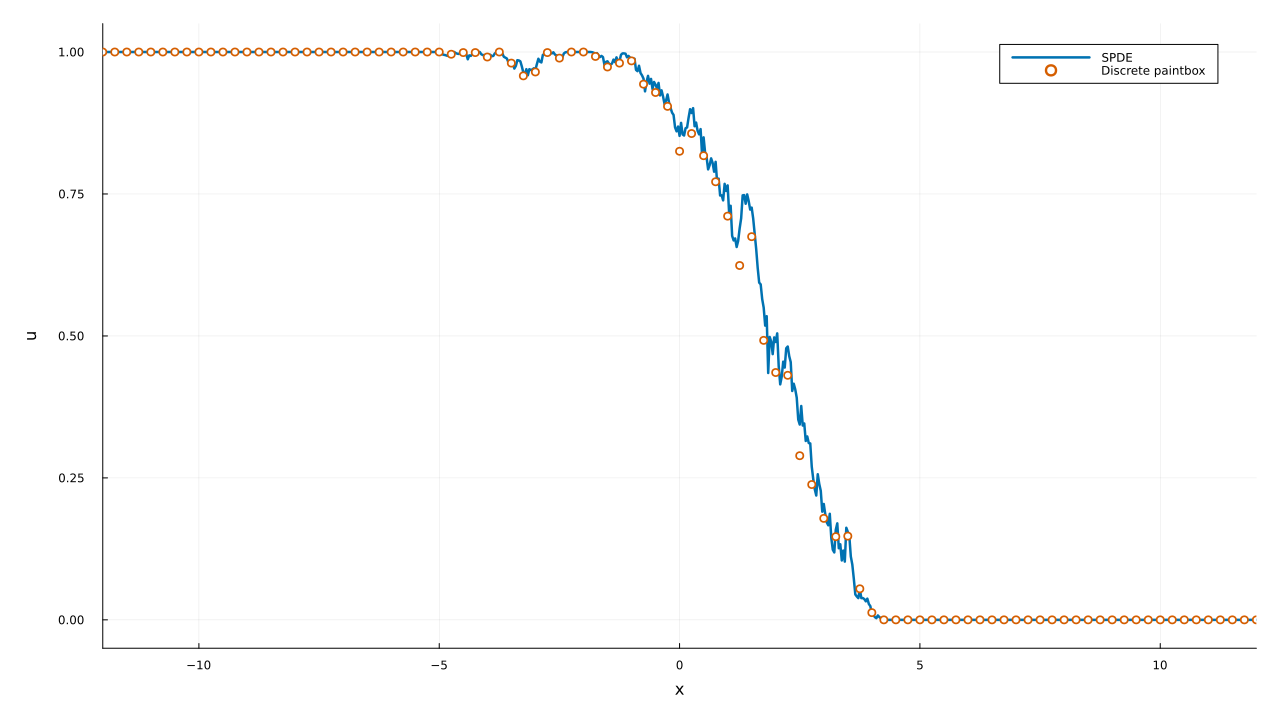}
	\caption{A sample of the solution of~\eqref{eq:gillespie-fkpp} with $\alpha=\kappa=\frac{1}{2}$ and $\varepsilon=\frac{1}{4}$, in blue, and of its individual-based approximation of Section~\ref{sec:ex-pop} with $N=1024$, in orange.}
	\label{fig:gillespie-fkpp-paintbox-comparison}
\end{figure}

\section{Setup and main result}\label{sec:setup-main}
In this section, we introduce the class of discrete population models that will serve as approximations to the spatial SPDEs studied in this paper. We begin with a one-colony paintbox Cannings-Gillespie model, in which reproduction is governed by an exchangeable random vector of offspring weights whose law may depend on the current type frequency. This framework allows both the selective bias and the effective offspring variance to vary with the population composition.

We then place independent copies of this update mechanism on the lattice and couple neighboring colonies through migration. After specifying the spatial rescaling, we state the main convergence result for the spatially interpolated frequency process and record concrete moment conditions on the one-colony model that will be verified in examples and used in the general convergence theorem proved later.

\subsection{\texorpdfstring{Paintbox Cannings-Gillespie (CG) model}{Paintbox Cannings-Gillespie (CG) model}}\label{sec:canning}
In this section, we will introduce the paintbox Cannings-Gillespie model, which is a generalization of the Cannings models in which the distribution of the weights depends on the frequency of the previous generation.

Let $N \in \N$  and suppose for any $x \in [0,1]$,  $\mathcal W(x) :=\mathcal W^N (x):= (w_1(x),\dots, w_N(x))$ is an exchangeable  $N$-tuple of  weights taking its values in
$$
	\left\{(w_1(x),\dots,w_N(x)): w_i(x) \ge 0,\, \text{ for all } 1 \le i \le N,\, \sum_{i = 1}^N w_i(x) = 1 \right\}.
$$
Assume that each individual has one of two types, either type $(A)$ or type $(a)$. In each generation $g \in \N$, the individuals are numbered by $i \in [N]:= \{1,\dots,N\}$, and denoted as $(i,g)$. A parental relation between individuals in generation $g$ and $g-1$ is defined in the following way. Let $K_{g}$ be the number of individuals in generation $g \in \N$ that are of type $A$, and let $X_g := \frac{K_g}{N}$ be the frequency of type $A$. Set $[N^{-1}]:=\{0,1/N,\dots,1\}$. Since the state space records only the type frequency, we relabel the individuals before each reproduction step so that, whenever $X_{g-1}=x\in[N^{-1}]$, the type-$A$ individuals are indexed by
\[
	I_x:=\{1,\dots,Nx\}\subset[N].
\]
Because $\mathcal W^N(x)$ is exchangeable, this relabeling does not change the law of the frequency process.

To determine the parental relation between individuals of generation $g-1$ and $g$, we first introduce a function $s_N:[0,1]\to (-1,1)$. Given $X_{g-1}=x$, sample $\mathcal W(x)$ independently of the past. Conditioned on $X_{g-1}=x$ and $\mathcal W(x)$, each individual $(j,g)$ independently chooses its parent according to
\begin{align*}
	\P \left((i,g-1) \text{ is parent of } (j,g)|X_{g-1} = x,\, \mathcal W(x) \right) = \frac{(1+\mathbf 1_{i \in I_x}s_N(x))w_i(x)}{1+s_N(x) W_A^N(x)},
\end{align*}
where $W_A^N(x) := \sum_{i \in I_x} w_i(x)$. In particular, conditioned on $X_{g-1}=x \in[N^{-1}]$ and $\mathcal W(x)$, individuals in generation $g$ inherit their parental type independently, with type-$A$ probability
\begin{align}\label{eq: pnxw}
	P_N(x,\mathcal W(x)):=\P \left(\text{the parent of $(i,g)$ is of type }A\mid X_{g-1}=x,\,\mathcal W(x)\right)
	=
	\frac{(1+s_N(x))W_A^N(x)}{1+s_N(x)W_A^N(x)}.
\end{align}
Let $\theta_N^+,\theta_N^-\in[0,1]$ satisfy $\theta_N^++\theta_N^-\le1$. After an offspring inherits its parental type, independently of the other offspring, the sampled weights, and the parental choices, it mutates from type $a$ to type $A$ with probability $\theta_N^+$ and from type $A$ to type $a$ with probability $\theta_N^-$. Thus its final type-$A$ probability is
\begin{align}\label{eq: mutation adjusted probability}
	\widehat P_N(x,\mathcal W(x))
	&:=\theta_N^+ +\bigl(1-\theta_N^+ - \theta_N^- \bigr) P_N(x,\mathcal W(x))\notag\\
	&=P_N(x,\mathcal W(x)) + \theta_N^ + \bigl(1-P_N(x,\mathcal W(x))\bigr) - \theta_N^- P_N(x,\mathcal W(x)).
\end{align}
Therefore, given $X_{g-1}$ and $\mathcal W(X_{g-1})$,
\begin{align}\label{eq:canning_conditional_binomial}
	X_g \sim N^{-1}\,\text{Binomial}\pr{N,\widehat P_N\pr{X_{g-1}, \mathcal W\pr{X_{g-1}}}}.
\end{align}

For $x\in[0,1]$, let
\begin{equation}\label{eq:cn scaling}
	c_N(x) := \E \left[\sum_{i = 1}^N w_i(x)^2 \right].
\end{equation}
In the neutral case, $c_N$ is the probability that two individuals have a common ancestor in the previous generation, given the frequency in that generation is $x$.

The following example demonstrates the method to model a population, whose size depends on the frequency, due to for example, type $A$ individuals consuming fewer resources.
\begin{example}[Dirichlet-Cannings-Gillespie]
	Imagine  that there is a frequency-dependent population size, skewed offspring distribution and classic selection $s_N(x)=s$ for some $s\in[0,1)$.

	In order to construct $\mathcal W$, we let
	\begin{align*}
		\hat w_i(x) = Y_iB_i(x),
	\end{align*}
	where $\{Y_i\}_{i \in [N]}$ are i.i.d.\ $\mathrm{Gamma}(a,1)$ random variables for some $a>0$, and $B_i(x) \sim \mathrm{Bernoulli}\bigl(\frac{1-\kappa}{1-\kappa x}\bigr)$, i.i.d.\ and independent of $\{Y_i\}_{i\in[N]}$, for some $\kappa \in (0,1)$. For $\sum_{i \in [N]} \hat w_i(x) >0$, we set
	\begin{align*}
		w_i(x):=  \frac{\hat w_i(x)}{\sum_{j \in [N]} \hat w_j(x)}.
	\end{align*}
	Else, set $w_i(x) = \frac{1}{N}$ for all $i \in [N]$. Since the $Y_i$ are i.i.d.\ $\mathrm{Gamma}(a,1)$, given $J(x):=\{i:B_i(x)=1\}\neq\emptyset$ the weights $(w_i(x))_{i\in J(x)}$ are $\mathrm{Dirichlet}(a,\dots,a)$ distributed, and $w_i(x)=0$ for $i\notin J(x)$.
\end{example}
We will assume in Sections~\ref{sec:canning}--\ref{sec:main-result}, \ref{sec:examples} and~\ref{sec:proof-canning} that there are fixed $b\in\mathcal C([0,1];\R)$, $\sigma\in\mathcal C([0,1];[0,\infty))$, constants $\theta^+,\theta^-\ge0$, and a positive sequence $h_N \rightarrow 0$ such that
\begin{align}\label{eq:bsigma def}
	\lim_{N\to\infty}\sup_{x\in[0,1]}
	\left(
	\abs{\frac{s_N(x)}{h_N^2}-b(x)}+\abs{\frac{c_N(x)}{h_N}-\sigma(x)}+\abs{\frac{\theta_N^+}{h_N^2}-\theta^+} +\abs{\frac{\theta_N^-}{h_N^2}-\theta^-}
	\right)
	=0.
\end{align}
We write
\begin{align}\label{eq: full CG drift}
	B(x):=b(x)x(1-x)+\theta^+(1-x)-\theta^-x, \qquad x\in[0,1],
\end{align}
for the full limiting drift. In particular, $B(0)=\theta^+\ge0$ and $B(1)=-\theta^-\le0$. For each $N \in \N$, let $\{X_g^N\}_{g \in \N}$ denote the frequency process of the CG model with parameters given as above. The one-colony chain may be realized by iterating i.i.d. copies of the following random update operator.
\begin{definition}[Update operator]\label{def: update operator}
	For each $x\in[N^{-1}]$, let $\mathcal W_x^N=(w_1^x,\dots,w_N^x)$ be an independent random vector with law $\mathcal W^N(x)$. Define
	\[
		W_A^N(x):=\sum_{i\in I_x} w_i^x.
	\]
	Conditioned on $\mathcal W_x^N$, let
	\[
		K_x^N\sim \mathrm{Binomial}\bigl(N,
		\widehat P_N(x,\mathcal W_x^N)\bigr).
	\]
	Assume moreover that the family $\{K_x^N\}_{x\in [N^{-1}]}$ is conditionally independent given $\{\mathcal W_x^N\}_{x\in [N^{-1}]}$. We then define the random map $\widetilde{\mathcal R}^N:[N^{-1}]\to [N^{-1}]$ by
	\[
		\widetilde{\mathcal R}^N(x):=\frac{K_x^N}{N},\qquad x\in [N^{-1}].
	\]
	Since $\mathcal W^N(x)$ is exchangeable, the law of $\widetilde{\mathcal R}^N(x)$ does not depend on the particular choice of the deterministic set $I_x$ with cardinality $Nx$. If $\{\widetilde{\mathcal R}_g^N\}_{g\ge 1}$ are i.i.d. copies of $\widetilde{\mathcal R}^N$, then the one-colony frequency process is given by
	\[
		X_g^N=\widetilde{\mathcal R}_g^N(X_{g-1}^N),\qquad g\ge 1.
	\]
\end{definition}
When the current frequency is $x$, we write $X_1:=\widetilde{\mathcal R}^N(x)$, and $\P_x$, $\E_x$ for probability and expectation under this one-step law.

\subsection{Spatial CG models}

For spatial models, let $N \in \mathbb N$ and $\Z = \{\dots, -1,0, 1,\dots\}$ be the  integer lattice. For any $f:\Z \to \R$, we  define the discrete Laplacian to be
\begin{align*}
	\Delta_d f(i) := f(i+1) - 2f(i) +f(i-1),\qquad i \in \Z.
\end{align*}
Here, the subscript stands for discrete.

For each site and each time step, let $V$ be uniform on $(0,1)$ and independent of $\widetilde{\mathcal R}^N$. Define the random function $U_N:[0,1]\to[N^{-1}]$ by
\[
	U_N(x):=\frac{\lfloor Nx\rfloor}{N}+\frac1N\1_{\{V<Nx-\lfloor Nx\rfloor\}},
	\qquad
	\overline{\mathcal R}^N(x):=\widetilde{\mathcal R}^N(U_N(x)).
\]
Then $U_N(x)\in[N^{-1}]$ and
\begin{align}\label{ineq:uniform_rv_estimates}
	\E[U_N(x)]=x,\qquad
	|U_N(x)-x|\le\frac1N,\qquad
	\E[(U_N(x)-x)^2]\le\frac1{4N^2}.
\end{align}
Moreover, $U_N(x)=x$ whenever $x\in[N^{-1}]$, and hence $\overline{\mathcal R}^N(x)=\widetilde{\mathcal R}^N(x)$ almost surely on $[N^{-1}]$.

On a common probability space $(\Omega,\mathcal F,\mathbb P)$, let $\{\overline{\mathcal R}^N_{i,g}\}_{(i,g)\in\Z\times\Z_+}$ be an i.i.d. family of random operators with common distribution $\overline{\mathcal R}^N$ defined above. Define the filtration by
\begin{align*}
	\mathcal F_g^N:=\sigma\pr{\overline{\mathcal R}_{i,k}^N:\, i\in\Z,\,0\le k<g},\qquad g\in\Z_+.
\end{align*}
In particular, $\{\overline{\mathcal R}_{i,g}^N\}_{i\in\Z}$ is independent of $\mathcal F_g^N$.

\textbf{Spatial frequency process:} Let us fix $\alpha \in (0,1)$ and suppose $X_0^N: \Z \to [0,1]$ is deterministic. For $g\in\Z_+$ and $i\in\Z$, define
\begin{align}\label{eq: discrete spatial frequency proces definition}
	X_{g+1}^N(i):=  \overline{\mathcal R}^N_{i,g}\pr{X_g^N(i) + \frac{\alpha}{2}\Delta_d X_g^N(i)}.
\end{align}
Hence $X_{g}^N \in \mathcal F_g^N$ for all $g \in \Z_+$.

In the mutation-free case $\theta_N^+=\theta_N^-=0$, the Wright--Fisher model with selection $s>0$ satisfies, for any $p \in [N^{-1}]$,
\begin{align*}
	\widetilde{\mathcal R}^N(p)\sim N^{-1}\text{Binom}\pr{N, \frac{\pr{1+\frac{s}{N^2}}p}{1+\frac{s}{N^2}p}}.
\end{align*}
In this mutation-free Wright--Fisher example, the migrated frequency in \eqref{eq: discrete spatial frequency proces definition} is first rounded without bias to an adjacent grid point, and the Wright--Fisher update is then applied at that grid point.

\subsection{\texorpdfstring{Main result: continuum spatial CG models}{Main result: continuum spatial CG models}}\label{sec:main-result}
Here we state our main result: the discrete CG models converge to the desired SPDEs on $\R$.

Let $\Z_N:= \Z h_N = \{\dots,0,h_N,2h_N,\dots\}$ be the scaled integer lattice. We define for $t \in \R_+$ and $x_i   \in  \Z_N$ that
\begin{align}\label{eq:xin}
	\xi^N_t(x_i):= X^N_{\lfloor t/h_N^2\rfloor}( x_i/h_N).
\end{align}
Furthermore, we denote for each $t\ge 0$, $\widetilde\xi^N_t$ the linear spatial interpolation for $\xi^N_t$.

For each $\lambda \in \R$, define the real valued function $e_\lambda(x):= \exp \pr{-\lambda |x|}$ for $x \in \R$, and  $\norm{f}_{(\lambda)}:= \sup_x \abs{f(x)}e_{\lambda}(x)$.  We define the collection of real valued functions on $\R$ with tempered growth to be
\begin{align*}
	\mathcal C_\R^\text{tem}:= \left\{f \in \mathcal C(\R)\text{ with }\norm{f}_{(\lambda)} < \infty\text{ for all } \lambda > 0 \right\}.
\end{align*}
$\mathcal C^\text{tem}_\R$ is equipped with the topology generated by the norms $\{\norm{\cdot}_{(\lambda)}\}_{\lambda >0}$. It can be shown that, under this topology, the space $\mathcal C^\text{tem}_\R$ is Polish, i.e., complete separable metrizable space.

For any topological space $\mathcal X$ and functions $f,g$ on $\mathcal X$, we will say $f(x) \lesssim g(x)$ uniformly in $x \in \mathcal X$ if there exists $C>0$ so that $f(x) \le C g(x)$ for all $x\in \mathcal X$.

\begin{thm}\label{thm:canning_convergence}
	Let $b,\sigma$ in~\eqref{eq:bsigma def} be continuous. Suppose for each $u_0 \in \mathcal C(\R;[0,1])$, there exists a (probabilistically) unique weak solution to
	\begin{align}\label{eq:spde_canning}
		\partial_t u = \frac{\alpha}{2}\Delta u + B(u) + \sqrt{\sigma(u)u(1-u)}\dot{W},\qquad u(0,\cdot) = u_0.
	\end{align}
	Suppose for some $p>8$ and $\varepsilon >0$, uniformly in $N$ and $x\in[N^{-1}]$,
	\begin{align}\label{ineq:p_moment_canning_thm}
		\E_x \left[\abs{W_A^N(x) - x}^p \right] \lesssim h_N^{4+\varepsilon}.
	\end{align}
	Assume in addition that $\widetilde\xi_0^N\to u_0$ in $\mathcal C^\mathrm{tem}_\R$ as $N\to\infty$. Then the spatially interpolated process $\widetilde \xi^N$ as in~\eqref{eq:xin} converges in (probabilistic) distribution to the solution of~\eqref{eq:spde_canning} in  $\mathcal D(\R_+;\mathcal C^\mathrm{tem}_\R)$.
\end{thm}
\begin{remark}\label{rem:existence}
	Note that the proof of Theorem \ref{thm:canning_convergence}, along with examples in Section \ref{sec:examples} gives a constructive proof to the existence of solution for equation \eqref{eq:spde_canning} when $\min\sigma>0$, which is an alternative method to Shiga~\cite{shiga1994two}. In addition, we do not assume continuity of the prelimit coefficients, $s_N,\, c_N$ as in \eqref{eq:bsigma def}.
\end{remark}
\begin{remark}
	Note that in this work, we do require a priori that the equation~\eqref{eq:spde_canning} have weak uniqueness.
	Corollary~\ref{cor:AT conditions} below uses the mutation terms in~\eqref{eq: full CG drift} to construct discrete approximations for the full finite-invariant-interval class covered by the Athreya--Tribe uniqueness result.
\end{remark}
\begin{remark}
	Under~\eqref{ineq:p_moment_canning_thm}, reproduction events in which one parent produces a positive fraction of its colony are too rare to appear in the limit. Therefore the limit has continuous sample paths. In non-spatial models, such large reproduction events lead to $\Lambda$-coalescents~\cites{pitman1999coalescents,sagitov1999general} and to jumps in the frequency process~\cite{mohle2001forward}.
\end{remark}

\subsection{A general convergence result}
We now record an abstract version of Theorem~\ref{thm:canning_convergence}. Let $(L_n)_{n\in\N}$ be positive real numbers with $L_n\to\infty$; we write $L=L_n$ and suppress $n$, so that $L\to\infty$ means $n\to\infty$. For each $n$, let $\mathcal R^L$ be a random operator from $[0,1]$ to $[0,1]$, jointly measurable in $(x,\omega)$, and let $\{Y_g^L\}_{g\in\Z_+}$ denote the associated one-colony frequency chain, defined by
\[
	Y_{g+1}^L=\mathcal R_g^L(Y_g^L),\qquad g\ge 0,
\]
where $\{\mathcal R_g^L\}_{g\ge 0}$ are i.i.d. copies of $\mathcal R^L$. Assume that
\begin{itemize}
	\item There exist continuous functions $b:[0,1]\to\R$ and $\sigma:[0,1]\to[0,\infty)$ such that, upon defining
	\begin{equation}\label{eq:R assumption}
		\begin{aligned}
			b_L(x):=L^2\E_x\left[\mathcal R^L(x)-x\right],\qquad	 \sigma_L(x):=L\E_x\left[\pr{\mathcal R^L(x)-x}^2\right], \qquad x\in[0,1],
		\end{aligned}
	\end{equation}
	one has
	\begin{align}\label{eq:coefficient convergence}
		\lim_{L\to\infty}\sup_{x\in[0,1]} \pr{\abs{b_L(x)-b(x)}+\abs{\sigma_L(x)-\sigma(x)}}=0.
	\end{align}
	\item There exist $p>8$ and $\kappa_p >3$ such that, uniformly for sufficiently large $L \ge 1$ and $x \in [0,1]$,
	\begin{align}\label{ineq:p_moment bound}
		\E_x\left[ \abs{\mathcal R^{L}(x)-x}^p\right] \lesssim L^{-1-\kappa_p}.
	\end{align}
\end{itemize}

To build the corresponding spatial model, we use the same nearest-neighbor migration mechanism as in Subsection 2.2, with $1/L$ in the role of $h_N$. Fix $\alpha\in(0,1)$ and a deterministic $Y_0^L:\Z\to[0,1]$, let $\{\mathcal R_{i,g}^L\}_{(i,g)\in\Z\times\Z_+}$ be an i.i.d. family with common law $\mathcal R^L$, and define
\begin{align}\label{eq:abstract spatial frequency process}
	Y_{g+1}^L(i):=\mathcal R_{i,g}^L\pr{Y_g^L(i)+\frac{\alpha}{2}\Delta_d Y_g^L(i)},\qquad i\in\Z,\ g\in\Z_+.
\end{align}
Let $\mathcal F_g:=\sigma(\mathcal R^L_{i,k}:\, i\in\Z,\,0\le k<g)$ for $g\in\Z_+$, so that $Y^L_g$ is $\mathcal F_g$-measurable and $\{\mathcal R^L_{i,g}\}_{i\in\Z}$ is independent of $\mathcal F_g$. Writing $x_i=i/L\in \Z/L$, we then set
\begin{align*}
	\xi_t^L(x_i):=Y_{\lfloor tL^2\rfloor}^L(i),\qquad t\ge 0,
\end{align*}
and let $\widetilde \xi_t^L$ denote the linear spatial interpolation of $\xi_t^L$.

Let us state the corresponding abstract limit theorem. The target equation is
\begin{align}\label{eq: spde target}
	\partial_t u = \frac{\alpha}{2}\Delta u+ b(u) + \sqrt{\sigma(u)}\dot{W},
\end{align}
where $\dot{W}$ is the space-time white noise on $\R_+\times\R$.

\begin{thm}\label{thm:main}
	Let $b$ and $\sigma$ in~\eqref{eq:coefficient convergence} be continuous with
	\begin{align*}
		b(0) \ge 0,\quad b(1) \le 0,\qquad \sigma(0) = \sigma(1) = 0.
	\end{align*}
	Suppose for all $u_0\in \mathcal C(\R;[0,1])$, the equation~\eqref{eq: spde target} has a unique weak solution with initial condition $u(0,\cdot)=u_0$. If $\widetilde \xi_0^L \to u_0$ in $\mathcal C_\R^\text{tem}$, then $\widetilde{\xi}^L$ converges in distribution in $\mathcal D(\R_+;\mathcal C_\R^\text{tem})$ to the solution of~\eqref{eq: spde target}.
\end{thm}

\section{Example constructions and the Athreya--Tribe class}\label{sec:examples}

The convergence criterion of Theorem~\ref{thm:canning_convergence} reduces the construction of a convergent spatial model to the verification of a one-colony paintbox update. We give two such constructions.  The first varies the number of active parents while keeping the census size fixed, and the second uses Dirichlet weights. At the end of the section, we show that the Athreya--Tribe class of reaction--diffusion equations~\cite{athreya2000uniqueness} is reached by each of these constructions.

Throughout this section, we work with the paintbox Cannings--Gillespie model of Section~\ref{sec:canning}. Thus $\mc W^N(x)=\pr{w_1(x),\dots,w_N(x)}$ is an exchangeable vector of weights in the simplex
\begin{align}\label{def: simplex SN}
	\mc S_N:=\set{w\in[0,1]^N:\textstyle\sum_{i=1}^N w_i=1}.
\end{align}
For $x\in[N^{-1}]$, we write
\begin{align*}
	W_A^N(x)=\sum_{i\in I_x}w_i(x),\qquad c_N(x)=\E\bracket{\sum_{i=1}^N w_i(x)^2}.
\end{align*}

For a fixed-census paintbox construction, the limiting coefficients are encoded by the scaling relations in~\eqref{eq:bsigma def}. Once the weak well-posedness and initial-condition hypotheses of Theorem~\ref{thm:canning_convergence} are satisfied, it remains only to verify the moment condition
\begin{align}\label{eq:mom-cond}
	\E_x\bracket{\,\abs{W_A^N(x)-x}^{\,p}\,}\lesssim h_N^{\,4+\varepsilon}
	\qquad\text{for some }p>8,\ \varepsilon>0,
\end{align}
uniformly in $N$ and $x\in[N^{-1}]$; the spatial process then converges to the corresponding equation~\eqref{eq:spde_canning}. We first record the constraint imposed on fixed-population models by the canonical normalization $h_N=1/N$.

\begin{lemma}\label{lem:wf-floor}
	For every $w\in\mc S_N$ one has $\sum_{i=1}^N w_i^2\ge 1/N$, with equality precisely for the uniform vector $w_i\equiv 1/N$. Consequently $c_N(x)\ge 1/N$ for every colony model, so under the normalization $h_N=1/N$ the diffusion profile satisfies $\sigma(x)\ge 1$, the value $\sigma\equiv 1$ corresponding to Wright--Fisher resampling.
\end{lemma}
\begin{proof}
	By the Cauchy--Schwarz inequality, $1=\pr{\sum_{i}w_i}^2\le N\sum_{i}w_i^2$, with equality if and only if all $w_i$ are equal. Taking expectations in~\eqref{eq:cn scaling} gives $c_N(x)\ge 1/N=h_N$, and hence $\sigma(x)=\lim_{N \to \infty}c_N(x)/h_N\ge 1$.
\end{proof}

Lemma~\ref{lem:wf-floor} shows that, under the normalization $h_N=1/N$, changing the offspring weights cannot produce a profile below one.  The two constructions below remove this restriction by choosing a different scalar normalization. The first varies the number of active parents, while the second tunes the concentration of Dirichlet weights.  Both remain in the fixed-census paintbox framework and realize every continuous diffusion profile bounded away from zero.

\subsection{Varying population size}\label{sec:ex-pop}

We remain in the fixed-census paintbox CG framework of Section~\ref{sec:canning}.  Thus every generation still contains $N$ labeled individuals.  Here, population size refers to the number of active parents, namely the individuals with positive parental weight, and this number depends on the current frequency.

Fix continuous functions $b:[0,1]\to\R$ and $\sigma:[0,1]\to(0,\infty)$, and constants $\theta^+,\theta^-\ge0$.  Set
\begin{align}\label{eq:pop-scale}
	\sigma_*
	&:={\min_{x\in[0,1]}\sigma(x)}>0,
	&
	h_N
	&:={\frac1{N\sigma_*}}.
\end{align}
Since $h_N\to0$, there exists $N_0\ge2$ such that, for every $N\ge N_0$,
\begin{align}\label{ineq:pop-admissible-range}
	h_N\le1,
	\qquad
	h_N\norm{\sigma}_\infty\le\frac12,
	\qquad
	h_N^2\norm{b}_\infty<\frac12,
	\qquad
	h_N^2\pr{\theta^++\theta^-}\le1.
\end{align}
We use the construction below for $N\ge N_0$; for $2\le N<N_0$, choose arbitrary admissible paintbox CG models, which do not affect the limiting statement.

\begin{definition}[Frequency-dependent fertile-parent CG model] \label{def:pop-update}
	For $N\ge N_0$ and $x\in[0,1]$, define the  effective population size
	\begin{align}\label{eq:pop-target-size}
		r_N(x)
		&:={\frac1{h_N\sigma(x)}}
		={\frac{N\sigma_*}{\sigma(x)}}.
	\end{align}
	For each $x \in [0,1]$, let $M_N(x)$ be a random variable taking values in $\{\lfloor r_N(x)\rfloor,\,\lfloor r_N(x)\rfloor+1\}$ with probabilities
	\begin{equation}\label{eq:pop-size-rounding}
		\begin{gathered}
			\P\pr{M_N(x)=\lfloor r_N(x)\rfloor}
			={\frac{\lfloor r_N(x)\rfloor
			\pr{\lfloor r_N(x)\rfloor+1-r_N(x)}}{r_N(x)}},\\
			\P\pr{M_N(x)=\lfloor r_N(x)\rfloor+1}
			={\frac{\pr{\lfloor r_N(x)\rfloor+1}
			\pr{r_N(x)-\lfloor r_N(x)\rfloor}}{r_N(x)}}.
		\end{gathered}
	\end{equation}
	Conditional on $M_N(x)$, choose independently a uniform subset $J_N(x)\subset[N]$ of cardinality $M_N(x)$ and set
	\begin{align}\label{eq:pop-active-weights}
		w_i^N(x)
		&:={\frac{\1_{\{i\in J_N(x)\}}}{M_N(x)}},
		\qquad i\in[N].
	\end{align}
	Define the selection functions and mutation rates  by
	\begin{align}\label{eq:pop-selection-mutation}
		s_N(x):= h_N^2 b(x),\text{ for } x \in [0,1],
		\qquad
		\theta_N^\pm:= h_N^2\theta^\pm.
	\end{align}
	By~\eqref{ineq:pop-admissible-range}, the selection and mutation parameters satisfy the admissibility conditions in Section~\ref{sec:canning}. Together with the weights above, they define a paintbox CG model and its update operator $\widetilde{\mathcal R}^N$ through Definition~\ref{def: update operator}.
\end{definition}
\begin{remark}\label{rem:pop-floor}
    In principle, we can set $M_N(x)= \lfloor r_N(x)\rfloor$ to be deterministic and it will not affect our results. Indeed, in this case, $c_N(x)=1/\lfloor r_N(x)\rfloor$ satisfies $0\le c_N(x)/h_N-\sigma(x)\le\|\sigma\|_\infty/(r_N(x)-1)$ instead of~\eqref{eq:pop-exact-coalescence}, so~\eqref{eq:bsigma def} still holds, and $M_N\ge r_N/2$, so Proposition~\ref{prop:pop-moment-limits} and Corollary~\ref{cor:pop-limit} hold with the same proofs. However the above choices simplify many calculations. 
\end{remark}
\begin{figure}[H]
	\centering
	\includegraphics[width=\textwidth]{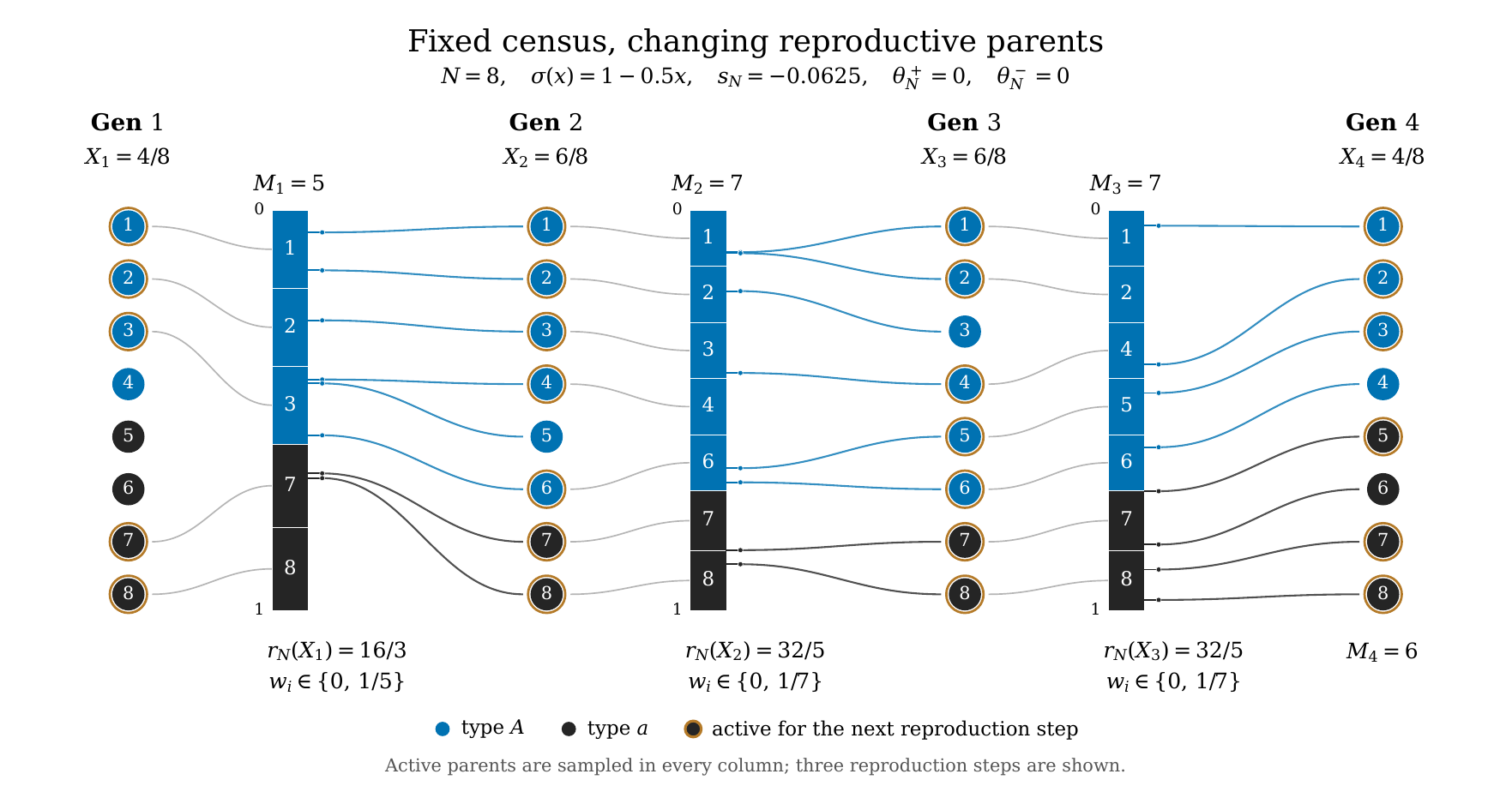}
	\caption{Three reproduction steps in a single colony of the CG model in Definition~\ref{def:pop-update} corresponding to~\eqref{eq:gillespie-fkpp} with $\kappa=1/2$ and $\varepsilon=1$; rings mark active parents.}
	\label{fig:active-parent-generations}
\end{figure}

\begin{prop}[Population-size and coalescence calibration]\label{prop:pop-calibration}
	For every $N\ge N_0$ and $x\in[0,1]$, the random weight vector in Definition~\ref{def:pop-update} is exchangeable and takes values in $\mc S_N$. Moreover, almost surely,
	\begin{align}\label{ineq:pop-size-calibration}
		\left|
		\frac{M_N(x)}N-\frac{\sigma_*}{\sigma(x)}
		\right|
		&\le\frac1N,
		&
		\frac{N\sigma_*}{2\norm{\sigma}_\infty}
		&\le M_N(x)\le N,
	\end{align}
	and the coalescence functional satisfies the exact identity
	\begin{align}\label{eq:pop-exact-coalescence}
		c_N(x)
		&:={\E\bracket{\sum_{i=1}^N{w_i^N(x)}^2}}
		={\E\bracket{\frac1{M_N(x)}}}
		={h_N\sigma(x)}.
	\end{align}
\end{prop}

\begin{proof}
	By~\eqref{eq:pop-scale},~\eqref{ineq:pop-admissible-range}, and~\eqref{eq:pop-target-size},
	\begin{align*}
		2
		&\le\frac{N\sigma_*}{\norm{\sigma}_\infty}
		\le r_N(x)
		\le N.
	\end{align*}
	We see by definition that $\abs{M_N(x)-r_N(x)}\le1$ and $M_N(x)\le N$ almost surely. Hence
	\begin{align*}
		\left|
		\frac{M_N(x)}N-\frac{\sigma_*}{\sigma(x)}
		\right|
		&{\overset{\eqref{eq:pop-target-size}}{=}}
		\frac1N\abs{M_N(x)-r_N(x)}
		\le\frac1N.
	\end{align*}
	Moreover,
	\begin{align*}
		M_N(x)
		&\ge\lfloor r_N(x)\rfloor
		\ge r_N(x)-1
		\ge\frac{r_N(x)}2
		\ge\frac{N\sigma_*}{2\norm{\sigma}_\infty},
	\end{align*}
	which proves~\eqref{ineq:pop-size-calibration}.

	Conditional on $M_N(x)$, the set $J_N(x)$ is uniform among subsets of $[N]$ with the prescribed cardinality. Hence the weights are exchangeable, nonnegative, and sum to one. Moreover,
	\begin{align*}
		\E\left[\sum_{i = 1}^N {w_i^N(x)}^2 \right]
		{\overset{\eqref{eq:pop-active-weights}}{=}}\E\bracket{\frac1{M_N(x)}}
		{\overset{\eqref{eq:pop-size-rounding}}{=}}\frac{1}{r_N(x)}
		{\overset{\eqref{eq:pop-target-size}}{=}}h_N\sigma(x).
	\end{align*}
	This proves~\eqref{eq:pop-exact-coalescence}.
\end{proof}

\begin{prop}[Type-mass concentration]\label{prop:pop-moment-limits}
	For every fixed $p\ge2$, there exists $C>0$ so that for all $N\ge N_0$,
	\begin{align}\label{ineq:pop-type-moment}
		\sup_{x\in[N^{-1}]}
		\E_x\bracket{\abs{W_A^N(x)-x}^p}
		&\le Ch_N^{p/2}.
	\end{align}
\end{prop}

\begin{proof}
	Fix $x\in[N^{-1}]$.  Conditional on $M_N(x)=m$, one has
	\begin{align*}
		mW_A^N(x)
		={\abs{J_N(x)\cap I_x}}
		\sim\operatorname{Hypergeometric}\pr{N,Nx,m}.
	\end{align*}
	In particular, $\E_x[W_A^N(x)\mid M_N(x)]=x$.  The Hoeffding--Serfling inequality for sampling without replacement~\cite{serfling1974probability} gives, for every $t>0$,
	\begin{align*}
		\P_x\pr{\abs{W_A^N(x)-x}\ge t}
		&{=}\E_x\bracket{
		\P_x\pr{\abs{W_A^N(x)-x}\ge t\mid M_N(x)}}\\
		&\le2\E_x\bracket{\exp\pr{-2M_N(x)t^2}}\\
		&{\overset{\eqref{ineq:pop-size-calibration}}{\le}}
		2\exp\pr{-\frac{t^2}{h_N\norm{\sigma}_\infty}}.
	\end{align*}
	Consequently,
	\begin{align*}
		\sup_{x\in[N^{-1}]}
		\E_x\bracket{\abs{W_A^N(x)-x}^p}
		&{=}\sup_{x\in[N^{-1}]}
		p\int_0^\infty t^{p-1}
		\P_x\pr{\abs{W_A^N(x)-x}\ge t}\,\mathrm dt\\
		&\le2p\int_0^\infty
		t^{p-1}\exp\pr{-\frac{t^2}{h_N\norm{\sigma}_\infty}}\,\mathrm dt\\
		&{=}Ch_N^{p/2},
	\end{align*}
	for some $C>0$ independent of $N$, which is~\eqref{ineq:pop-type-moment}.
\end{proof}

\begin{cor}[Varying-population CG approximation with mutation] \label{cor:pop-limit}
	Fix $\alpha\in(0,1)$. Suppose that, for every $u_0\in\mathcal C(\R;[0,1])$, the equation
	\begin{align}\label{eq:pop-limit-spde}
		\partial_tu
		&{=}\frac\alpha2\Delta u+b(u)u(1-u)
		+\theta^+(1-u)-\theta^-u
		+\sqrt{\sigma(u)u(1-u)}\,\dot W,
		\qquad u(0,\cdot)=u_0,
	\end{align}
	has a $[0,1]$-valued weak solution that is unique in law.  Construct the spatial CG model from Definition~\ref{def:pop-update}, with scalar normalization $h_N$ as in~\eqref{eq:pop-scale}. If $\widetilde\xi_0^N\to u_0$ in $\mathcal C_\R^{\mathrm{tem}}$, then $\widetilde\xi^N$ converges in distribution to the solution of~\eqref{eq:pop-limit-spde} in $\mathcal D(\R_+;\mathcal C_\R^{\mathrm{tem}})$.

	In particular, for $b\equiv-1$ and $\sigma(x)=1-\kappa x$, where $0\le\kappa<1$, one has
	\begin{align*}
		\sigma_*
		&{=}1-\kappa,
		&
		h_N
		&{=}\frac1{N(1-\kappa)},
		&
		r_N(x)
		&{=}\frac{N(1-\kappa)}{1-\kappa x}.
	\end{align*}
	The resulting spatial CG models converge to the two-way-mutation Gillespie--FKPP equation
	\begin{align}\label{eq:gillespie-fkpp-mutation}
		\partial_tu
		{=}\frac\alpha2\Delta u-u(1-u)
		+\theta^+(1-u)-\theta^-u
		+\sqrt{u(1-u)(1-\kappa u)}\,\dot W.
	\end{align}
\end{cor}

\begin{proof}
	Note that Proposition~\ref{prop:pop-calibration} verifies~\eqref{eq:bsigma def} exactly. Fix $p>8$ and choose $0<\varepsilon<p/2-4$.  Since $h_N\le1$, we see by Proposition~\ref{prop:pop-moment-limits} that there exists $C>0$ so that for all $N\ge N_0$,
	\begin{align*}
		\sup_{x\in[N^{-1}]}
		\E_x\bracket{\abs{W_A^N(x)-x}^p}
		&{\overset{\eqref{ineq:pop-type-moment}}{\le}}C h_N^{p/2}
		\le C h_N^{4+\varepsilon}.
	\end{align*}
	Thus~\eqref{ineq:p_moment_canning_thm} holds, and Theorem~\ref{thm:canning_convergence} gives the asserted convergence. The specialization follows by substituting $b\equiv-1$ and $\sigma(x)=1-\kappa x$ into~\eqref{eq:pop-limit-spde}.
\end{proof}

\begin{remark}[Example~\ref{ex:gillespie-resource} as a varying-population CG model]\label{rem:resource-example}
	For $N\ge N_0$, Example~\ref{ex:gillespie-resource} with selection under the diffusive scaling $\beta_R:=h_N^2/(1-h_N^2)$ is the CG model of Definition~\ref{def:pop-update} with
	\begin{align*}
		\sigma(x)=\varepsilon^2(1-\kappa x),\qquad b\equiv-1,\qquad \theta^\pm=0,\qquad M_N(x)=\lfloor r_N(x)\rfloor,
	\end{align*}
	the last choice being allowed by Remark~\ref{rem:pop-floor}. Indeed, $\sigma_*=\varepsilon^2(1-\kappa)$, so $h_N=1/(\varepsilon^2R)$ and $r_N(x)=R/(1-\kappa x)$, hence $M_N(x)=N_R(x)$: the active parents are the individuals that reach reproductive age. Since $c_N(x)=1/N_R(x)$, for $\kappa>0$ the noise becomes smaller as type $A$ individuals become more common. Here, Corollary~\ref{cor:pop-limit} with Remark~\ref{rem:pop-floor} shows that the spatial version of Example~\ref{ex:gillespie-resource} converges to~\eqref{eq:gillespie-fkpp}. For every $u_0\in\mathcal C(\R;[0,1])$, this equation has a $[0,1]$-valued weak solution by Remark~\ref{rem:existence}, and it is unique in law by moment duality, see~\cite{shiga1988stepping} and~\cite{barnes2025wright}*{Theorem~1.1} for $\kappa=0$, and~\cite{hou2026subcritical}*{Proposition~3.3 and Lemma~3.4} for $0<\kappa<1$.
\end{remark}

\subsection{Prescribed coefficients via Dirichlet weights}\label{sec:ex-dir}

For this subsection, let $h_N\to0$ be a positive scalar normalization. The notation $c_N(x)$ continues to denote the coalescence functional in~\eqref{eq:cn scaling}.  Fix a continuous function $\sigma:[0,1]\to(0,\infty)$ and a continuous function $b:[0,1]\to\R$. For all sufficiently large $N$, prescribe
\begin{align}\label{eq:dir-target-coalescence}
	c_N(x)=h_N\sigma(x).
\end{align}
By Lemma~\ref{lem:wf-floor}, the prescription \eqref{eq:dir-target-coalescence} is admissible only if $1/N\le h_N\sigma(x)$.  We therefore assume
\begin{align}\label{eq:dir-scale-compatibility}
	\frac1{Nh_N}
	\le
	\min_{x\in[0,1]}\sigma(x)
\end{align}
for all sufficiently large $N$. On the other hand, since \eqref{eq:dir-scale-compatibility} holds for all large $N$, $h_N\to0$, and $b$ and $\sigma$ are bounded continuous functions on $[0,1]$, we see there is $N_0\ge2$ such that for all $N \ge N_0$ and $x \in [0,1]$, it holds that
\begin{align}\label{eq:dir-admissible-range}
	\frac1N
	\le
	h_N\sigma(x)
	<1,
	\qquad
	h_N\le1,\quad
	h_N^2\norm{b}_\infty<\frac12.
\end{align}

Define
\begin{align}\label{eq:dir-selection}
	s_N(x)
	&:={}
	\begin{cases}
		h_N^2b(x),&N\ge N_0,\\
		0,&N<N_0,
	\end{cases}
	\qquad x\in[0,1].
\end{align}
For $N\ge N_0$ and $x\in[0,1]$, define
\begin{align}\label{eq:dirichlet-concentration}
	a_N(x)
	:=
	\begin{cases}
		\dfrac{1-h_N\sigma(x)}{Nh_N\sigma(x)-1},
		& h_N\sigma(x)\in(1/N,1),\\[2ex]
		\infty,
		& h_N\sigma(x)=1/N.
	\end{cases}
\end{align}
If $a_N(x)<\infty$, let
\[
	\mc W^N(x)
	=
	\pr{w_1^N(x),\ldots,w_N^N(x)}
	\sim
	\operatorname{Dirichlet}
	\pr{a_N(x),\ldots,a_N(x)}.
\]
If $a_N(x)=\infty$, set
\[
	\mc W^N(x)
	=
	\pr{\frac1N,\ldots,\frac1N}
\]
deterministically. For the finitely many $N<N_0$, take the uniform weight vector. We will use the convention that $a/\infty =0$ for any $a \in \R$.

A direct computation from~\eqref{eq:dirichlet-concentration}, together with~\eqref{eq:dir-target-coalescence} and the above convention, gives, for every $N\ge N_0$ and $x\in[0,1]$,
\begin{align}\label{eq:dir-total-concentration}
	\frac1{Na_N(x)+1}
	=
	\frac{Nc_N(x)-1}{N-1}
	=
	\frac{Nh_N\sigma(x)-1}{N-1}.
\end{align}

The Dirichlet weights have a natural biological interpretation~\cites{gonzalez2016individual,baake2019modelling,boenkost2021haldaneweak}.  In each generation, every individual produces an offspring family of large i.i.d. random size, and the next generation consists of $N$ individuals sampled from the union of these families.  In the limit of large families, the sampled individuals choose their parents independently, with probabilities proportional to the family sizes.  If the rescaled family sizes are $\operatorname{Gamma}(a_N(x),1)$ distributed, these probabilities form exactly the $\operatorname{Dirichlet}(a_N(x),\ldots,a_N(x))$ vector above~\cite{huillet2022asymptotic}; the geometric family sizes used in~\cites{gonzalez2016individual,baake2019modelling} correspond to $a_N(x)=1$.  Selection enters as in~\cite{boenkost2021haldaneweak}, by reweighting the parents according to their type before sampling. 

The concentration $a_N(x)$  measures how unequal the family sizes are: as $a_N(x)\to\infty$, all families have the same size and one recovers the Wright--Fisher model, whereas for small $a_N(x)$ a few families dominate.  In the neutral case, the number of a parent's offspring in the next generation has variance $(N-1)c_N(x)=(N-1)h_N\sigma(x)$, so~\eqref{eq:dir-target-coalescence} prescribes a frequency-dependent variance in offspring number, which is reminiscent of Gillespie's works~\cites{gillespie1974natural,gillespie1975natural}.

\begin{figure}[H]
	\centering
	\includegraphics[width=0.8\textwidth]{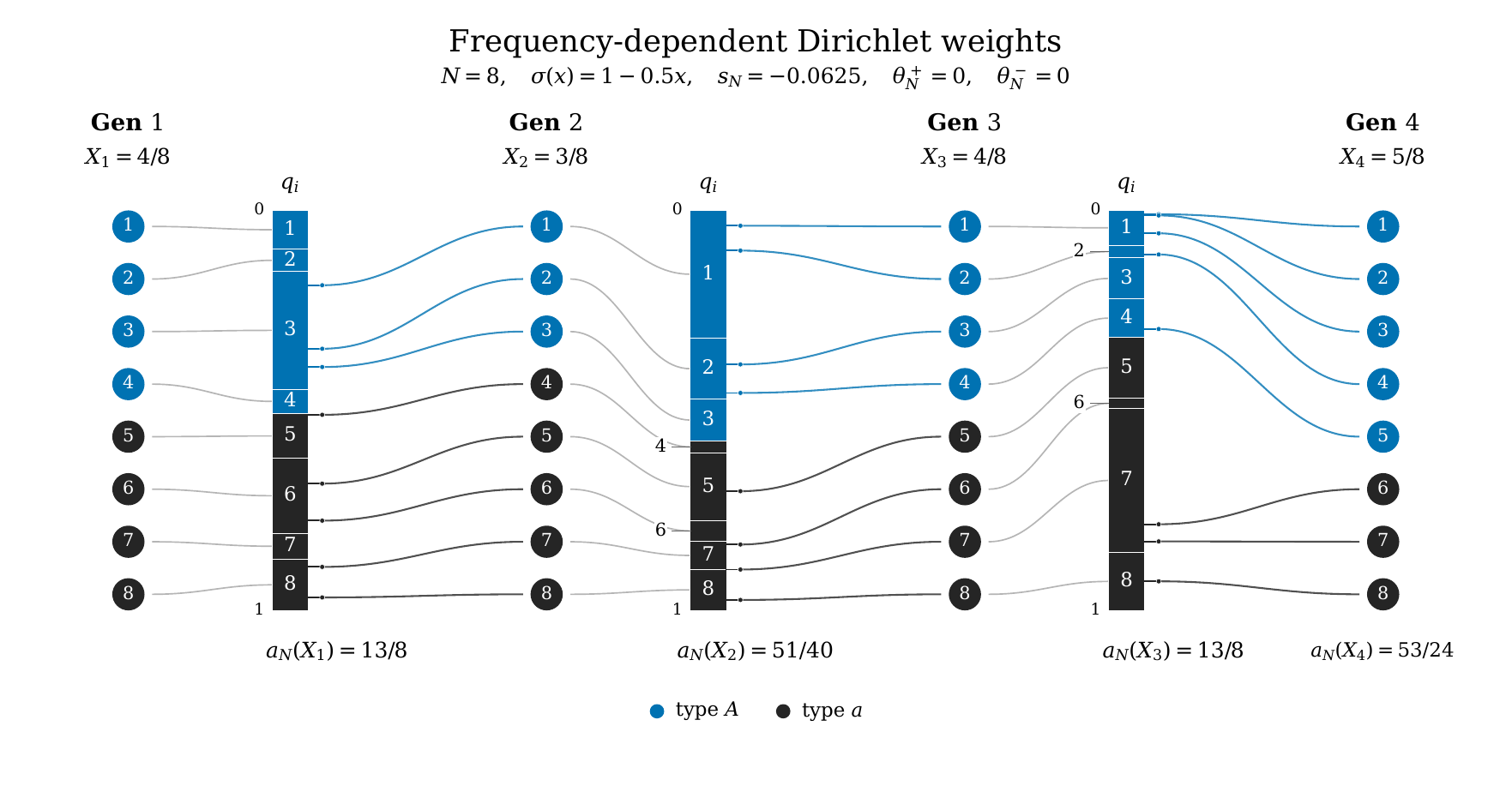}
	\caption{The analogue of Figure~\ref{fig:active-parent-generations} for the Dirichlet-weight CG model, with the same parameters.
	Parental weights are redrawn at each step from~\eqref{eq:dirichlet-concentration}.}
	\label{fig:dirichlet-cg-eq11}
\end{figure}

\begin{prop}\label{prop:dir-calibration}
	For every $N\ge N_0$ and $x\in[0,1]$,
	\begin{align}\label{eq:dir-exact-calibration}
		c_N(x)=\E\bracket{\sum_{i=1}^N{w_i^N(x)}^2}=h_N\sigma(x).
	\end{align}
\end{prop}

\begin{proof}
	By~\eqref{eq:dirichlet-concentration}, it holds that
	\begin{align}\label{eq:prop:dir-calibration step 1}
		a_N(x)+1= \frac{(N-1)h_N\sigma(x)}{Nh_N\sigma(x) - 1},\qquad Na_N(x)+1
		=
		\frac{N-1}{Nh_N\sigma(x)-1}.
	\end{align}
	For a symmetric $\operatorname{Dirichlet}(a,\ldots,a)$ vector, the moment formula for Dirichlet distribution (cf.~\cite{ng2011dirichlet}*{Eq.~(2.5), p.~39}) gives
	\[
		\E\bracket{{w_i^N(x)}^2}
		=
		\frac{a_N(x)(a_N(x)+1)}{Na_N(x)(Na_N(x)+1)}
		=
		\frac{a_N(x)+1}{N(Na_N(x)+1)},\qquad i = 1,\dots, N.
	\]
	Therefore, by symmetry,
	\[
		\E\bracket{\sum_{i=1}^N{w_i^N(x)}^2}
		= N\E\bracket{w_1^N(x)^2}=
		\frac{a_N(x)+1}{Na_N(x)+1}
		\overset{\eqref{eq:prop:dir-calibration step 1}}{=}
		h_N\sigma(x),
	\]
	which implies~\eqref{eq:dir-exact-calibration}.

	If $h_N\sigma(x) = 1/N$, the weights are deterministically uniform, so~\eqref{eq:dir-exact-calibration} holds trivially.
\end{proof}

\begin{prop}\label{prop:dir-moment-limits}
	Let $\widetilde{\mathcal R}^N$ be the update operator of Definition~\ref{def: update operator} associated with the weights above and the selection functions in~\eqref{eq:dir-selection}. Suppose also that, for some constants $\theta^+,\theta^-\ge0$,
	\begin{align}\label{eq:dir-mutation-scaling}
		\lim_{N\to\infty}h_N^{-2}\theta_N^\pm
		&{=}\theta^\pm.
	\end{align}
	Then one has
	\begin{equation}\label{eq:dir-coefficient-limits}
		\begin{aligned}
			\lim_{N\to\infty}\sup_{x\in[N^{-1}]}
			\Bigg(
			&\abs{
			h_N^{-2}\E_x\bracket{\widetilde{\mathcal R}^N(x)-x}- b(x)x(1-x)-\theta^+(1-x)+\theta^-x
			}\\
			&\quad
			+\abs{
			h_N^{-1}\E_x
			\bracket{\pr{
			\widetilde{\mathcal R}^N(x)-x
			}^2
			}
			-\sigma(x)x(1-x)
			}
			\Bigg)
			=
			0.
		\end{aligned}
	\end{equation}
	Moreover, for every fixed $p\ge2$, there are constants $C>0$ and $N_p\ge N_0$ such that, for every $N\ge N_p$, the following bounds hold:
	\begin{align}
		\sup_{x\in[N^{-1}]}
		\E_x\bracket{
		\abs{W_A^N(x)-x}^p
		}
		&\le
		Ch_N^{p/2},
		\label{eq:dir-type-moment}\\
		\sup_{x\in[N^{-1}]}
		\E_x\bracket{
		\abs{\widetilde{\mathcal R}^N(x)-x}^p
		}
		&\le
		Ch_N^{p/2}.
		\label{eq:dir-update-moment}
	\end{align}
\end{prop}

\begin{proof}
	It suffices to consider $N\ge N_0$.  Fix $p\ge2$.  For $x\in[N^{-1}]$, if $x\in\{0,1\}$ or $a_N(x)=\infty$, then $W_A^N(x)=x$ deterministically.  At every remaining grid point, the gamma representation of a Dirichlet vector \cite{ng2011dirichlet}*{Thm.~2.1, Eq.~(2.7), p.~40} and additivity of independent gamma variables give
	\[
		W_A^N(x)
		\sim
		\operatorname{Beta}\pr{
		Nxa_N(x),N(1-x)a_N(x)
		}.
	\]
	The standard beta mean and variance formulas (cf.~\cite{ng2011dirichlet}*{Eq.~(2.6), p.~39}), together with the deterministic cases, give
	\begin{align}\label{eq:dir-W-first-two-moments}
		\E_x\bracket{W_A^N(x)}
		=
		x,\qquad
		\var_x\pr{W_A^N(x)}
		=
		\frac{x(1-x)}{Na_N(x)+1}
		\overset{\eqref{eq:dir-total-concentration}}{=}
		x(1-x)\frac{Nh_N\sigma(x)-1}{N-1}.
	\end{align}

	The explicit upper bound in \cite{marchal2017sub}*{Def.~1.1, p.~2, and Thm.~2.1, p.~5} gives, at every nondegenerate grid point and for every $\lambda\in\R$,
	\[
		\E_x\bracket{
		\exp\pr{\lambda\pr{W_A^N(x)-x}}
		}
		\le
		\exp\pr{\frac{\lambda^2}{8\pr{Na_N(x)+1}}}
		\le
		\exp\pr{\frac{\lambda^2h_N\norm{\sigma}_\infty}{4}},
	\]
	where the second inequality follows from~\eqref{eq:dir-total-concentration} and $N/(N-1)\le2$.  The same estimate is trivial in the deterministic cases.  Chernoff's bound therefore yields, for every $t>0$,
	\[
		\sup_{x\in[N^{-1}]}
		\P_x\pr{\abs{W_A^N(x)-x}\ge t}
		\le
		2\exp\pr{-\frac{t^2}{h_N\norm{\sigma}_\infty}}.
	\]
	 Consequently, by the moment bounds of Gaussian random variables, we have, for some $C>1$, 
	\begin{align*}
		\sup_{x\in[N^{-1}]}
		\E_x\bracket{\abs{W_A^N(x)-x}^p}
		&=
		\sup_{x\in[N^{-1}]}
		p\int_0^\infty
		t^{p-1}\P_x\pr{\abs{W_A^N(x)-x}\ge t}\,\mathrm dt\nonumber\\
        &\le C\E\left[ \abs{\mathcal N(0,h_N\norm{\sigma}_\infty)}^p\right]\\
		&\le
		C
		\norm{\sigma}_\infty^{p/2}h_N^{p/2},
	\end{align*}
	where $\mathcal N(a,b)$ denotes the normal random variable with mean $a$ and variance $b$. This proves~\eqref{eq:dir-type-moment}.

	For the remaining assertions, note that $\widetilde{\mathcal R}^N(x)$ has the same law as the one-step update $X_1$ under $\P_x$. For some $C>0$ and all sufficiently large $N$,
	\begin{align}\label{ineq:dir-mutation-bound}
		\theta_N^++\theta_N^-  {=}h_N^2\left(h_N^{-2}\theta_N^++h_N^{-2}\theta_N^-\right) {\overset{\eqref{eq:dir-mutation-scaling}}{\le}}Ch_N^2.
	\end{align}
	For the drift term, we see by the triangular inequality  that
	\begin{align*}
		&{\qquad}\limsup_{N\to\infty}\sup_{x\in[N^{-1}]}
		\left|
		h_N^{-2}\E_x\bracket{\widetilde{\mathcal R}^N(x)-x}
		-b(x)x(1-x)-\theta^+(1-x)+\theta^-x
		\right|\\
		&{\overset{\eqref{eq:dir-selection}}{\le }}
		\limsup_{N\to\infty}\Bigg[
		h_N^{-2}\sup_{x\in[N^{-1}]}
		\left|
		\E_x\bracket{\widetilde{\mathcal R}^N(x)-x}
		-s_N(x)x(1-x)-\theta_N^+(1-x)+\theta_N^-x
		\right|\Bigg]\\
		&\hspace{41mm}
		+\limsup_{N \to \infty}\left(\left|h_N^{-2}\theta_N^+-\theta^+\right|
		+\left|h_N^{-2}\theta_N^--\theta^-\right|\right)
		\\
		&\overset{\eqref{eq:dir-mutation-scaling}}{=}\limsup_{N\to\infty}\Bigg[
		h_N^{-2}\sup_{x\in[N^{-1}]}
		\left|
		\E_x\bracket{\widetilde{\mathcal R}^N(x)-x}
		-s_N(x)x(1-x)-\theta_N^+(1-x)+\theta_N^-x
		\right|\Bigg]\\
		&{\overset{\mathrm{Lem.}\,\ref{lem:canning_drift}}{\le}}
		\limsup_{N\to\infty}
		\Bigg[
		\norm{b}_\infty\pr{h_N\norm{\sigma}_\infty+N^{-1}}
		+2h_N^2\norm{b}_\infty^2
		+\frac14\norm{b}_\infty\pr{\theta_N^++\theta_N^-}
		\Bigg]\\
		&{\overset{\eqref{ineq:dir-mutation-bound}}{\le}}
		\limsup_{N\to\infty}\Bigg[
		\norm{b}_\infty\pr{h_N\norm{\sigma}_\infty+N^{-1}}
		+2h_N^2\norm{b}_\infty^2
		+\frac C4h_N^2\norm{b}_\infty
		\Bigg]\\
		&=0.
	\end{align*}

	\begin{samepage}
		For the variance term,
		\begin{align*}
			&\qquad\limsup_{N\to\infty}\sup_{x\in[N^{-1}]}
			\left|
			h_N^{-1}\E_x\bracket{\pr{\widetilde{\mathcal R}^N(x)-x}^2}-\sigma(x)x(1-x)
			\right|\\
			&{\overset{\eqref{eq:dir-exact-calibration}}{=}}
			\limsup_{N\to\infty}h_N^{-1}\sup_{x\in[N^{-1}]}
			\left|
			\E_x\bracket{\pr{\widetilde{\mathcal R}^N(x)-x}^2}-c_N(x)x(1-x)
			\right|\\
			&{\overset{\mathrm{Lem.}\,\ref{lem:canning_var}}{\le}}
			\limsup_{N\to\infty}
			\Bigg[
			\frac{h_N\norm{b}_\infty}{2N}+h_N\norm{b}_\infty
			\sqrt{\frac{h_N\norm{\sigma}_\infty}{2}}
			+\frac14h_N^3\norm{b}_\infty^2+3h_N^{-1}\pr{\theta_N^++\theta_N^-}
			\Bigg]\\
			&{\overset{\eqref{ineq:dir-mutation-bound}}{\le}}\limsup_{N\to\infty}
			\Bigg[
			\frac{h_N\norm{b}_\infty}{2N}
			+h_N\norm{b}_\infty
			\sqrt{\frac{h_N\norm{\sigma}_\infty}{2}}
			+\frac14h_N^3\norm{b}_\infty^2
			+3Ch_N
			\Bigg]\\
			&{=}0.
		\end{align*}
		Hence,~\eqref{eq:dir-coefficient-limits} holds.
	\end{samepage}

	\begin{samepage}
		For every fixed $p\ge2$ and all sufficiently large $N$,
		\begin{align*}
			\sup_{x\in[N^{-1}]}\E_x\bracket{\abs{\widetilde{\mathcal R}^N(x)-x}^p}
			&{\overset{\mathrm{Lem.}\,\ref{lem:canning_higher_moments}}{\le}}
			C\Bigg(
			N^{-p/2}+h_N^{2p}\norm{b}_\infty^p+\pr{\theta_N^++\theta_N^-}^p
			+\sup_{x\in[N^{-1}]}
			\E_x\bracket{\abs{W_A^N(x)-x}^p}
			\Bigg)\\
			&{\overset{\eqref{eq:dir-type-moment}}{\le}}
			C\pr{N^{-p/2}+h_N^{2p}
			+\pr{\theta_N^++\theta_N^-}^p+h_N^{p/2}}\\
			&{\overset{\eqref{ineq:dir-mutation-bound}}{\le}}
			C\pr{N^{-p/2}+h_N^{2p}+h_N^{2p}+h_N^{p/2}}\\
			&{\overset{\eqref{eq:dir-scale-compatibility}}{\le}}
			Ch_N^{p/2},
		\end{align*}
		where we also used the fact that $h_N \le 1$ for sufficiently large $N$.  This proves~\eqref{eq:dir-update-moment}.
	\end{samepage}

\end{proof}

The following corollary is an immediate consequence of Propositions~\ref{prop:dir-calibration} and~\ref{prop:dir-moment-limits}, together with Theorem~\ref{thm:canning_convergence}.  The argument is the same as in the proof of Corollary~\ref{cor:pop-limit} and is omitted.

\begin{cor}[Dirichlet CG approximation with mutation]\label{cor:dirichlet-convergence}
	Fix $\alpha\in(0,1)$ and $\theta^+,\theta^-\ge0$.  Let $b\in\mathcal C([0,1];\R)$, $\sigma\in\mathcal C([0,1];(0,\infty))$, and let $h_N\to0$ be a positive sequence satisfying \eqref{eq:dir-scale-compatibility}.  Consider the Dirichlet-weight CG models specified by~\eqref{eq:dir-target-coalescence},~\eqref{eq:dir-selection}, and~\eqref{eq:dirichlet-concentration}, with admissible mutation probabilities satisfying~\eqref{eq:dir-mutation-scaling}. Suppose that, for every $u_0\in\mathcal C(\R;[0,1])$, the equation
	\begin{align}\label{eq:dir-limit-spde}
		\partial_tu
		&{=}\frac\alpha2\Delta u+b(u)u(1-u)
		+\theta^+(1-u)-\theta^-u
		+\sqrt{\sigma(u)u(1-u)}\,\dot W,
		\qquad u(0,\cdot)=u_0,
	\end{align}
	has a $[0,1]$-valued weak solution that is unique in law.   Let $\widetilde\xi^N$ be the corresponding spatially interpolated processes, with scalar normalization $h_N$.  If $\widetilde\xi_0^N\to u_0$ in $\mathcal C_\R^{\mathrm{tem}}$, then $\widetilde\xi^N$ converges in distribution to the solution of \eqref{eq:dir-limit-spde} in $\mathcal D(\R_+;\mathcal C_\R^{\mathrm{tem}})$.
\end{cor}

\subsection{The Athreya--Tribe class}
We describe the class with a finite invariant interval considered in \cite{athreya2000uniqueness}.  Fix $-\infty<a_-<a_+<\infty$ and let $b_{\mathrm{AT}}$ and $\sigma_{\mathrm{AT}}$ be real-valued functions defined on a neighborhood of $[a_-,a_+]$.  Assume that $\sigma_{\mathrm{AT}}\not\equiv0$ and that
\begin{equation}\label{eq: AT invariant interval conditions}
	\begin{gathered}
		\sigma_{\mathrm{AT}}\ge0\ \text{ on }[a_-,a_+],
		\qquad
		\sigma_{\mathrm{AT}}(a_-)=\sigma_{\mathrm{AT}}(a_+)=0,\\
		b_{\mathrm{AT}}(a_-)\ge0, \qquad b_{\mathrm{AT}}(a_+)\le0.
	\end{gathered}
\end{equation}
Let $R_0:=\max\{|a_-|,|a_+|\}$.  Suppose, for some $R_1>R_0$, that $b_{\mathrm{AT}}$ and $\sigma_{\mathrm{AT}}$ admit holomorphic extensions to the complex disk $\{z\in\C:|z|<R_1\}$, with convergent power-series expansions
\begin{align}\label{eq: AT power series}
	b_{\mathrm{AT}}(z)&=\sum_{k\ge0}b_k z^k,&\sigma_{\mathrm{AT}}(z)&=\sum_{k\ge0}\sigma_k z^k.
\end{align}
Assume further that these coefficients satisfy~\cite{athreya2000uniqueness}*{Condition~(H2)}: for some $R\in(R_0,R_1)$,
\begin{align}\label{ineq: AT H2 dominance}
	\sigma_2&<-\sum_{k\ne2}|\sigma_k|R^{k-2},
\end{align}
and, when $b_{\mathrm{AT}}\not\equiv0$, assume in addition that
\begin{align}\label{ineq: AT H2 drift dominance}
	b_1&<-\sum_{k\ne1}|b_k|R^{k-1}.
\end{align}
If $b_{\mathrm{AT}}\equiv0$, the latter condition is omitted.

For any $v_0\in\mathcal C(\R;[a_-,a_+])$, consider
\begin{align}\label{eq: AT original equation}
	\partial_t v =\frac12\Delta v+b_{\mathrm{AT}}(v) + \sqrt{\sigma_{\mathrm{AT}}(v)}\,\dot W, \qquad v(0,\cdot)=v_0.
\end{align}

\begin{corollary}[Athreya--Tribe approximation]\label{cor:AT conditions}
	Under the preceding conditions, equation~\eqref{eq: AT original equation} admits an $[a_-,a_+]$-valued weak solution $v$ that is unique in law.  Fix $\alpha\in(0,1)$ and set $\ell:=a_+-a_-$.  Each of Corollaries~\ref{cor:pop-limit} and~\ref{cor:dirichlet-convergence} yields a sequence of mutation--selection spatial CG models, the former through varying population size and the latter through Dirichlet weights.  For either sequence, write $\widetilde\xi^N$ for its spatial interpolation.  If $\widetilde\xi_0^N\to(v_0-a_-)/\ell$ in $\mathcal C_\R^{\mathrm{tem}}$, then
	\begin{align}\label{eq: AT approximation conclusion}
		\left(a_-+\ell\,\widetilde\xi^N_{t/\alpha}\right)_{t\ge0}\Longrightarrow (v_t)_{t\ge0}
	\end{align}
	in $\mathcal D(\R_+;\mathcal C_\R^{\mathrm{tem}})$.
\end{corollary}

\begin{proof}
	Existence of an $[a_-,a_+]$-valued weak solution follows from the construction described in~\cite{athreya2000uniqueness}*{Section~1.5, pp.~1715--1716}. It also follows from the proof of Theorem~\ref{thm:canning_convergence}, see Remark~\ref{rem:existence}. Uniqueness in law follows from~\cite{athreya2000uniqueness}*{Theorem~1, p.~1716}, which states that Condition~(H2) alone is sufficient.

	For $x\in[0,1]$, set
	\begin{align*}
		B_{\mathrm{AT}}(x)
		&:=\ell^{-1}b_{\mathrm{AT}}(a_-+\ell x),
		&
		\Sigma_{\mathrm{AT}}(x)
		&:=\ell^{-2}\sigma_{\mathrm{AT}}(a_-+\ell x).
	\end{align*}
	Define
	\begin{align*}
		w_t:=\frac{v_{\alpha t}-a_-}{\ell}.
	\end{align*}
	Applying the affine normalization and the time rescaling directly to~\eqref{eq: AT original equation}, and using that the latter multiplies the space--time white-noise term in law by $\sqrt\alpha$, shows that, after denoting the rescaled white noise again by $\dot W$, the process $w$ is a weak solution of
	\begin{align}\label{eq: AT alpha scaled equation}
		\partial_t w
		=\frac\alpha2\Delta w+\alpha B_{\mathrm{AT}}(w)
		+\sqrt{\alpha\Sigma_{\mathrm{AT}}(w)}\,\dot W,
		\qquad
		w_0=\frac{v_0-a_-}{\ell}.
	\end{align}
	Conversely, if $w$ is any weak solution of~\eqref{eq: AT alpha scaled equation}, then $v_t:=a_-+\ell w_{t/\alpha}$ is a weak solution of~\eqref{eq: AT original equation}. Since every initial condition in $\mathcal C(\R;[0,1])$ is the affine normalization of one in $\mathcal C(\R;[a_-,a_+])$, the same correspondence establishes weak existence and uniqueness in law for~\eqref{eq: AT alpha scaled equation} for every initial condition. It remains only to realize its coefficients through the two CG constructions.

	Set
	\begin{align*}
		r_+&:=B_{\mathrm{AT}}(0)\ge0,
		&
		r_-&:=-B_{\mathrm{AT}}(1)\ge0.
	\end{align*}
	For $x\in(0,1)$, define
	\begin{align*}
		\beta(x):=
		\frac{B_{\mathrm{AT}}(x)-r_+(1-x)+r_-x}{x(1-x)}.
	\end{align*}
	The numerator admits a holomorphic extension to a complex neighborhood of $[0,1]$ and vanishes at $0$ and $1$. By~\cite{conway1978functions}*{Corollary~IV.3.9}, $\beta$ therefore extends analytically across both endpoints. Consequently,
	\begin{align}\label{eq: AT drift decomposition}
		B_{\mathrm{AT}}(x)
		=\beta(x)x(1-x)+r_+(1-x)-r_-x,
		\qquad x\in[0,1].
	\end{align}

	For $|z|=R$,~\eqref{ineq: AT H2 dominance} gives
	\begin{align*}
		\left|\sigma_{\mathrm{AT}}(z)-\sigma_2z^2\right|
		&=\left|\sum_{k\ne2}\sigma_kz^k\right|
		\le\sum_{k\ne2}|\sigma_k|R^k
		=R^2\sum_{k\ne2}|\sigma_k|R^{k-2}
		< -\sigma_2R^2
		=|\sigma_2z^2|.
	\end{align*}
	Hence, by Rouch\'e's theorem~\cite{conway1978functions}*{Theorem~V.3.8}, $\sigma_{\mathrm{AT}}$ and $\sigma_2z^2$ have the same number of zeros in $\{|z|<R\}$, namely two, counted with multiplicity. This implies that the two zeros $a_-\neq a_+$ are simple and that $\sigma_\mathrm{AT}$ has no zeros in $(a_-,a_+)$. By~\cite{conway1978functions}*{Corollary~IV.3.9}, there is therefore an analytic function $\gamma$ in a neighborhood of $[0,1]$ such that
	\begin{align}\label{eq: AT noise decomposition}
		\Sigma_{\mathrm{AT}}(x)=x(1-x)\gamma(x).
	\end{align}
	The factorization and the absence of interior zeros give $\gamma>0$ on $(0,1)$. By continuity and simplicity of the endpoint zeros, $\gamma(0)$ and $\gamma(1)$ are also positive. Hence
	\begin{align}\label{ineq: AT gamma positive}
		\min_{x\in[0,1]}\gamma(x)>0.
	\end{align}

	Define
	\begin{align*}
		b(x)&:=\alpha\beta(x),
		&
		\sigma(x)&:=\alpha\gamma(x),
		&
		\theta^+&:=\alpha r_+,
		&
		\theta^-&:=\alpha r_-.
	\end{align*}
	By~\eqref{eq: AT drift decomposition} and~\eqref{eq: AT noise decomposition},
	\begin{gather*}
		b(x)x(1-x)+\theta^+(1-x)-\theta^-x
		=\alpha B_{\mathrm{AT}}(x),\\
		\sigma(x)x(1-x)
		=\alpha\Sigma_{\mathrm{AT}}(x).
	\end{gather*}
	Thus~\eqref{eq: AT alpha scaled equation} is precisely the limiting equation in Corollaries~\ref{cor:pop-limit} and~\ref{cor:dirichlet-convergence} for these coefficients.

	For the varying-population construction, since $b$ and $\sigma$ are continuous and $\sigma>0$ on $[0,1]$, Corollary~\ref{cor:pop-limit} applies. For the Dirichlet construction, take $h_N=N^{-1/2}$ and, for all sufficiently large $N$, set $\theta_N^\pm=h_N^2\theta^\pm$, choosing the finitely many remaining mutation probabilities subject to admissibility. Moreover,
	\begin{align*}
		\frac1{Nh_N}
		&=\frac1{\sqrt N}
		\le\alpha\min_{x\in[0,1]}\gamma(x)
		=\min_{x\in[0,1]}\sigma(x)
	\end{align*}
	for all sufficiently large $N$, where the inequality follows from~\eqref{ineq: AT gamma positive}. Hence~\eqref{eq:dir-scale-compatibility} holds, and Corollary~\ref{cor:dirichlet-convergence} applies. Under either construction, $\widetilde\xi^N$ converges in distribution to $w$ in $\mathcal D(\R_+;\mathcal C_\R^{\mathrm{tem}})$. The map $f\mapsto a_-+\ell f_{\cdot/\alpha}$ is continuous on this path space, so the continuous mapping theorem gives~\eqref{eq: AT approximation conclusion}.
\end{proof}

\section{Proof of Theorem~\ref{thm:canning_convergence}}\label{sec:proof-canning}

Recall $W_A^N(x)$ and $P_N(x,\mathcal W(x))$ from~\eqref{eq: pnxw}, and $\widehat P_N(x,\mathcal W(x))$ from~\eqref{eq: mutation adjusted probability}. The exact identity
\begin{align}\label{eq:tylor expansion}
	P_N(x,\mathcal W(x))
	=W_A^N(x)+s_N(x)W_A^N(x)(1-W_A^N(x))+\rho_N(x)
\end{align}
holds, where
\begin{align*}
	\rho_N(x):= -\frac{s_N(x)^2W_A^N(x)^2(1-W_A^N(x))}{1+s_N(x)W_A^N(x)}.
\end{align*}
Whenever $\abs{s_N(x)}<1/2$,
\begin{align}\label{ineq:canning_taylor_error}
	\abs{\rho_N(x)}\le 2\abs{s_N(x)}^2,\qquad  \abs{P_N(x,\mathcal W(x))-W_A^N(x)} \le \frac{\abs{s_N(x)}}{2}.
\end{align}
Furthermore,~\eqref{eq: mutation adjusted probability} gives
\begin{equation}\label{ineq: mutation probability difference}
	\begin{gathered}
		\abs{\widehat P_N(x,\mathcal W(x))-P_N(x,\mathcal W(x))}
		\le \theta_N^++\theta_N^-,\\
		\abs{\widehat P_N(x,\mathcal W(x))-W_A^N(x)}
		\overset{\eqref{ineq:canning_taylor_error}}{\le} \frac{\abs{s_N(x)}}2+\theta_N^++\theta_N^-.
	\end{gathered}
\end{equation}
One can check from exchangeability of $\mathcal W(x)$ that
\begin{align*}
	\E\left[W_A^N(x) \right] = x.
\end{align*}
Recall $X_1$, $\P_x$ and $\E_x$ from Section~\ref{sec:canning}.
\begin{lemma}\label{lem:index exchangeable}
	Let $N\ge2$ and $x\in[N^{-1}]$. For $1\le i,j\le N$ with $i\ne j$, exchangeability of $\mathcal W(x)$ gives
	\begin{align*}
		\E_x[w_i(x)]=\frac1N,\qquad
		\E_x[w_i(x)^2]=\frac{c_N(x)}N,\qquad
		\E_x[w_i(x)w_j(x)]=\frac{1-c_N(x)}{N(N-1)}.
	\end{align*}
\end{lemma}

\begin{proof}
	By exchangeability, the normalization of the weights, and~\eqref{eq:cn scaling},
	\begin{align*}
		1 =\E_x\left[\sum_{\ell=1}^N w_\ell(x)\right]
		=N\E_x[w_i(x)],\qquad
		c_N(x)
		=\E_x\left[\sum_{\ell=1}^N w_\ell(x)^2\right]
		=N\E_x[w_i(x)^2].
	\end{align*}
	Moreover,
	\begin{align*}
		1
		&=\E_x\left[\left(\sum_{\ell=1}^Nw_\ell(x)\right)^2\right]\\
		&=\E_x\left[\sum_{\ell=1}^Nw_\ell(x)^2\right]
		+\E_x\left[\sum_{\substack{\ell,m=1\\\ell\ne m}}^N
		w_\ell(x)w_m(x)\right]\\
		&=c_N(x)+N(N-1)\E_x[w_i(x)w_j(x)],
	\end{align*}
	where we used exchangeability of $\mathcal W(x)$ in the last equality. Solving for $\E_x[w_i(x)w_j(x)]$ gives the last identity.
\end{proof}

\begin{lemma}\label{lem:wa var}
	For each $N\ge2$ and $x\in[N^{-1}]$,
	\begin{align*}
		\E_x\left[W_A^N(x)\left(1-W_A^N(x)\right)\right]
		&=x(1-x)-\frac{N}{N-1}x(1-x)
		\left(c_N(x)-\frac1N\right),\\
		\E_x\left[\left(W_A^N(x)-x\right)^2\right]
		&=x(1-x)\frac{N}{N-1}
		\left(c_N(x)-\frac1N\right).
	\end{align*}
\end{lemma}

\begin{proof}
	Since $I_x$ is deterministic and has cardinality $Nx$,
	\begin{align*}
		\E_x[W_A^N(x)]
		&=\sum_{i\in I_x}\E_x[w_i(x)]
		\overset{\mathrm{Lem.}\,\ref{lem:index exchangeable}}{=}\frac{Nx}{N}
		=x,
	\end{align*}
	and
	\begin{align*}
		\E_x\left[W_A^N(x)^2\right]
		&=\sum_{i\in I_x}\E_x[w_i(x)^2]
		+\sum_{\substack{i,j\in I_x\\i\ne j}}
		\E_x[w_i(x)w_j(x)]\\
		&\overset{\mathrm{Lem.}\,\ref{lem:index exchangeable}}{=}
		Nx\frac{c_N(x)}N
		+Nx(Nx-1)\frac{1-c_N(x)}{N(N-1)}.
	\end{align*}
	Consequently,
	\begin{align*}
		\E_x\left[\left(W_A^N(x)-x\right)^2\right]
		&=\E_x\left[W_A^N(x)^2\right]-x^2\\
		&=xc_N(x)+\frac{x(Nx-1)}{N-1}\left(1-c_N(x)\right)-x^2\\
		&=x(1-x)\frac{Nc_N(x)-1}{N-1}\\
		&=x(1-x)\frac{N}{N-1}
		\left(c_N(x)-\frac1N\right).
	\end{align*}
	Finally,
	\begin{align*}
		\E_x\left[W_A^N(x)\left(1-W_A^N(x)\right)\right]
		&=x-\E_x\left[W_A^N(x)^2\right]\\
		&=x-\left(
		\E_x\left[\left(W_A^N(x)-x\right)^2\right]+x^2\right)\\
		&=x(1-x)-\frac{N}{N-1}x(1-x)
		\left(c_N(x)-\frac1N\right).
	\end{align*}
\end{proof}

\begin{lemma}\label{lem:canning_drift}
	Let $N\ge2$ and suppose $\abs{s_N(x)}<1/2$ for all $x \in [N^{-1}]$, then the following  holds
	\begin{align*}
		&\abs{\mathbb E_x\left[X_1-x\right]
		-s_N(x)x(1-x)-\theta_N^+(1-x)+\theta_N^-x}
		\\
		\le&
		\abs{s_N(x)}\pr{c_N(x)+N^{-1}}
		+2\abs{s_N(x)}^2
		+\frac14\abs{s_N(x)}(\theta_N^++\theta_N^-),
		\qquad x \in [N^{-1}].
	\end{align*}
\end{lemma}
\begin{proof}
	We first note that by the conditional binomial law, we have
	\begin{align}\label{eq:canning_drift_1}
		\E_x\left[X_1-x\right]&=\E_x\left[\E_x\left[X_1-x\mid\mathcal W(x)\right]\right]\nonumber\\
		&=\E_x\left[\widehat P_N(x,\mathcal W(x))-x\right]\nonumber\\
		&{\overset{\eqref{eq: mutation adjusted probability}}{=}}(1-\theta_N^+-\theta_N^-)\E_x \left[ \left(P_N(x,\mathcal W(x))-x\right)\right]+\theta_N^+(1-x)-\theta_N^-x.
	\end{align}
	By Lemma~\ref{lem:wa var}, we have
	\begin{align*}
		\E_x\left[P_N(x,\mathcal W(x)) - x \right]&{\overset{\eqref{eq:tylor expansion}}{=}}\E_x\left[W_A^N(x)-x+s_N(x)W_A^N(x)(1-W_A^N(x))+\rho_N(x)\right]\\
		&=s_N(x)\E_x\left[W_A^N(x)(1-W_A^N(x))\right]+\E_x\left[\rho_N(x)\right]\\
		&{\overset{\mathrm{Lem.}\,\ref{lem:wa var}}{=}}s_N(x)x(1-x)
		\left(1-\frac{N}{N-1}\left(c_N(x)-\frac1N\right)\right)
		+\E_x\left[\rho_N(x)\right].
	\end{align*}
	Therefore,
	\begin{align}
		&\abs{\E_x[P_N(x,\mathcal W(x))-x]-s_N(x)x(1-x)}\notag\\
		&\quad\le
		\abs{s_N(x)}\frac{N}{N-1}x(1-x)
		\abs{c_N(x)-N^{-1}}+\E_x\left[\abs{\rho_N(x)}\right]\notag\\
		&\quad\le
		\abs{s_N(x)}\pr{c_N(x)+N^{-1}}+2\abs{s_N(x)}^2.
		\label{ineq:canning_drift_2}
	\end{align}
	Here we used $\frac{N}{N-1}x(1-x)\le1$ for all $N \ge 2$, $x \in [0,1]$ and~\eqref{ineq:canning_taylor_error} in the last line. Hence,
	\begin{align*}
		&\qquad\abs{\E_x\left[X_1-x\right]-s_N(x)x(1-x)
		-\theta_N^+(1-x)+\theta_N^-x}\\
		&\overset{\eqref{eq:canning_drift_1}}{=}
		\abs{(1-\theta_N^+-\theta_N^-)
		\E_x\left[P_N(x,\mathcal W(x))-x\right]
		-s_N(x)x(1-x)}\\
		&=
		\abs{(1-\theta_N^+-\theta_N^-)
		\left(\E_x\left[P_N(x,\mathcal W(x))-x\right]
		-s_N(x)x(1-x)\right)
		-(\theta_N^++\theta_N^-)s_N(x)x(1-x)}\\
		&\le(1-\theta_N^+-\theta_N^-)
		\abs{\E_x\left[P_N(x,\mathcal W(x))-x\right]
		-s_N(x)x(1-x)}
		+(\theta_N^++\theta_N^-)\abs{s_N(x)}x(1-x)\\
		&\overset{\eqref{ineq:canning_drift_2}}{\le}
		\abs{s_N(x)}\pr{c_N(x)+N^{-1}}
		+2\abs{s_N(x)}^2
		+\frac14\abs{s_N(x)}(\theta_N^++\theta_N^-),
	\end{align*}
	where the last inequality follows from $\theta_N^++\theta_N^-\le1$ and $x(1-x)\le1/4$.
\end{proof}

We now consider the one-step squared increment of the paintbox CG model.
\begin{lemma}\label{lem:canning_var}
	Let $N\ge2$ and suppose $\abs{s_N(x)}<1/2$ for all $x \in [N^{-1}]$, then the following  holds
	\begin{align*}
		\abs{\E_x\left[\pr{X_1-x}^2\right]-x(1-x)c_N(x)}
		\le
		\frac{\abs{s_N(x)}}{2N}
		+\abs{s_N(x)}\sqrt{\frac{c_N(x)}2}
		+\frac{\abs{s_N(x)}^2}{4}
		+3(\theta_N^++\theta_N^-).
	\end{align*}
\end{lemma}
\begin{proof}
	Note, it holds for all $x \in [N^{-1}]$ that
	\begin{align}\label{eq:canning_var_step1}
		\E_x\left[\pr{X_1-x}^2\mid\mathcal W(x)\right]
		\overset{\eqref{eq:canning_conditional_binomial}}{=}\frac{\widehat P_N(x,\mathcal W(x))
		\left(1-\widehat P_N(x,\mathcal W(x))\right)}N
		+\left(\widehat P_N(x,\mathcal W(x))-x\right)^2.
	\end{align}
	Set
	\[
		F_N(y,x):=\frac{y(1-y)}N+(y-x)^2.
	\]
	For $x,y,y'\in[0,1]$,
	\begin{align}\label{ineq: mutation variance comparison}
		\abs{F_N(y,x)-F_N(y',x)}
		&\le \frac1N\abs{y-y'}+2\abs{y-y'}
		\le3\abs{y-y'}.
	\end{align}
	Hence, by the triangular inequality, it holds for all $N \ge 2$ and $x \in [N^{-1}]$ that
	\begin{align*}
		\abs{\E_x\left[\pr{X_1-x}^2 \right] - x(1-x)c_N(x)}&{\overset{\eqref{eq:canning_var_step1}}{=}} \abs{\E_x\left[F_N(\widehat{P}_N(x,\mathcal W(x)),x)\right]-x(1-x)c_N(x)}\\
		&\le \abs{\E_x\left[F_N(\widehat{P}_N(x,\mathcal W(x)),x)-F_N(P_N(x,\mathcal W(x)),x)\right]}\\
		&\qquad + \abs{\E_x \left[F_N(P_N(x,\mathcal W(x)),x) - x(1-x)c_N(x) \right]}\\
		&{\overset{\eqref{ineq: mutation variance comparison}}{\le}}3\E_x\left[\abs{\widehat{P}_N(x,\mathcal W(x))-P_N(x,\mathcal W(x))}\right]\\
		&\qquad +\abs{\E_x \left[F_N(P_N(x,\mathcal W(x)),x) - x(1-x)c_N(x) \right]}\\
		&{\overset{\eqref{ineq: mutation probability difference}}{\le}}3(\theta_N^++\theta_N^-)+\abs{\E_x \left[F_N(P_N(x,\mathcal W(x)),x) - x(1-x)c_N(x) \right]}.
	\end{align*}
	Therefore, it is enough to show
	\begin{align}\label{ineq:canning_var_reduction_target}
		\abs{\E_x \left[F_N(P_N(x,\mathcal W(x)),x) - x(1-x)c_N(x) \right]}
		\le
		\frac{\abs{s_N(x)}}{2N}
		+\abs{s_N(x)}\sqrt{\frac{c_N(x)}2}
		+\frac{\abs{s_N(x)}^2}{4}.
	\end{align}

	Next, consider the following identities,
	\begin{align*}
		&P_N(x,\mathcal W(x))
		\left(1-P_N(x,\mathcal W(x))\right)
		-W_A^N(x)\left(1-W_A^N(x)\right)\\
		&\qquad
		=\left(P_N(x,\mathcal W(x))-W_A^N(x)\right)
		\left(1-2W_A^N(x)\right)
		-\left(P_N(x,\mathcal W(x))-W_A^N(x)\right)^2,\\
		&\left(P_N(x,\mathcal W(x))-x\right)^2
		-\left(W_A^N(x)-x\right)^2\\
		&\qquad
		=2\left(W_A^N(x)-x\right)
		\left(P_N(x,\mathcal W(x))-W_A^N(x)\right)
		+\left(P_N(x,\mathcal W(x))-W_A^N(x)\right)^2.
	\end{align*}
	Lemma~\ref{lem:wa var} yields
	\begin{align*}
		\E_x\left[F_N(W_A^N(x),x)\right]
		=x(1-x)c_N(x).
	\end{align*}
	This implies, for all $x \in [N^{-1}]$,
	\begin{align}\label{eq:canning_var_error_remainder}
		\E_x\left[F_N(P_N(x,\mathcal W(x)),x)\right]-x(1-x)c_N(x)
		&=\E_x\left[F_N(W_A^N(x),x) \right]-x(1-x)c_N(x)\nonumber\\
		&\qquad + \E_x\left[F_N(P_N(x,\mathcal W(x)),x) - F_N(W_A^N(x),x) \right]\nonumber\\
		&=r_{N,1}(x)+r_{N,2}(x)+r_{N,3}(x),
	\end{align}
	where
	\begin{align*}
		r_{N,1}(x)
		&:=\frac1N\E_x\left[
		\left(P_N(x,\mathcal W(x))-W_A^N(x)\right)
		\left(1-2W_A^N(x)\right)
		\right],\\
		r_{N,2}(x)
		&:=2\E_x\left[
		\left(W_A^N(x)-x\right)
		\left(P_N(x,\mathcal W(x))-W_A^N(x)\right)
		\right],\\
		r_{N,3}(x)
		&:=\left(1-\frac1N\right)
		\E_x\left[
		\left(P_N(x,\mathcal W(x))-W_A^N(x)\right)^2
		\right].
	\end{align*}
	For the error terms, we see by~\eqref{ineq:canning_taylor_error} that
	\begin{align*}
		\abs{r_{N,1}(x)}
		\le\frac{\abs{s_N(x)}}{2N},\text{ and }
		\abs{r_{N,3}(x)}
		\le\frac{\abs{s_N(x)}^2}{4},\qquad \text{ for all }x \in [N^{-1}].
	\end{align*}
	For $r_{N,2}$ term, consider by Cauchy-Schwarz inequality that
	\begin{align*}
		\abs{r_{N,2}(x)}
		&\le2 \norm{W_A^N(x) - x}_{L^2_\Omega} \norm{P_N(x,\mathcal W(x))-W_A^N(x)}_{L^2_\Omega}\\
		&{\overset{\eqref{ineq:canning_taylor_error}}{\le} }\norm{W_A^N(x) - x}_{L^2_\Omega} \abs{s_N(x)}\\
		&{\overset{\mathrm{Lem.}\,\ref{lem:wa var}}{=}}\pr{x(1-x)\frac{N}{N-1}\left(c_N(x)-\frac1N\right)}^\frac{1}{2}\abs{s_N(x)}\\
		&\overset{\mathrm{Lem.}\,\ref{lem:wf-floor}}{\le}\abs{s_N(x)}\sqrt{\frac{c_N(x)}2}.
	\end{align*}
	Inserting the bounds for $r_{N,1}$, $r_{N,2}$, and $r_{N,3}$ into~\eqref{eq:canning_var_error_remainder} gives~\eqref{ineq:canning_var_reduction_target}, which completes the proof.
\end{proof}

Next we consider higher moments of one-step difference of the CG model.
\begin{lemma}\label{lem:canning_higher_moments}
	Let $p\ge2$, $N\ge2$, and $x\in[N^{-1}]$, and suppose $\abs{s_N(x)}<1/2$. Then
	\begin{align*}
		\E_x \left[\abs{X_1 - x}^p \right]
		\lesssim
		N^{-p/2}+\abs{s_N(x)}^p
		+(\theta_N^++\theta_N^-)^p
		+\E_x\left[\abs{W_A^N(x)-x}^p\right].
	\end{align*}
\end{lemma}

\begin{proof}
	Let $p \ge 2$, we see by the triangle inequality, it holds that
	\begin{align*}
		\E_x\left[\abs{X_1-x}^p\right]
		&\lesssim
		\E_x\left[
		\E_x\left[\abs{X_1-\widehat P_N(x,\mathcal W(x))}^p
		\mid\mathcal W(x)\right]
		+\abs{\widehat P_N(x,\mathcal W(x))-x}^p\right]\\
		&{\overset{\eqref{eq:canning_conditional_binomial}}{\lesssim}}
		N^{-p/2}+\E_x\left[\abs{\widehat P_N(x,\mathcal W(x))-x}^p\right]\\
		&\lesssim N^{-p/2}+\E_x\left[\abs{W_A^N(x)-x}^p\right]+\E_x\left[\abs{W_A^N(x) - \widehat{P}_N(x,\mathcal W(x)) }^p \right]\\
		&{\overset{\eqref{ineq: mutation probability difference}}{\lesssim}} N^{-p/2}
		+\E_x\left[\abs{W_A^N(x) - x}^p \right]
		+\abs{s_N(x)}^p +\pr{\theta_N^++\theta_N^-}^p,
	\end{align*}
	which is the desired result.
\end{proof}

\begin{lemma}\label{lem:canning_unbiased_transfer}
	Under the assumptions of Theorem~\ref{thm:canning_convergence},
	\begin{align}\label{eq:mean_uniform_rv_canning_cov}
		\lim_{N\to\infty}\sup_{x\in[0,1]}
		\left|h_N^{-2}\E_x[\overline{\mathcal R}^N(x)-x]
		-B(x)\right|=0.
	\end{align}
	Moreover,
	\begin{align}\label{eq:var_uniform_rv_canning_cov}
		\lim_{N\to\infty}\sup_{x\in[0,1]}
		\left|h_N^{-1}\E_x[(\overline{\mathcal R}^N(x)-x)^2]
		-\sigma(x)x(1-x)\right|=0.
	\end{align}
	With $p>8$ and $\varepsilon>0$ as in~\eqref{ineq:p_moment_canning_thm}, it holds that
	\[
		\sup_{x\in[0,1]}\E_x\left[|\overline{\mathcal R}^N(x)-x|^p\right]
		\lesssim h_N^{1+\kappa_p},
		\qquad \kappa_p:=\min\{p/2,4+\varepsilon\}-1>3.
	\]
\end{lemma}
\begin{proof}
	The limits in~\eqref{eq:bsigma def} and the fact that $c_N(z)\ge N^{-1}$ imply that, for some $C>0$ and all sufficiently large $N$,
	\begin{align}\label{ineq:canning_uniform_initial_bounds}
		\sup_{z\in[N^{-1}]}|s_N(z)|\le Ch_N^2,
		\qquad N^{-1}\le\sup_{z\in[0,1]}c_N(z)\le(\|\sigma\|_\infty+1)h_N,
		\qquad
		\theta_N^++\theta_N^-\le Ch_N^2.
	\end{align}
	In particular, $|s_N(z)|<1/2$ for all $z\in[N^{-1}]$ and all sufficiently large $N$, so Lemmas~\ref{lem:canning_drift}--\ref{lem:canning_higher_moments} apply. Recall the uniform variable $V$ defining $U_N$.  Since $\E[U_N(x)]=x$, it holds for all $x \in [0,1]$ and $N \in \N$ that
	\begin{align*}
		\E_x[\overline{\mathcal R}^N(x)-x]
		&=\E\left[
		\E\left[\widetilde{\mathcal R}^N(U_N(x))-U_N(x)\mid V\right]
		\right]
		+\E[U_N(x)-x]\\
		&=\E\left[
		\E\left[\widetilde{\mathcal R}^N(U_N(x))-U_N(x)\mid V\right]
		\right].
	\end{align*}
	Adding and subtracting $B(U_N(x))$ now yields
	\begin{align*}
		h_N^{-2}\E_x[\overline{\mathcal R}^N(x)-x]
		-B(x)
		=\E[
		h_N^{-2}\E[\widetilde{\mathcal R}^N(U_N(x))
		-U_N(x)\mid V]
		-B(U_N(x))
		]
		+\E[B(U_N(x))-B(x)].
	\end{align*}
	Hence, by the triangular inequality, it holds for $N \in \N$ that
	\begin{align*}
		\sup_{x\in[0,1]}
		\left|h_N^{-2}\E_x[\overline{\mathcal R}^N(x)-x]-B(x)\right|
		&\overset{\eqref{ineq:uniform_rv_estimates}}{\le} \sup_{z\in[N^{-1}]}
		\left|h_N^{-2}\E_z[\widetilde{\mathcal R}^N(z)-z]-B(z)\right|\\
		&\qquad+\sup_{\substack{x\in[0,1],\ z\in[N^{-1}]\\|x-z|\le N^{-1}}}
		|B(z)-B(x)|.
	\end{align*}
	The first conclusion follows from Lemma~\ref{lem:canning_drift},~\eqref{eq:bsigma def}, and uniform continuity of $B$ on $[0,1]$.

	For the second moment, conditioning on $V$ and adding and subtracting $\sigma(U_N(x))U_N(x)(1-U_N(x))$ gives
	\begin{equation}\label{eq:canning_unbiased_transfer_var}
		\begin{aligned}
			&h_N^{-1}\E_x[(\overline{\mathcal R}^N(x)-x)^2]
			-\sigma(x)x(1-x)\\
			&=h_N^{-1} \E\left[\pr{\widetilde{\mathcal R}^N(U_N(x)) - U_N(x) + U_N(x) - x}^2\right]-\sigma(x)x(1-x)\\
			&=\E\left[
			h_N^{-1}\E\left[(\widetilde{\mathcal R}^N(
			U_N(x))-U_N(x))^2\mid V\right]
			-\sigma(U_N(x))U_N(x)(1-U_N(x))
			\right]\\
			&\quad+\E\left[\sigma(U_N(x))
			U_N(x)(1-U_N(x))
			-\sigma(x)x(1-x)\right]\\
			&\quad+\frac2{h_N}\E\left[(U_N(x)-x)
			\E\left[\widetilde{\mathcal R}^N(U_N(x))
			-U_N(x)\mid V\right]\right]\\
			&\quad+\frac1{h_N}\E[(U_N(x)-x)^2].
		\end{aligned}
	\end{equation}
	By Lemma~\ref{lem:canning_var}, it holds that
	\begin{align*}
		&\abs{\E\left[
		h_N^{-1}\E\left[(\widetilde{\mathcal R}^N(
		U_N(x))-U_N(x))^2\mid V\right]
		-\sigma(U_N(x))U_N(x)(1-U_N(x))
		\right]} \\
		\le &\sup_{z \in [N^{-1}]}
		\abs{h_N^{-1}\E\left[(\widetilde{\mathcal R}^N(z)-z)^2\right]
		-\sigma(z)z(1-z)}\\
		\le &\sup_{z \in [N^{-1}]}\abs{h_N^{-1}\pr{\E \left[ (\widetilde{\mathcal R}^N(z)-z)^2\right] - z(1-z)c_N(z)} }+ \sup_{z \in [N^{-1}]}\abs{z(1-z)\frac{c_N(z)}{h_N} - \sigma(z)z(1-z)}\\
		\le & \sup_{z \in [N^{-1}]}\pr{\frac{\abs{s_N(z)}}{2Nh_N}
		+\abs{s_N(z)}\sqrt{\frac{c_N(z)}2}h_N^{-1}
		+\frac{\abs{s_N(z)}^2}{4h_N}
		+3(\theta_N^++\theta_N^-)h_N^{-1}}
		+ \sup_{z \in [N^{-1}]}\abs{\frac{c_N(z)}{h_N} - \sigma(z)}\\
		\overset{\eqref{ineq:canning_uniform_initial_bounds}}{\le}& C\pr{N^{-1}h_N + h_N^{\frac{3}{2}} + h_N^3+h_N} + \sup_{z \in [N^{-1}]}\abs{\frac{c_N(z)}{h_N} - \sigma(z)}.
	\end{align*}
	Also,
	\begin{align*}
		\sup_{x \in [0,1]}
		\abs{\E\left[\sigma(U_N(x))U_N(x)(1-U_N(x))
		-\sigma(x)x(1-x)\right]}
		&\le \sup_{\substack{x,z \in [0,1], \\|x-z|\le N^{-1}}}
		\abs{\sigma(z)z(1-z)-\sigma(x)x(1-x)}.
	\end{align*}
	By the conditional binomial law, $\E_z[W_A^N(z)]=z$, and \eqref{ineq:canning_taylor_error} and \eqref{ineq: mutation probability difference}, it holds that
	\begin{align*}
		\frac2{h_N}\abs{\E\left[(U_N(x)-x)\E\left[\widetilde{\mathcal R}^N(U_N(x))-U_N(x)\mid V\right]\right]}&
		{\overset{\eqref{ineq:uniform_rv_estimates}}{\le}}
		\frac{2}{h_NN}\sup_{z\in[N^{-1}]}\abs{\E_z[\widetilde{\mathcal R}^N(z)-z]}\\
		&\le\frac{2}{h_NN}\pr{\frac12
		\sup_{z\in[N^{-1}]}\abs{s_N(z)}+\theta_N^+ +\theta_N^-}\\
		&{\overset{\eqref{ineq:canning_uniform_initial_bounds}}{\le}} Ch_N^2.
	\end{align*}
	Lastly, there is some $C>0$ such that, for all sufficiently large $N$,
	\[
		\sup_{x \in [0,1]}h_N^{-1}\E[(U_N(x)-x)^2]
		\overset{\eqref{ineq:uniform_rv_estimates}}{\le}
		N^{-2}h_N^{-1}
		\overset{\eqref{ineq:canning_uniform_initial_bounds}}{\le} C N^{-1}.
	\]
	Inserting these estimates into~\eqref{eq:canning_unbiased_transfer_var}, there is some $C>0$ such that, for all sufficiently large $N$,
	\begin{align*}
		\sup_{x \in [0,1]}\abs{h_N^{-1}\E_x[(\overline{\mathcal R}^N(x)-x)^2]-\sigma(x)x(1-x)}
		&\le C \pr{h_N+N^{-1}} +\sup_{z \in [N^{-1}]}\abs{\frac{c_N(z)}{h_N} - \sigma(z)}\\
		&\qquad+ \sup_{\substack{x,z \in [0,1], \\|x-z|\le N^{-1}}}\abs{\sigma(z)z(1-z)-\sigma(x)x(1-x)},
	\end{align*}
	which converges to $0$ as $N \to \infty$ by~\eqref{eq:bsigma def} and uniform continuity of $\sigma$ on $[0,1]$. This proves~\eqref{eq:var_uniform_rv_canning_cov}.

	Finally, by the definition of $\overline{\mathcal R}^N$ and Lemma~\ref{lem:canning_higher_moments}, it holds uniformly for sufficiently large $N \in \N$ that
	\begin{align*}
		\sup_{x\in[0,1]}\E_x\left[\abs{\overline{\mathcal R}^N(x)-x}^p\right]
		&{=}\sup_{x\in[0,1]}\E\left[
		\abs{\pr{\widetilde{\mathcal R}^N(U_N(x))
		-U_N(x)}
		+\pr{U_N(x)-x}}^p\right]\\
		&\le 2^{p-1}\sup_{x\in[0,1]}\E\left[
		\abs{\widetilde{\mathcal R}^N(U_N(x))
		-U_N(x)}^p
		+\abs{U_N(x)-x}^p\right]\\
		&{\overset{\eqref{ineq:uniform_rv_estimates}}{\le}} 2^{p-1}\left(
		\sup_{z\in[N^{-1}]}\E_z\left[\abs{\widetilde{\mathcal R}^N(z)-z}^p\right]
		+N^{-p}\right)\\
		&{\overset{\mathrm{Lem.}\,\ref{lem:canning_higher_moments}}{\lesssim}}
		N^{-p/2}+\sup_{z\in[N^{-1}]}\abs{s_N(z)}^p
		+(\theta_N^++\theta_N^-)^p\\
		&\quad+\sup_{z\in[N^{-1}]}\E_z\left[\abs{W_A^N(z)-z}^p\right]+N^{-p}\\
		&{\overset{\substack{\eqref{ineq:canning_uniform_initial_bounds}}}{\lesssim}}
		h_N^{p/2}+h_N^{2p}
		+h_N^{4+\varepsilon}+h_N^p\\
		&\lesssim h_N^{\min\{p/2,4+\varepsilon\}},
	\end{align*}
	where we also used~\eqref{ineq:p_moment_canning_thm} in the fifth line. This completes the proof of the lemma.
\end{proof}

\begin{proof}[Proof of Theorem~\ref{thm:canning_convergence}]
	We use Theorem~\ref{thm:main}, which is stated along a sequence $L\to\infty$: an $N$-indexed sequence of pairs $(L_N,\mathcal R_N)$ with $L_N>0$ and $L_N\to\infty$ may replace $(L,\mathcal R^L)$ throughout its hypotheses and conclusion. Set
	\[
		L_N:=h_N^{-1},
		\qquad
		\mathcal R_N:=\overline{\mathcal R}^N.
	\]
	Lemma~\ref{lem:canning_unbiased_transfer} gives
	\begin{align*}
		\lim_{N\to\infty}\sup_{x\in[0,1]}
		\abs{h_N^{-2}\E_x\bracket{\mathcal R_N(x)-x}
		-B(x)}&=0,\\
		\lim_{N\to\infty}\sup_{x\in[0,1]}
		\abs{h_N^{-1}\E_x\bracket{\pr{\mathcal R_N(x)-x}^2}
		-\sigma(x)x(1-x)}&=0.
	\end{align*}
	In terms of the coefficients defined in~\eqref{eq:R assumption}, these estimates give~\eqref{eq:coefficient convergence}, with limiting coefficients $B(x)$ and $\sigma(x)x(1-x)$, respectively. Since $B(0)=\theta^+\ge0$, $B(1)=-\theta^-\le0$, and $\sigma(x)x(1-x)$ vanishes at $x=0,1$, the boundary assumptions of Theorem~\ref{thm:main} hold. The third conclusion of Lemma~\ref{lem:canning_unbiased_transfer} gives~\eqref{ineq:p_moment bound} with $\kappa_p>3$. Hence the desired result follows from Theorem~\ref{thm:main}.
\end{proof}

\section{Preliminaries}\label{sec:prelim-estimates}
In Sections~\ref{sec:prelim-estimates}--\ref{sec:limit-points} we write $L$ for the sequence $(L_n)$ and assume, without loss of generality, that $L\in\N$; the general case follows by rescaling space by $\lfloor L\rfloor/L\to1$ and time by $(\lfloor L\rfloor/L)^2$.

\subsection{Stochastic partial differential equations}

Following \cite{shiga1994two}, we fix a filtered probability space $(\Omega,\mathcal F,(\mathcal F_t)_{t\ge 0},\P)$ and an $(\mathcal F_t)$-space--time white noise $W(t,x)$ on $\R_+\times\R$. Formally, $\dot W$ is the generalized Gaussian field with covariance
\begin{align*}
	\E\bigl[\dot W(t,x)\dot W(t',x')\bigr]=\delta(t-t')\,\delta(x-x'),
	\qquad t,t'\ge 0,\ x,x'\in \R.
\end{align*}
A convenient rigorous interpretation is via the Walsh stochastic integral with respect to the associated martingale measure; see\ \cite{walsh1986introduction}*{Section 2}.

For any $(\mathcal F_t)$-predictable process $\phi(t,x)$ satisfying
\begin{align*}
	\int_0^t\int_{\R} \abs{\phi(s,x)}^2\,\mathrm dx\,\mathrm ds < \infty,
	\qquad \text{for every } t>0,
\end{align*}
we can define the stochastic integral $\int_0^t\int_{\R}\phi(s,x)\,W(\mathrm ds,\mathrm dx)$, which is an $(\mathcal F_t)$-local martingale with quadratic variation $\int_0^t\int_{\R} \abs{\phi(s,x)}^2\,\mathrm dx\,\mathrm ds$.

Let $a(t,x,u)$ and $b(t,x,u)$ be measurable functions on $\R_+\times \R \times \R$, and  consider reaction--diffusion type SPDEs of the form
\begin{equation}\label{eq:spde-generic}
	\begin{gathered}
		\partial_t u(t,x)=\frac{\alpha}{2}\Delta u(t,x)+b\bigl(t,x,u(t,x)\bigr)+a\bigl(t,x,u(t,x)\bigr)\,\dot W(t,x),\qquad (t,x)\in \R_+\times\R,\\
		u(0,x) = f(x).
	\end{gathered}
\end{equation}

More precisely,~\eqref{eq:spde-generic} is understood in the sense of Schwartz distributions: for every $\varphi\in C_c^\infty(\R)$ and $t\ge 0$,
\begin{equation}\label{eq:spde-weak-form}
	\begin{aligned}
		\langle u_t,\varphi\rangle
		&=\langle f,\varphi\rangle
		+\int_0^t \frac{\alpha}{2}\,\langle u_s,\Delta\varphi\rangle\,\mathrm ds
		+\int_0^t \langle b(s,\cdot,u_s(\cdot)),\varphi\rangle\,\mathrm ds \\
		&\qquad +\int_0^t\int_{\R} a(s,x,u_s(x))\,\varphi(x)\,W(\mathrm ds,\mathrm dx).
	\end{aligned}
\end{equation}
If $u$ satisfies~\eqref{eq:spde-weak-form}, then we say $u$ is a  solution to~\eqref{eq:spde-generic}.

Let $G_t(x):=\frac{1}{\sqrt{2\pi \alpha t}}\exp\!\left(-\frac{x^2}{2\alpha t}\right)$ denote the one-dimensional heat kernel associated with $\partial_t u=\frac{\alpha}{2}\Delta u$. We say $u$ is a mild solution to~\eqref{eq:spde-generic} if for every $t\ge 0$ and $x\in \R$,
\begin{align}\label{eq:spde-mild-form}
	u(t,x)
	&= \int_{\R} G_t(x-y) f(y)\,\mathrm dy
	+ \int_0^t\int_{\R} G_{t-s}(x-y)\, b\bigl(s,y,u(s,y)\bigr)\,\mathrm dy\,\mathrm ds \\
	&\qquad + \int_0^t\int_{\R} G_{t-s}(x-y)\, a\bigl(s,y,u(s,y)\bigr)\, W(\mathrm ds,\mathrm dy).
\end{align}
Under mild regularity and growth conditions on the coefficients $a$ and $b$, $u$ is a  solution to~\eqref{eq:spde-generic} if and only if it is a mild solution (see\ \cite{shiga1994two}*{Theorem 2.1}).

Lastly, a \emph{probabilistic weak solution} of~\eqref{eq:spde-generic} consists of a filtered probability space $(\Omega,\mathcal F,(\mathcal F_t)_{t\ge0},\P)$, a space--time white noise $W$ on it, and an $(\mathcal F_t)$-adapted process $u$ with paths in $\mathcal C(\R_+;\mathcal C^{\mathrm{tem}}_\R)$ such that~\eqref{eq:spde-weak-form} holds almost surely for every $\varphi\in C_c^\infty(\R)$ and $t\ge0$. When the coefficients are defined only for $u\in[0,1]$, as in~\eqref{eq:spde_canning} and~\eqref{eq: spde target}, $u$ is required to take values in $[0,1]$. We say that \emph{weak uniqueness} (uniqueness in law) holds if any two weak solutions with the same initial condition $f$, possibly defined on different probability spaces, have the same law on $\mathcal C(\R_+;\mathcal C^{\mathrm{tem}}_\R)$.

\subsection{Discrete heat kernel}\label{sec:discrete heat kernel}
Let us consider the (lazy) symmetric random walk $(S_n)_{n\ge 0}$ on $\Z$ with parameter $\alpha\in(0,1)$, defined by
\begin{align*}
	\P\bigl(S_{n+1}=x+e\mid S_n=x\bigr)=\frac{\alpha}{2},\qquad \text{for any }e \in \{-1,1\} ,
\end{align*}
and
\begin{align*}
	\P\bigl(S_{n+1}=x\mid S_n=x\bigr)=1-\alpha.
\end{align*}
For $n\ge 0$ and $x,y\in \Z$, we write its transition probabilities as
\begin{align}\label{eq:p_n unscaled}
	p_n(x,y):=\P\bigl(S_n=y\mid S_0=x\bigr).
\end{align}

Recall the discrete Laplacian $\Delta_d$ on $\Z$ from Section~\ref{sec:setup-main},
\begin{align*}
	(\Delta_d f)(x)=f(x+1)-2f(x)+f(x-1).
\end{align*}
Then the one-step transition operator $P$ is
\begin{align*}
	(P f)(x)=(1-\alpha)f(x)+\frac{\alpha}{2}\bigl(f(x+1)+f(x-1)\bigr)
	= f(x)+\frac{\alpha}{2}(\Delta_d f)(x),
\end{align*}
so the transition probability matrix is
\begin{align}\label{eq:discrete semi-group and laplacian}
	P = I+\frac{\alpha}{2}\Delta_d,
\end{align}
and $p_n(x,y)=(P^n)(x,y)$.

\begin{prop}[{\cite{lawler2010random}*{Propositions 2.4.2 and 2.4.4}}]\label{prop:lawler-242}
	Suppose $p_n$ is the $n$-step transition probability matrix of the above described random walk on $\Z$, then
	\begin{align}\label{ineq:discrete-heat-kernel-gradient-abs-difference}
		\sum_{z \in \Z} \abs{p_n(z) - p_n(z+y)} \le c\abs{y}n^{-1/2},
	\end{align}
	and
	\begin{align*}
		p_n(0) \le c n^{-1/2}.
	\end{align*}
\end{prop}

Some simple consequences are given as follows: we will denote the transition probability by $p_n(x,y)$ as $p_n(x-y)$ in the following.

\begin{corollary}\label{cor:heat kernel on diagonal bound}
	Uniformly in $n \in \N$,
	\[
		\sup_{x,z} p_n(x,z) \lesssim n^{-\frac{1}{2}}.
	\]
\end{corollary}
\begin{proof}[Proof]
	See Appendix~\ref{app:heat-kernel-proofs}.
\end{proof}

\begin{corollary}\label{cor:gradient l2 bound}
	Uniformly in $x,y \in \Z$ and $n \in \N$,
	\begin{align*}
		\sum_z \pr{p_n(x,z) - p_n(y,z)}^2 \lesssim \min\left\{n^{-1/2}\frac{\abs{x-y}}{n^{1/2}}, n^{-1/2}\right\}.
	\end{align*}
	Moreover, uniformly in $x\in\Z$ and $n,h\ge1$,
	\begin{align}\label{ineq:heat kernel temporal l1 difference}
		\sum_{z\in\Z}\abs{p_{n+h}(x,z)-p_n(x,z)}
		\lesssim\min\left\{1,\sqrt{\frac{h}{n}}\right\}.
	\end{align}
\end{corollary}
\begin{proof}[Proof]
	See Appendix~\ref{app:heat-kernel-proofs}.
\end{proof}
We will also need the following version of local central limit theorem. Here, we let
\begin{align*}
	\overline p_n(z) :=\frac{1}{\pr{2\pi n \alpha}^\frac{1}{2}} \exp \pr{- \frac{z^2}{2\alpha n}},
	\qquad z \in \Z,\, n \in \N.
\end{align*}

\begin{lemma}[{\cite{lawler2010random}*{Theorem 2.1.1}}]\label{lem:local clt}
	There exists $c>0$ and for every $k \ge 4$ there is a $c(k)< \infty$ such that for all integers $n >0$ and $x \in \Z$, the following holds with $z:= x/ \sqrt{n}$:
	\begin{equation}\label{ineq: local clt useful}
		\bigl|p_n(x)-\overline p_n(x)\bigr|
		\le  \frac{c(k)}{n^{3/2}}\pr{(|z|^k+1)\,e^{-\frac{z^2}{2\alpha} }+\frac{1}{n^{(k-3)/2}}}.
	\end{equation}
	\begin{equation}
		\bigl|p_n(x)-\overline p_n(x)\bigr|
		\le \frac{c}{n^{3/2}|z|^2} .
	\end{equation}
\end{lemma}

\begin{lemma}\label{lem: heat kernel temporal l2 difference bound}
	Let $p_n(x,y)$ be the discrete heat kernel on $\Z$. Then, uniformly in $n,k \in \N$,
	\begin{align*}
		\sum_y\pr{p_{n+k}(x,y) - p_{n}(x,y)}^2 \lesssim \frac{k}{n^{3/2}}.
	\end{align*}
\end{lemma}
\begin{proof}[Proof]
	See Appendix~\ref{app:heat-kernel-proofs}.
\end{proof}
\begin{lemma}\label{lem:weighted discrete semigroup stability}
	For $\lambda,T>0$ and $L\in\N$, define for  $f: \Z/L \to \R$ that
	\[
		\norm{f}_{(\lambda),L}
		:=\sup_{x_i\in\Z/L}e^{-\lambda|x_i|}\abs{f(x_i)}.
	\]
	For $k\in\Z_+$ and $x_i\in\Z/L$, let $P_k^Lf(x_i):=\sum_{j\in\Z}p_k(Lx_i,j)f(j/L)$.
	Then, for every $f$ with $\norm{f}_{(\lambda),L}<\infty$ and every $k\in\Z_+$ satisfying $k\le TL^2$,
	\[
		\norm{P_k^Lf}_{(\lambda),L}
		\le
		2\exp\left\{\frac{\alpha T\lambda^2}{2}\cosh\lambda\right\}
		\norm{f}_{(\lambda),L}.
	\]
\end{lemma}
\begin{proof}
	See Appendix~\ref{app:heat-kernel-proofs}.
\end{proof}

For the rest of this article, for each $L \in \N$, we let
\begin{align}\label{eq:q_scaled kernel}
	q^L_n(x_i,x_j) := p_n(x_iL,x_j L), \qquad \text{ for }n \in \Z_+, \, x_i,x_j \in \Z/L,
\end{align}
and observe for any $f: \Z/L \to \R$ and $x_i \in \Z/L$, we have 
\begin{align}\label{eq:discrete heat kernel and laplace}
	f(x_i)+\frac{\alpha}{2}\Delta_L f(x_i) = \sum_{x_\ell \in \Z/L}q_{1}^L(x_i,x_\ell)f(x_\ell),
\end{align}
where $\Delta_Lf(x):=f(x+\frac1L)-2f(x)+f(x-\frac1L)$. Then for $i \in \Z$ and $x_i = i / L$,
$$
	\Delta_dY_{\lfloor tL^2\rfloor}(i) = \Delta_L\xi_t(x_i)=  \xi_t(x_{i+1}) - 2 \xi_t(x_i)+ \xi_t(x_{i-1}).
$$
\section{Discrete evolution equations and moment bounds}\label{sec:discrete-evolution}
In this section, we will derive two evolution equations for $\xi^L$ that correspond to the mild and weak formulations of~\eqref{eq: spde target}. Then we will use the mild formulation for $\xi^L$ to derive the moment bound for its spatial-temporal increments, which is needed for tightness results.
\subsection{Discrete mild/weak formulations}
For each  $k \in \N$, $L \ge 1$ and $ x_i\in \Z/L$,
\begin{align}\label{eq:martingale decomposition one step}
	\xi_{k/L^2}^L(x_i) -\xi^L_{(k-1)/L^2}(x_i) = \mathbb E\left[ {\xi^L_{k/L^2}(x_i)}  - \xi^L_{(k-1)/L^2}(x_i) \mid \mathcal F_{k-1}\right]+\eta^L_k(x_i),
\end{align}
where
\begin{align*}
	\eta^L_k(x_i)  := \xi_{k/L^2}^L(x_i) - \E[\xi_{k/L^2}^L(x_i)\mid \mathcal F_{k-1}].
\end{align*}
Note that for each $x_i \in \Z/L$, $\sum_{j = 1}^k \eta^L_j(x_i)$ is a martingale in $k\in \N$ with respect to the filtration $(\mathcal F_{k})_{k \ge 1}$.

For $f:\Z/L\to[0,1]$, define
\begin{align}\label{eq:BL definition}
	B_L(f)(x_i):=b_L\pr{f(x_i)+\frac{\alpha}{2}\Delta_Lf(x_i)},
	\qquad x_i\in\Z/L.
\end{align}

We see for all $k,L \in \N$ and $x_i \in \Z/L$,
\begin{equation}
	\begin{aligned}
		&\qquad\mathbb E\left[ \xi_{k/L^2}^L(x_i)  - \xi_{(k-1)/L^2}^L(x_i) \mid \mathcal F_{k-1}\right]\\
		&\overset{\eqref{eq:abstract spatial frequency process}}{=}  \E\left[\mathcal R^L_{i,k-1}\pr{\xi^L_{(k-1)/L^2}(x_i)+\frac{\alpha}{2}\Delta_L \xi^L_{(k-1)/L^2}(x_i)} - \xi^L_{(k-1)/L^2}(x_i)\mid\mathcal F_{k-1} \right]\\
		&\overset{\eqref{eq:R assumption}}{=} \frac{\alpha}{2}\Delta_L \xi^L_{(k-1)/L^2}(x_i) + L^{-2}b_L\pr{\xi^L_{(k-1)/L^2}(x_i)+\frac{\alpha}{2}\Delta_L \xi^L_{(k-1)/L^2}(x_i)}\\
		&\overset{\eqref{eq:BL definition}}{=}\frac{\alpha}{2}\Delta_L \xi^L_{(k-1)/L^2}(x_i) + L^{-2}B_L\pr{\xi^L_{(k-1)/L^2}}(x_i).
	\end{aligned}
\end{equation}
Inserting this identity into~\eqref{eq:martingale decomposition one step} gives the exact discrete equation
\begin{equation}\label{eq:discrete SPDE}
	\xi_{k/L^2}^L(x_i) =\pr{1+\frac{\alpha}{2}\Delta_L}\xi^L_{(k-1)/L^2}(x_i) + L^{-2}B_L\pr{\xi^L_{(k-1)/L^2}}(x_i) + \eta^L_k(x_i).
\end{equation}
Iterate this to see for any $k \ge 1$ and $x_i \in \Z/L$ that
\begin{equation}\label{eq: mild formulation prelim}
	\begin{aligned}
		\xi^L_{k/L^2}(x_i) &= \sum_{\ell} q_{k}^L(x_i,x_\ell) \xi_0^L(x_\ell)+ \sum_{j = 0}^{k-1}\sum_{x_\ell \in \Z/L} q^L_{k - 1 - j}(x_i,x_\ell) L^{-2}  B_L(\xi^L_{j/L^2})(x_\ell)\\
		&\qquad+\sum_{j = 1}^{k} \sum_{x_\ell \in \Z/L}q_{k-j}^L\pr{x_i,x_\ell}\eta^L_j\pr{x_\ell}.
	\end{aligned}
\end{equation}
Indeed,~\eqref{eq: mild formulation prelim} holds for $k = 1$ because $q_0^L$ is the identity kernel. Assume by induction that~\eqref{eq: mild formulation prelim} holds for some $k \in \N$. Then
\begin{align*}
	\xi_{(k+1)/L^2}^L(x_i) &\overset{\eqref{eq:discrete SPDE}}{=}\pr{1+\frac{\alpha}{2}\Delta_L}\xi^L_{k/L^2}(x_i) + L^{-2}{ B_L\pr{\xi^L_{k/L^2}}(x_i)} + \eta^L_{k+1}(x_i)\\
	& \overset{\eqref{eq:discrete heat kernel and laplace}}{=}\sum_{x_r\in\Z/L} q^L_1\pr{x_i,x_r}\xi^L_{k/L^2}(x_r) + L^{-2}{ B_L\pr{\xi^L_{k/L^2}}(x_i)} + \eta^L_{k+1}(x_i)\\
	&\overset{\eqref{eq: mild formulation prelim}}{=}\sum_{\ell} q_{k+1}^L(x_i,x_\ell)\xi_0^{L}(x_\ell) + \sum_{j = 0}^{k-1} \sum_{\ell} q_{k-j}^L(x_i,x_\ell) L^{-2} { B_L(\xi_{j/L^2}^L)(x_\ell)}\\
	&\qquad + \sum_{j = 1}^k \sum_{\ell} q_{k+1-j}^L(x_i,x_\ell) \eta_j^L(x_\ell)\\
	&\qquad+ L^{-2}{ B_L\pr{\xi^L_{k/L^2}}(x_i)} + \eta^L_{k+1}(x_i)\\
	&=\sum_\ell q_{k+1}^L(x_i,x_\ell) \xi_0^L(x_\ell)+\sum_{j = 0}^k \sum_\ell q_{k-j}^L(x_i,x_\ell) L^{-2} { B_L(\xi_{j/L^2}^L)(x_\ell)}\\
	&\qquad + \sum_{j = 1}^{k+1} \sum_\ell q_{k+1-j}^L(x_i,x_\ell) \eta_{j}^L(x_\ell),
\end{align*}
which is exactly~\eqref{eq: mild formulation prelim} with $k$ replaced by $k+1$, where we used the semigroup property and the fact that $q_0^L$ is the identity kernel. Similarly, one can obtain the following weak formulation of the discrete equation:
\begin{align}\label{eq: weak formulation prelimit}
	\xi_{k/L^2}^L(x_i)& = \xi_0^L(x_i) +L^{-2}\sum_{j = 0}^{k-1} \frac{\alpha L^2}{2}\Delta_L \xi^L_{j/L^2}(x_i) + L^{-2} \sum_{j = 0}^{k-1}  B_L(\xi^L_{j/L^2})(x_i)+ \sum_{j = 1}^{k} \eta^L_{j}(x_i).
\end{align}

\noindent Recall the martingale term from~\eqref{eq:martingale decomposition one step}. For each $k \in \N$ and $i \in \Z$,
\begin{align*}
	\eta_k^L(x_i) &= \xi_{k/L^2}^L(x_i) - \E[\xi^L_{k/L^2}(x_i)|\mathcal F_{k-1}]\\
	& =  \xi_{k/L^2}^L(x_i) -\pr{ \xi_{\frac{k-1}{L^2}}^L(x_i) +\frac{\alpha}{2}\Delta_L  \xi_{\frac{k-1}{L^2}}^L(x_i)}\\
	&\qquad+ \E[\pr{\xi_{\frac{k-1}{L^2}}^L(x_i)+\frac{\alpha}{2}\Delta_L \xi_{\frac{k-1}{L^2}}^L(x_i)} - \xi_{k/L^2}^L(x_i)|\mathcal F_{k-1}]\\
	&\overset{\eqref{eq:abstract spatial frequency process}}{=} \Ru_{i,k-1}^L \pr{ \xi_{\frac{k-1}{L^2}}^L(x_i) +\frac{\alpha}{2}\Delta_L  \xi_{\frac{k-1}{L^2}}^L(x_i)} - \pr{ \xi_{\frac{k-1}{L^2}}^L(x_i) +\frac{\alpha}{2}\Delta_L  \xi_{\frac{k-1}{L^2}}^L(x_i)}\\
	&\qquad - \E \bracket{\Ru_{i,k-1}^L \pr{ \xi_{\frac{k-1}{L^2}}^L(x_i) +\frac{\alpha}{2}\Delta_L  \xi_{\frac{k-1}{L^2}}^L(x_i)} - \pr{ \xi_{\frac{k-1}{L^2}}^L(x_i) +\frac{\alpha}{2}\Delta_L  \xi_{\frac{k-1}{L^2}}^L(x_i)}\mid \F_{k-1}}.
\end{align*}
Therefore, for $p>8$ as in~\eqref{ineq:p_moment bound}, it holds uniformly in $k \in \N$, $i \in \Z$ and all sufficiently large $L$ that
\begin{align}\label{ineq:mtg_increment bounds}
	\E\bracket{\abs{\eta^L_{k}(x_i)}^p\mid \F_{k-1}}
	&\lesssim \E_x\bracket{\abs{\Ru^L(x)-x}^p\mid x=\xi_{\frac{k-1}{L^2}}^L(x_i)+\frac{\alpha}{2}\Delta_L\xi_{\frac{k-1}{L^2}}^L(x_i)}\overset{\eqref{ineq:p_moment bound}}{\lesssim} L^{-1-\kappa_p}.
\end{align}
Also, for all $k,L \in \N$, $i \in \Z$, it holds that
\begin{equation}\label{eq:mtg_increment l2}
	\begin{aligned}
		\E\bracket{\eta_{k}^L(x_i)^2\mid \mathcal F_{k-1}} &= \E\bracket{\pr{\Ru_{i,k-1}^L(x) - x}^2|x = Y^L_{k-1}(i) +\frac{\alpha}{2}\Delta_d Y^L_{k-1}(i)}\\
		&\quad\quad- \pr{\E\bracket{\Ru_{i,k-1}^L(x) - x|x = Y^L_{k-1}(i) +\frac{\alpha}{2}\Delta_d Y^L_{k-1}(i)}}^2\\
		&\overset{\eqref{eq:R assumption}}{=}\frac{1}{L}\sigma_L\pr{Y^L_{k-1}(i) +\frac{\alpha}{2}\Delta_d Y^L_{k-1}(i)}
		-\frac{1}{L^4}b_L\pr{Y^L_{k-1}(i) +\frac{\alpha}{2}\Delta_d Y^L_{k-1}(i)}^2\\
		&=\frac{1}{L}\sigma_L\pr{\xi_{\frac{k-1}{L^2}}(x_i) +\frac{\alpha}{2}\Delta_L \xi_{\frac{k-1}{L^2}}(x_i)}
		-\frac{1}{L^4}b_L\pr{\xi_{\frac{k-1}{L^2}}(x_i) +\frac{\alpha}{2}\Delta_L \xi_{\frac{k-1}{L^2}}(x_i)}^2,
	\end{aligned}
\end{equation}
which is bounded above by $L^{-1}$ up to some multiplicative constant uniformly in $k \in \N$, $i \in \Z$ and sufficiently large $L \in \N$ by~\eqref{eq:coefficient convergence}.

\subsection{Spatial-temporal moment bounds}
By~\eqref{eq: mild formulation prelim} we see for each $x_i \in \Z/L$ and $t\ge 0$,
\begin{align*}
	\xi_{t}^L(x_i) &= \sum_{j} q^L_{\lfloor tL^2\rfloor}(x_i,x_j) \xi_0(x_j) + L^{-2} \sum_{k = 0}^{\lfloor tL^2\rfloor -1} \sum_{j}q^L_{\lfloor tL^2\rfloor - 1 -k }(x_i,x_j) { B_L\pr{\xi_{k/L^2}^L}(x_j)}\\
	&\qquad \qquad  +\sum_{k = 0}^{\lfloor tL^2\rfloor - 1} \sum_j q^L_{\lfloor tL^2\rfloor - 1 -k}(x_i,x_j)\eta^L_{k+1}(x_j) \\
	&:= I_1^L(t,x_i) +I_2^L(t,x_i)+I_3^{L}(t,x_i).
\end{align*}

We now would like to consider moment bounds for spatial increments.

\begin{prop}\label{prop:spatial_increment_moment_bound}
	For any $T>0$, the following inequality holds uniformly for sufficiently large $L \ge 1$, $0<t\le T$ and $x_i,x_j \in \Z/L$ with $|x_i - x_j| \le 1$,
	\begin{align*}
		\E\left[ \abs{I_2^L(t,x_i)-I_2^{L}(t,x_j)}^p\right] +  \E\left[ \abs{I_3^L(t,x_i)-I_3^{L}(t,x_j)}^p\right]\le C \abs{x_i - x_j}^{\frac{p}{4}\wedge (1+\kappa_p)}.
	\end{align*}
\end{prop}

\begin{proof}
	{ By~\eqref{eq:coefficient convergence}, there exist $L_0>1$ and $C>0$ such that, for all $L\ge L_0$, $t\ge0$, and $x_i\in\Z/L$,
	\begin{align*}
		\abs{B_L\pr{\xi_t^L}(x_i)}\le C.
	\end{align*}}
	If $t<L^{-2}$, then $\lfloor tL^2\rfloor=0$ and $I^L_2(t,\cdot)=I^L_3(t,\cdot)=0$; the claim is also trivial if $x_i=x_j$. Hence we assume $t\ge L^{-2}$ and $x_i\neq x_j$, so that $|x_i-x_j|\ge L^{-1}$. In the sums over $k$ below, the term $k=\lfloor tL^2\rfloor-1$ has kernel $q^L_0$, the identity kernel.
	Hence there is some $C>0$ so that for all $L\ge L_0$, $i,j \in \Z$ and $t>0$
	\begin{align*}
		\abs{I^L_2(t,x_i) - I^L_2(t,x_j)}^p  &\le \pr{L^{-2}\sum_{k = 0}^{\lfloor tL^2\rfloor-1}\sum_\ell\abs{q^L_{\lfloor tL^2\rfloor - 1 -k}(x_i,x_\ell) - q^L_{\lfloor tL^2\rfloor - 1 -k}(x_j,x_\ell)}}^p\\
		&\overset{\eqref{ineq:discrete-heat-kernel-gradient-abs-difference}}{\le} C\pr{2L^{-2}+L^{-2} \sum_{k =0}^{\lfloor tL^2\rfloor - 2} \frac{\abs{x_i - x_j}}{\sqrt{(\lfloor tL^2\rfloor-1-k)/L^2}}}^p\\
		&\le C\pr{L^{-2}+\abs{x_i - x_j}\int_0^t \frac{1}{\sqrt{t-s}}\mathrm ds}^p\\
		&\le C |x_i -x_j|^p.
	\end{align*}
	Here $2L^{-2}$ is from when $k=\lfloor tL^2\rfloor-1$, for which $\sum_\ell|q^L_0(x_i,x_\ell)-q^L_0(x_j,x_\ell)|=2$, and in the last line we used $L^{-2}\le L^{-1}\le|x_i-x_j|$.

	On the other hand, by Rosenthal's inequality (cf.\ \cite{hall2014martingale}*{Theorem 2.12}), we see it holds uniformly in $t>0$ and $x_i,x_j \in \Z/L$ that
	\begin{align*}
		&\E\left[ \abs{I_3^L(t,x_i) - I_3^L(t,x_j)}^p\right] \\\lesssim  &\mathbb E\left[\pr{\sum_{k=0}^{\lfloor tL^2\rfloor - 1} \E \left[\pr{ \sum_{\ell}\pr{q^L_{\lfloor tL^2\rfloor - 1 -k}(x_i,x_\ell) -  q^L_{\lfloor tL^2\rfloor - 1 -k}(x_j,x_\ell)} \eta^L_{k+1}(x_\ell)}^2 \mid\mathcal F_{k}\right] }^{p/2}\right]\\
		&\quad+ \sum_{k = 0}^{\lfloor tL^2\rfloor-1} \E\left[\abs{\sum_{\ell} \pr{q^L_{\lfloor tL^2\rfloor - 1 -k}(x_i,x_\ell) -  q^L_{\lfloor tL^2\rfloor - 1 -k}(x_j,x_\ell)} \eta^L_{k+1}(x_\ell)}^p\right]\\
		&=:J_1(t,x_i,x_j)+J_2(t,x_i,x_j).
	\end{align*}
	For fixed $x_i,x_j$, define
	\begin{align*}
		\delta_m^{i,j}(x_\ell):= q^L_{m}(x_i,x_\ell) - q^L_m(x_j,x_\ell),
	\end{align*}

	We note that, by independence of $\{\Ru_{i,g}^L\}_{i,g}$, the sequence $\{\eta^L_{g+1}(x_\ell)\}_{\ell \in \Z}$ is independent, mean zero, conditioned on $\mathcal F_g$ for each $g\in \N$. Hence for all $\ell \neq \ell' \in \Z$,
	\begin{align*}
		\E\left[\eta^L_{k+1}(x_\ell)\eta^L_{k+1}(x_{\ell'})\mid \mathcal F_k\right]
		&=\E\left[\eta^L_{k+1}(x_\ell)\mid \mathcal F_k\right]\E\left[\eta^L_{k+1}(x_{\ell'})\mid \mathcal F_k\right]=0.
	\end{align*}
	Hence, by~\eqref{eq:mtg_increment l2} and Corollary~\ref{cor:gradient l2 bound}, for $m \ge 1$, 
	\begin{align*}
		\E\left[\pr{\sum_\ell \delta_m^{i,j}(x_\ell) \eta^L_{k+1}(x_\ell)}^2 |\mathcal F_{k}\right] &= \sum_{\ell, \ell'}\E\left[ \delta_m^{i,j}(x_\ell) \delta_m^{i,j}(x_{\ell'}) \eta^L_{k+1}(x_\ell) \eta^L_{k+1}(x_{\ell'})|\mathcal F_{k}\right]\\
		&=\sum_\ell\delta_m^{i,j}(x_\ell)^2 \E\left[ \eta^L_{k+1}(x_\ell)^2 |\mathcal F_{k}\right]\\
		&{\overset{\eqref{eq:mtg_increment l2},\,\eqref{eq:coefficient convergence}}{\lesssim} L^{-1}\sum_\ell \delta_m^{i,j}(x_\ell)^2} \\
		&{\overset{\text{Cor.~\ref{cor:gradient l2 bound}}}{\lesssim} L^{-1} \min\left\{\frac{L\abs{x_i - x_j}}{m}, m^{-1/2} \right\}}.
	\end{align*}
	Note that $\delta_0^{i,j} = \mathbf{1}_{\{x_i\}} - \mathbf{1}_{\{x_j\}}$, and since $L^{-1} \le \abs{x_i -x_j}$ for any $x_i \neq x_j$, we see
	\begin{align*}
		{ L^{-1}\sum_{k = 0}^{\lfloor tL^2\rfloor -1} \sum_{\ell}\delta_{\lfloor tL^2\rfloor -1-k}^{i,j}(x_\ell)^2 \le 2\abs{x_i -x_j}+L^{-1}\sum_{k = 1}^{\lfloor tL^2\rfloor} \sum_\ell \delta_k^{i,j}(x_\ell)^2.}
	\end{align*}
	Consider now for $t\ge L^{-2}$ that
	\begin{align*}
		J_1(t,x_i,x_j) &{\lesssim} \abs{x_i - x_j}^{\frac p2}+\pr{{ L^{-1}}\pr{\sum_{k = 1}^{\lfloor \pr{t\wedge \abs{x_i - x_j}}L^2\rfloor} +\sum^{\lfloor tL^2\rfloor}_{\lfloor \pr{t\wedge \abs{x_i - x_j}}L^2\rfloor}}\sum_\ell \delta_k^{i,j}(x_\ell)^2}^\frac{p}{2}\\
		&{\lesssim} \abs{x_i - x_j}^{\frac p2}+\pr{{ L^{-1}}\sum_{k = 1}^{\lfloor \pr{t\wedge \abs{x_i - x_j}}L^2\rfloor} k^{-1/2}+\sum^{\lfloor tL^2\rfloor}_{\lfloor \pr{t\wedge \abs{x_i - x_j}}L^2\rfloor} \frac{\abs{x_i - x_j}}{k}}^\frac{p}{2}\\
		&{\lesssim} \abs{x_i - x_j}^{\frac p2}+\pr{{ L^{-1}} \int_0^{\ \pr{t\wedge \abs{x_i - x_j}}L^2} s^{-1/2} \mathrm ds + \sum^{\lfloor tL^2\rfloor}_{\lfloor \pr{t\wedge \abs{x_i - x_j}}L^2\rfloor} \frac{\abs{x_i - x_j}}{k} } ^\frac{p}{2}\\
		&\lesssim  \abs{x_i - x_j}^{\frac p2}+\pr{\pr{t\wedge \abs{x_i - x_j}}^\frac{1}{2}+\abs{x_i - x_j} \ln \pr{\frac{t}{t\wedge |x_i - x_j|}}}^{\frac{p}{2}}\\
		&\lesssim \abs{x_i - x_j}^\frac{p}{4},
	\end{align*}
	where we used the fact that $|x_i - x_j| \le 1$ in the last line.

	For $J_2$ term, by conditional independence and Rosenthal's inequality, we see, for $m\ge1$,
	\begin{align*}
		\E\left[\abs{\sum_\ell \delta_m^{i,j}(x_\ell) \eta^L_{k+1}(x_\ell) } ^p\bigg | \mathcal F_{k}\right]&\lesssim \pr{\sum_\ell  \delta_m^{i,j}(x_\ell)^2 \E \left[\eta^L_{k+1}(x_\ell)^2\bigg| \mathcal F_{k}\right]}^\frac{p}{2} +\sum_\ell |\delta_m^{i,j}(x_\ell)|^p \E\left[ |\eta^L_{k+1}(x_\ell)|^p \bigg|\mathcal F_{k}\right]\\
		&{\overset{\eqref{eq:mtg_increment l2},\,\eqref{ineq:p_moment bound}}{\lesssim} \pr{\sum_\ell \delta_m^{i,j}(x_\ell)^2 L^{-1}}^\frac{p}{2} + L^{-1-\kappa_p} \sum_\ell |\delta_m^{i,j}(x_\ell)|^p} \\
		&{\lesssim} \pr{\min\left\{ \frac{|x_i - x_j|}{m}, L^{-1} m^{-1/2}\right\}}^\frac{p}{2}+ m^{-\frac{p-2}{2}} { L^{-1-\kappa_p}} \sum_\ell \delta_m^{i,j}(x_\ell)^2\\
		&{\lesssim} \frac{|x_i-x_j|^\frac{p}{2}}{m^{p/2}}+ m^{-\frac{p-2}{2}}{ L^{-\kappa_p}}\min\left\{\frac{\abs{x_i-x_j}}{m}, { L^{-1}} m^{-1/2} \right\}\\
		&{\lesssim} m^{-\frac{p}{2}} \pr{\abs{x_i - x_j}^\frac{p}{2} + { L^{-\kappa_p}}\abs{x_i - x_j}}\\
		&\lesssim m^{-\frac{p}{2}}\pr{\abs{x_i - x_j}^\frac{p}{2} +\abs{x_i - x_j}^{1+\kappa_p}}\\
		&\lesssim m^{-\frac{p}{2}}\abs{x_i - x_j }^{\frac{p}{2}\wedge \pr{1+\kappa_p}}
	\end{align*}
	where we used Corollary~\ref{cor:gradient l2 bound} and Corollary~\ref{cor:heat kernel on diagonal bound} in the third line, and the fact that for $x_i \neq x_j$, $|x_i - x_j| \ge L^{-1}$ in the second to the last line, and $\abs{x_i - x_j} \le 1$ in the last line. For $m=0$, then the above equals $\E\bracket{\abs{\eta^L_{k+1}(x_i)-\eta^L_{k+1}(x_j)}^p}\lesssim L^{-1-\kappa_p}\le\abs{x_i-x_j}^{1+\kappa_p}$ by~\eqref{ineq:mtg_increment bounds}. Therefore, for any $p >2$,
	\begin{align*}
		J_2(t,x_i,x_j) &\lesssim \abs{x_i - x_j}^{1+\kappa_p}+\abs{x_i - x_j}^{\frac{p}{2}\wedge\pr{1+\kappa_p}}\sum_{m = 1}^{\infty} m^{-p/2} \lesssim \abs{x_i - x_j}^{\frac{p}{2}\wedge\pr{1+\kappa_p}}.
	\end{align*}
	Hence, it holds uniformly in $L \in \N$, $\abs{x_i - x_j}\le 1$ and $t \in [0,T]$ that
	\begin{align*}
		\E \bracket{\abs{I_3(t,x_i) - I_3(t,x_j)}^p} \lesssim J_1(t,x_i,x_j) +J_2(t,x_i,x_j) \lesssim \abs{x_i - x_j}^{\frac{p}{4}\wedge \pr{1+\kappa_p}}.
	\end{align*}
	Combined with the moment bounds for $I_2$ term, we see the desired result holds.
\end{proof}

Now we consider the moment bound for temporal increment.
\begin{prop}\label{prop:temporal_increment_moment_bound}
	For every $T>0$, let $p>8$ and $\kappa_p>3$ be as in~\eqref{ineq:p_moment bound}. Uniformly over all sufficiently large $L$, $x_i\in\Z/L$, and $s,t\in L^{-2}\Z_+$ satisfying $0\le s<t\le T$,
	\begin{align*}
		\E\bracket{\abs{I_2^L(t,x_i)-I_2^L(s,x_i)}^p}
		+\E\bracket{\abs{I_3^L(t,x_i)-I_3^L(s,x_i)}^p}
		\lesssim_T (t-s)^{\frac{p}{4}\wedge\frac{1+\kappa_p}{2}}.
	\end{align*}
\end{prop}
\begin{proof}
	Let us fix an arbitrary $T>0$ and consider first $I_2^L$. Set $h:=(t-s)L^2\ge1$. By the uniform boundedness of $\abs{B_L(\xi^L)}$, we see it holds uniformly in $0\le s,t \le T$ and $L \in \N$ that
	\begin{align*}
		\abs{I_2^L(t,x_i)-I_2^L(s,x_i)}
		&\le L^{-2}\sum_{k=sL^2}^{tL^2-1}\sum_{x_\ell\in\Z/L}
		q_{tL^2-1-k}^L(x_i,x_\ell)
		\abs{B_L(\xi_{k/L^2}^L)(x_\ell)}\\
		&\qquad+L^{-2}\sum_{k=0}^{sL^2-1}\sum_{x_\ell\in\Z/L}
		\abs{q_{tL^2-1-k}^L(x_i,x_\ell)-q_{sL^2-1-k}^L(x_i,x_\ell)}
		\abs{B_L(\xi_{k/L^2}^L)(x_\ell)}\\
		&\lesssim L^{-2}\left(h+
		\sum_{k=0}^{sL^2-1}\sum_{x_\ell\in\Z/L}
		\abs{q_{tL^2-1-k}^L(x_i,x_\ell)-q_{sL^2-1-k}^L(x_i,x_\ell)}\right)\\
		&=L^{-2}\left(h+
		\sum_{n=0}^{sL^2-1}\sum_{x_\ell\in\Z/L}
		\abs{q_{n+h}^L(x_i,x_\ell)-q_n^L(x_i,x_\ell)}\right)\\
		&{\overset{\eqref{ineq:heat kernel temporal l1 difference}}{\lesssim}}
		L^{-2}\left(h+1+
		\sum_{n=1}^{sL^2-1}\min\left\{1,\sqrt{\frac{h}{n}}\right\}\right)\\
		&\lesssim L^{-2}\pr{h+1+\sqrt{hsL^2}}\\
		&\lesssim (t-s)+\sqrt{s(t-s)}\\
		&\lesssim\sqrt{t-s},
	\end{align*}
	where we used the fact that $h\ge1$ in the sixth line. Hence, it holds uniformly in $L \in \N$, $0 \le s,t \le T$ and $x_i \in \Z/L$ that
	\begin{align*}
		\E\bracket{\abs{I_2^L(t,x_i)-I_2^L(s,x_i)}^p}
		\lesssim\abs{t-s}^{p/2}
		\lesssim\abs{t-s}^{\frac p4\wedge\frac{1+\kappa_p}{2}}.
	\end{align*}

	We now prove the estimate for $I_3^L$. Note that
	\begin{align}\label{ineq:temporal difference moment martingale}
		\norm{I_3^L(t,x_i)-I_3^L(s,x_i)}_{L^p_\Omega}^p&\lesssim \norm{\sum_{k = sL^2}^{tL^2 - 1} \sum_\ell q^L_{tL^2-1-k}(x_i,x_\ell) \eta^L_{k+1}(x_\ell)}_{L^p_\Omega}^p\nonumber\\
		&\qquad \qquad + \norm{\sum_{k = 0}^{sL^2 - 1}\sum_\ell \pr{q^L_{tL^2-1-k}(x_i,x_\ell)- q^L_{sL^2-1-k}(x_i,x_\ell)}\eta^L_{k+1}(x_\ell)}_{L^p_\Omega}^p\nonumber\\
		&:= J_1+J_2.
	\end{align}
	By Rosenthal's inequality, we see it holds uniformly in $L \in \N$, $0 \le s,t\le T$ and $x_i \in \Z/L$ that
	\begin{align*}
		J_1&\lesssim \norm{\sum_{k = sL^2}^{tL^2 - 1} \E \bracket{\pr{\sum_\ell q^L_{tL^2 -1-k}(x_i,x_\ell) \eta^L_{k+1}(x_\ell)}^2\bigg|\mathcal F_{k}}}_{L^{p/2}_\Omega}^{p/2}\\
		&\qquad + \sum_{k = sL^2}^{tL^2 -1} \norm{\sum_\ell q^L_{tL^2 -1-k}(x_i,x_\ell) \eta^L_{k+1}(x_\ell) }_{L^p_\Omega}^p\\
		&:= J_1^1+J_1^2.
	\end{align*}
	By conditional independence and the fact that $h \ge 1$, we see
	\begin{align*}
		J_1^1&=\norm{\sum_{k = sL^2}^{tL^2-1} \sum_\ell q^L_{tL^2-1-k}(x_i,x_\ell)^2\E\bracket{\eta^L_{k+1}(x_\ell)^2\bigg |\mathcal F_{k}}}_{L^\frac{p}{2}_\Omega}^\frac{p}{2}\\
		&{\lesssim} \pr{{ L^{-1}}\sum_{k = sL^2}^{tL^2-1} q^L_{2tL^2-2-2k}(x_i,x_i)}^\frac{p}{2}\\
		&{\overset{\mathrm{Cor.}\,\ref{cor:heat kernel on diagonal bound}}{\lesssim}}\pr{L^{-1}\pr{1+\sum_{r=1}^{h-1}r^{-1/2}}}^\frac{p}{2}\\
		&\lesssim\pr{L^{-1}\sqrt h}^\frac{p}{2}\\
		&=\abs{t-s}^\frac{p}{4}.
	\end{align*}
	Also,
	\begin{align*}
		J_1^2&\lesssim \sum_{k = sL^2}^{tL^2 -1} \E\left[\pr{\sum_\ell q^L_{tL^2 -1-k}(x_i,x_\ell)^2 \E[\eta^L_{k+1}(x_\ell)^2\mid\mathcal F_k]} ^\frac{p}{2}\right]+\sum_{k = sL^2}^{tL^2 -1} \sum_\ell  q^L_{tL^2 -1-k}(x_i,x_\ell)^p \norm{\eta^L_{k+1}(x_\ell)}_{L^p_\Omega}^p.
	\end{align*}
	First term on the right hand side is bounded by
	\begin{align*}
		\sum_{k = sL^2}^{tL^2 -1} \pr{\sum_\ell q^L_{tL^2 -1-k}(x_i,x_\ell)^2{ L^{-1}} }^\frac{p}{2}&{\overset{\mathrm{Cor.}\,\ref{cor:heat kernel on diagonal bound}}{\lesssim}} { L^{-\frac{p}{2}}}\sum_{k = sL^2}^{tL^2 -2} \pr{tL^2-1-k}^{-p/4} + { L^{-p/2}}\\
		&{\lesssim} L^{-\frac{p}{2}}\pr{1+ \sum_{k = 1}^\infty k^{-\frac{p}{4}}}\\
		&\lesssim L^{-\frac{p}{2}} \le (t-s)^\frac{p}{4},
	\end{align*}
	where we used the fact that  $t-s\ge L^{-2}$ and $p >8$ in the last line. Consider now the second term on the right hand side for $J_1^2$,
	\begin{align*}
		\sum_{k = sL^2}^{tL^2 -1} \sum_\ell  q^L_{tL^2 -1-k}(x_i,x_\ell)^p \norm{\eta^L_{k+1}(x_\ell)}_{L^p_\Omega}^p&{\overset{\eqref{ineq:p_moment bound}}{\lesssim}}{ L^{-1-\kappa_p}\sum_{k=sL^2}^{tL^2-1}\sum_\ell q^L_{tL^2-1-k}(x_i,x_\ell)^p}\\
		&\le L^{-1-\kappa_p} \sum_{k = sL^2}^{tL^2-1} \pr{\sup_\ell q^L_{tL^2-1-k}(x_i,x_\ell)}^{p-1}\sum_\ell q^L_{tL^2-1-k}(x_i,x_\ell)\\
		&{\overset{\mathrm{Cor.}\,\ref{cor:heat kernel on diagonal bound}}{\lesssim}} { L^{-1-\kappa_p}}\pr{1+\sum_{k=1}^{(t-s)L^2-1}k^{-\frac{p-1}{2}}}\\
		&{\lesssim} { L^{-1-\kappa_p}}\\
		&{\le} {(t-s)^{\frac{1+\kappa_p}{2}}},
	\end{align*}
	where we used the fact that $p>8$ in the fourth inequality and  $t-s\ge L^{-2}$ in the last inequality. Combining the previous estimates to see
	\begin{align}\label{ineq:temporal J1}
		J_1 \lesssim J_1^1+J_1^2 \lesssim \abs{t-s}^{\frac{p}{4}\wedge \frac{1+\kappa_p}{2}}.
	\end{align}

	For $J_2$, let us denote for $0\le s,t \le T$, $i \in \Z$, $k \in \N$ and $x_\ell \in \Z/L$ that
	\begin{align*}
		\delta_k^{s,t;i}(x_\ell)
		:=\abs{q^L_{tL^2-k-1}(x_i,x_\ell)-q^L_{sL^2-k-1}(x_i,x_\ell)}.
	\end{align*}
	Similar to the $J_1$ term, we may use Rosenthal's inequality in $k$, followed by conditional Rosenthal's inequality in $x_\ell$, to see that
	\begin{align*}
		J_2 \lesssim J_2^1+J_2^2+J_2^3,
	\end{align*}
	where
	\begin{align*}
		J_2^1&:=\E\bracket{\pr{\sum_{k=0}^{sL^2-1}\sum_\ell
		\delta_k^{s,t;i}(x_\ell)^2
		\E\bracket{\eta_{k+1}^L(x_\ell)^2\mid\mathcal F_k}}^{p/2}},\\
		J_2^2&:=\sum_{k=0}^{sL^2-1}\E\bracket{\pr{\sum_\ell
		\delta_k^{s,t;i}(x_\ell)^2
		\E\bracket{\eta_{k+1}^L(x_\ell)^2\mid\mathcal F_k}}^{p/2}},\\
		J_2^3&:=\sum_{k=0}^{sL^2-1}\sum_\ell
		\E\bracket{\delta_k^{s,t;i}(x_\ell)^p
		\abs{\eta_{k+1}^L(x_\ell)}^p}.
	\end{align*}

	Recall $h=(t-s)L^2$. By Corollary~\ref{cor:heat kernel on diagonal bound} and Lemma~\ref{lem: heat kernel temporal l2 difference bound}, it holds uniformly in $h,n \in \N$ that
	\begin{align}\label{ineq:heat kernel square time difference}
		\sum_\ell\abs{q^L_{n+h}(x_i,x_\ell)-q^L_n(x_i,x_\ell)}^2
		\lesssim\min\left\{n^{-1/2},hn^{-3/2}\right\},
	\end{align}
	Hence, it holds uniformly in $L \in \N$, $0\le s,t\le T$ and  $x_i \in \Z/L$ that
	\begin{align*}
		J_2^1 &{\overset{\eqref{eq:coefficient convergence}}{\lesssim}} \pr{L^{-1}\sum_{k = 0}^{sL^2 - 1} \sum_\ell \delta_k^{s,t;i}(x_\ell)^2}^\frac{p}{2}\\
		&{\overset{\eqref{ineq:heat kernel square time difference}}{\lesssim}}\pr{L^{-1}\pr{1+\sum_{n=1}^{h}n^{-1/2}
		+h\sum_{n=h+1}^{\infty}n^{-3/2}}}^{p/2}\\
		&\lesssim\pr{L^{-1}\pr{1+\sqrt h}}^{p/2}\\
		&\lesssim\pr{L^{-1}\sqrt h}^{p/2}=(t-s)^{p/4}.
	\end{align*}
	where we used the fact that $h \ge 1$ in the last inequality. Similarly, since $p>8$,
	\begin{align*}
		J_2^2
		&\lesssim L^{-p/2}\left(1+\sum_{n=1}^{\infty}
		\min\left\{n^{-1/2},hn^{-3/2}\right\}^{p/2}\right)\le L^{-p/2}\left(1+\sum_{n=1}^{\infty}n^{-p/4}\right)
		\lesssim (L^{-2})^{p/4}
		\le(t-s)^{p/4},
	\end{align*}
	where the last inequality uses $t-s\ge L^{-2}$.

	By Corollary~\ref{cor:heat kernel on diagonal bound}, it holds uniformly in  $L,n,h \in \N$ and $x_i \in \Z/L$ that
	\begin{align*}
		\sum_\ell\abs{q^L_{n+h}(x_i,x_\ell)-q^L_n(x_i,x_\ell)}^p
		&\le\sup_\ell\pr{q^L_{n+h}(x_i,x_\ell)+q^L_n(x_i,x_\ell)}^{p-1}\sum_\ell\pr{q^L_{n+h}(x_i,x_\ell)+q^L_n(x_i,x_\ell)}\\
		&\lesssim n^{-(p-1)/2}.
	\end{align*}
	Thus, using the fact that $\abs{t-s} \ge L^{-2}$ when $t \neq s$, we see it holds uniformly in $L \in \N$, $0\le s,t \le T$ and $x_i \in \Z/L$ that
	\begin{align*}
		J_2^3
		&{\overset{\eqref{ineq:mtg_increment bounds}}{\lesssim}}L^{-1-\kappa_p} \sum_{k = 0}^{sL^2 - 1} \sum_\ell\abs{\delta^{s,t;i}_k(x_\ell)}^p\\
		&\lesssim L^{-1-\kappa_p} \pr{1+\sum_{k = 1}^\infty k^{-\frac{p-1}{2}}}\\
		&\lesssim L^{-1-\kappa_p}
		\le(t-s)^{(1+\kappa_p)/2}.
	\end{align*}
	Collecting the above estimates to see it holds uniformly in $L \in \N$, $0\le t,s \le T$  and $x_i \in \Z/L$ that
	\begin{align*}
		J_2 \lesssim \abs{t-s}^{\frac{p}{4}\wedge\frac{1+\kappa_p}{2}}.
	\end{align*}
	Inserting this and~\eqref{ineq:temporal J1} into~\eqref{ineq:temporal difference moment martingale} to see
	\begin{align*}
		\E\bracket{\abs{I_3^L(t,x_i)-I_3^L(s,x_i)}^p}
		&\lesssim\abs{t-s}^{\frac p4\wedge\frac{1+\kappa_p}{2}},
	\end{align*}
	which is the desired result.
\end{proof}

\section{\texorpdfstring{Tightness in $\mathcal D(\R_+; \mathcal C^{\mathrm{tem}}_\R)$}{Tightness in D(R+; Ctem(R))}}\label{sec:tightness}
In this section, we will show the sequence of spatially interpolated processes $\{\widetilde \xi^L\}_{L}$ is tight in $\mathcal D(\R_+;\mathcal C^{\mathrm{tem}}_\R)$ via bounds of the moments of the  spatial-temporal increments that are obtained in Propositions~\ref{prop:spatial_increment_moment_bound} and~\ref{prop:temporal_increment_moment_bound}.

For $f \in \mathcal C_\R$, define $\pi_L^\text{disc} :\mathcal C(\R)\to \mathcal C_{\Z/L}$ by $\pi_L^\text{disc}f(x_i) = f(x_i)$ for all $x_i \in \Z/L$. Define $\pi^\text{int}_L: \mathcal C_{\Z/L} \to \mathcal C_\R$ by linear interpolation. Then there exists a symmetric, non-negative function $\rho \in \mathcal C_c(\R)$  with $\int_\R \rho(x) \mathrm dx = 1$, so that for any $u_L : \Z /L \to \R$,
\begin{align}\label{eq:interpolation convolution identity}
	\widetilde u_L(x):= \frac{1}{L}\sum_{x_i \in \Z /L} u_L(x_i) \rho_L(x-x_i) = \pi_L^\text{int}u_L(x),\qquad \text{for each }x \in \R,
\end{align}
where $\rho_L(x):= L\rho(Lx)$ for each $L \in \N$ and $x \in \R$.

In this section, we denote for any function $f:\Z/L \to \R$ by
\begin{align*}
	P^L_{k}f(x_i) := \sum_{x_\ell \in \Z/L} q^L_k(x_i,x_\ell) f(x_\ell),\qquad k \in \Z_+,\, x_i \in \Z/L,
\end{align*}
whenever the right hand side is summable. Also, denote for any measurable $f: \R \to \R$ by
\begin{align*}
	P_tf(x):= \int_\R G_t(x-y) f(y) \mathrm dy,\qquad (t,x) \in \R_+\times \R,
\end{align*}
whenever the right hand side makes sense.

For  interval $K\subset \R$, we denote for any continuous function $f:\R \to \R$ and $\gamma \in (0,1]$ the H\"older norms
\begin{align*}
	\norm{f}_{\mathcal C_K^{\gamma}}:= \norm{f}_{\mathcal C^0_K} + \sup_{x,y \in K\atop x\neq y} \frac{\abs{f(x) - f(y)}}{\abs{x-y}^{\gamma}}.
\end{align*}
Lastly, we let $\omega_f(\delta):= \sup_{|x-y| \le \delta} |f(x) - f(y)|$ for any continuous function $f$ on $\R$. For $\lambda>0$, $\delta\ge0$, and $f\in\mathcal C^\mathrm{tem}_\R$, we also write
\begin{align}\label{eq:weighted_modulus_of_continuity}
	\omega_{f,\lambda}(\delta)
	:=\sup_{|x-y|\le\delta}e^{-\lambda|x|}\abs{f(x)-f(y)}.
\end{align}
We will need the following lemma. Recall $\norm{\cdot}_{(\lambda),L}$ for $\lambda >0$ and $L \in \N$ as in Lemma~\ref{lem:weighted discrete semigroup stability}.
\begin{lemma}\label{lem:initial convergence}
	Suppose $u_0 \in \mathcal C_b(\R)$ and  $u_L:\Z/L\to\R$ for each $L \in \N$ is bounded and satisfies, for every $\lambda>0$,
	\begin{align}\label{ineq:initial discretized convergence weighted}
		\norm{u_L-\pi_L^\text{disc}u_0}_{(\lambda),L}\longrightarrow 0,\qquad \text{ as }L \to \infty.
	\end{align}
	For $t\ge0$, set $A_L(t):=\pi_L^\text{int}P^L_{\lfloor tL^2\rfloor}u_L$. Then for every $T,\lambda>0$,
	\begin{align}\label{ineq:initial semigroup convergence}
		\lim_{L \to \infty}\sup_{t\in[0,T]}\norm{A_L(t)-P_tu_0}_{(\lambda)}=0.
	\end{align}
	Let $\widehat A_L$ be the temporal interpolation of $A_L$, then for each $T,\lambda >0$, it holds uniformly in $L \in \N$ that
	\begin{align}\label{ineq:initial temporal interpolation gap}
		\sup_{t\in[0,T]}\norm{A_L(t)-\widehat A_L(t)}_{(\lambda)}
		\lesssim
		\norm{u_L-\pi_L^\text{disc}u_0}_{(\lambda),L}
		+\omega_{u_0,\lambda}(1/L),
	\end{align}
	which converges to zero as $L\to\infty$.
\end{lemma}
\begin{proof}
	Let $T,\lambda >0$ be arbitrary but fixed. We will first prove~\eqref{ineq:initial semigroup convergence}.  For each $L \in \N$, set $r_L:=u_L-\pi_L^\text{disc}u_0$. We first note that, for any $f:\Z/L\to\R$ with $\norm{f}_{(\lambda),L}<\infty$,
	\begin{align}\label{ineq:continuous and discrete weighted norm}
		\norm{\pi_L^\text{int}f}_{(\lambda)}\le e^{\lambda/L}\norm{f}_{(\lambda),L}.
	\end{align}
	Indeed, if $x\in[x_i,x_{i+1}]$, then $\pi_L^\text{int}f(x)$ is a convex combination of $f(x_i)$ and $f(x_{i+1})$, and $\max\{|x_i|-|x|,|x_{i+1}|-|x|\}\le L^{-1}$. Hence, by Lemma~\ref{lem:weighted discrete semigroup stability},~\eqref{ineq:continuous and discrete weighted norm} and~\eqref{ineq:initial discretized convergence weighted},
	\begin{align*}
		\sup_{t\in[0,T]}\norm{\pi_L^\text{int}P^L_{\lfloor tL^2\rfloor}r_L}_{(\lambda)}
		&\le 2e^{\lambda/L}\exp\left\{\frac{\alpha T\lambda^2}{2}\cosh\lambda\right\}\norm{r_L}_{(\lambda),L}\longrightarrow0, \qquad \text{ as  }L \to \infty,
	\end{align*}
	Therefore, by the triangle inequality, it suffices to prove
	\begin{align*}
		\sup_{t\in[0,T]}\left\|\pi_L^\text{int}P^L_{\lfloor tL^2\rfloor}\pi_L^\text{disc}u_0-P_tu_0\right\|_{(\lambda)}\longrightarrow0,\qquad \text{as }L \to \infty.
	\end{align*}

	Suppose first that $u_0\in\mathcal C_0(\R)$. For every $\varphi\in\mathcal C_c^\infty(\R)$ and $x_i \in \Z/L$,
	\begin{align*}
		L^2(P_1^L-I)\pi_L^\text{disc}\varphi(x_i)
		=
		\frac{\alpha L^2}{2}
		\bigl(
		\varphi(x_i+L^{-1})-2\varphi(x_i)
		+\varphi(x_i-L^{-1})
		\bigr).
	\end{align*}
	Taylor's theorem therefore gives
	\begin{align*}
		\sup_{x_i\in\Z/L}
		\left|
		L^2(P_1^L-I)\pi_L^\text{disc}\varphi(x_i)
		-\frac{\alpha}{2}\varphi''(x_i)
		\right|
		&\longrightarrow0 \qquad \text{ as }L \to \infty.
	\end{align*}
	Since $\mathcal C_c^\infty(\R)$ is a core for the generator $\frac{\alpha}{2}\Delta$ of the heat semigroup on $\mathcal C_0(\R)$, we may apply \cite{ethier2009markov}*{Chapter 1, Theorem 6.5} with the contractions $P_1^L$ on the supremum-norm spaces of lattice functions that vanish at infinity, time steps $L^{-2}$, and comparison maps $\pi_L^\text{disc}$. It follows that
	\begin{align*}
		J_{1,L}
		:=\sup_{t\in[0,T]}\sup_{x_i\in\Z/L}
		\abs{P^L_{\lfloor tL^2\rfloor}\pi_L^\text{disc}u_0(x_i)-P_tu_0(x_i)}
		\longrightarrow0,\qquad \text{ as } L \to \infty.
	\end{align*}
	Because $\pi_L^\text{int}$ is a contraction in the supremum norm, we see that for all $L \in \N$,
	\begin{align*}
		\sup_{t\in[0,T]}
		\left\|
		\pi_L^\text{int}P^L_{\lfloor tL^2\rfloor}\pi_L^\text{disc}u_0
		-\pi_L^\text{int}\pi_L^\text{disc}P_tu_0
		\right\|_{\mathcal C_0(\R)}
		\le J_{1,L}.
	\end{align*}
	Moreover, for $t>0$ and $x,y\in\R$,
	\begin{align*}
		\abs{P_tu_0(x)-P_tu_0(y)}
		&\le \int_\R G_t(z)\abs{u_0(x-z)-u_0(y-z)}\,\mathrm dz
		\le \omega_{u_0}(\abs{x-y}),
	\end{align*}
	and the same bound holds at $t=0$ because $P_0u_0=u_0$. Hence, it holds for all $L \in \N$   that
	\begin{align*}
		\sup_{t\in[0,T]}\norm{P_tu_0-\pi_L^\text{int}\pi_L^\text{disc}P_tu_0}_{\mathcal C_0(\R)}
		\le \omega_{u_0}(1/L).
	\end{align*}
	Since every $u_0\in\mathcal C_0(\R)$ is uniformly continuous, $\omega_{u_0}(1/L)\to0$. Together with $J_{1,L}\to0$, this proves the desired convergence for $u_0\in\mathcal C_0(\R)$.

	Now suppose $u_0 \in \mathcal C_b(\R)$. For a bounded continuous function $v$, set
	\begin{align*}
		D_L(v):=\sup_{t\in[0,T]}\norm{\pi_L^\text{int}P^L_{\lfloor tL^2\rfloor}\pi_L^\text{disc}v-P_tv}_{(\lambda)}.
	\end{align*}
	For each $M\ge1$, choose a continuous function $\chi_M:\R\to[0,1]$ such that $\chi_M=1$ on $[-M,M]$ and $\chi_M=0$ outside $[-M-1,M+1]$, and set $v_M:=(1-\chi_M)u_0$. Then
	\begin{align*}
		D_L(u_0)\le D_L(\chi_Mu_0)+D_L(v_M),
	\end{align*}
	where $D_L(\chi_Mu_0)\to0$ for each fixed $M$ by the preceding $\mathcal C_0(\R)$ case.

	We next control $D_L(v_M)$ uniformly in $L$. Fix $M_0>0$ and take $M>M_0+1$, with $R:=M-M_0-1$. Since $P_k^L$, $P_t$, and $\pi_L^\text{int}$ are contractions in the supremum norm,
	\begin{align*}
		\sup_{\substack{|x|>M_0\\t\in[0,T]}}
		e^{-\lambda|x|}\abs{\pi_L^\text{int}P^L_{\lfloor tL^2\rfloor}\pi_L^\text{disc}v_M(x)-P_tv_M(x)}
		\le 2\norm{u_0}_{\mathcal C_b(\R)}e^{-\lambda M_0}.
	\end{align*}
	Since $v_M$ vanishes on $[-M,M]$, translation invariance of $q_k^L$ gives, for all $x_j \in \Z/L$ with $\abs{x_j} \le M_0+1$, it holds that
	\begin{align*}
		\abs{P_k^L\pi_L^\text{disc}v_M(x_j)}
		&\le \norm{u_0}_{\mathcal C_b(\R)}
		\sum_{|x_\ell|>M}q_k^L(0,x_\ell-x_j)\\
		&\le \norm{u_0}_{\mathcal C_b(\R)}
		\sum_{|z_\ell|>R}q_k^L(0,z_\ell),
	\end{align*}
	In the last line we wrote $z_\ell:=x_\ell-x_j$ and used the fact that $|z_\ell|\ge |x_\ell|-|x_j|>M-M_0-1=R$. Since for all $\abs{x} \le M_0$, there exists $x_i \in \Z/L$ with $\abs{x_i} \le M_0+1$ and $x \in [x_{i-1},x_{i}]$, it follows that
	\begin{align*}
		\sup_{\substack{|x|\le M_0\\k\le TL^2}}
		\abs{\pi_L^\text{int}P^L_k\pi_L^\text{disc}v_M(x)}
		&\le \norm{u_0}_{\mathcal C_b(\R)}
		\sup_{k\le TL^2}\sum_{|z_\ell|>R}q_k^L(0,z_\ell).
	\end{align*}
	Similarly, for $|x|\le M_0$ and $t>0$,
	\begin{align*}
		\abs{P_tv_M(x)}
		&\le \norm{u_0}_{\mathcal C_b(\R)}
		\int_{|y|>M}G_t(x-y)\,\mathrm dy\\
		&\le \norm{u_0}_{\mathcal C_b(\R)}
		\int_{|z|>R}G_t(z)\,\mathrm dz,
	\end{align*}
	At $t=0$, both $\pi_L^\text{int}\pi_L^\text{disc}v_M(x)$ and $P_0v_M(x)=v_M(x)$ vanish for $|x|\le M_0$. Therefore,
	\begin{align*}
		\sup_{\substack{|x|\le M_0\\t\in(0,T]}}
		\abs{P_tv_M(x)}
		&\le \norm{u_0}_{\mathcal C_b(\R)}
		\sup_{t\in(0,T]}\int_{|z|>R}G_t(z)\,\mathrm dz.
	\end{align*}
	Under $q_k^L(0,\cdot)$, the coordinate variable has the same law as $S_k/L$ and therefore has variance $\alpha k/L^2$; likewise, $G_t$ has variance $\alpha t$. Hence Chebyshev's inequality yields
	\begin{align*}
		\sup_{k\le TL^2}\sum_{|z_\ell|>R}q_k^L(0,z_\ell)
		+\sup_{t\in(0,T]}\int_{\abs{z}>R}G_t(z)\,\mathrm dz
		\le \frac{2\alpha T}{R^2}.
	\end{align*}
	Consequently, for every $L\in\N$,
	\begin{align*}
		D_L(v_M)
		\le 2\norm{u_0}_{\mathcal C_b(\R)}e^{-\lambda M_0}
		+\frac{2\alpha T}{R^2}\norm{u_0}_{\mathcal C_b(\R)}.
	\end{align*}
	Since $D_L(\chi_Mu_0)\to0$ for each fixed $M$, we obtain
	\begin{align*}
		\limsup_{L\to\infty}D_L(u_0)
		\le 2\norm{u_0}_{\mathcal C_b(\R)}e^{-\lambda M_0}
		+\frac{2\alpha T}{(M-M_0-1)^2}\norm{u_0}_{\mathcal C_b(\R)}.
	\end{align*}
	Choosing first $M_0$ and then $M$ sufficiently large shows that $D_L(u_0)\to0$, proving~\eqref{ineq:initial semigroup convergence}.

	It remains to prove~\eqref{ineq:initial temporal interpolation gap}. By the definition of $\widehat A_L$, we see it holds for all $t \ge 0$ that
	\begin{align*}
		\norm{A_L(t)-\widehat A_L(t)}_{(\lambda)}
		&\le\norm{\pi_L^\text{int}(P^L_{\lfloor tL^2\rfloor+1}-P^L_{\lfloor tL^2\rfloor})u_L}_{(\lambda)}\\
		&{\overset{\eqref{ineq:continuous and discrete weighted norm}}{\le}}{e^{\lambda/L}\norm{P_{\lfloor tL^2\rfloor}^L(P_1^L-I)u_L}_{(\lambda),L}}\\
		&\le 2e^{\lambda/L}\exp\left\{\frac{\alpha T\lambda^2}{2}\cosh\lambda\right\}
		\norm{(P^L_1-I)u_L}_{(\lambda),L}.
	\end{align*}
	where we used Lemma~\ref{lem:weighted discrete semigroup stability} in the last inequality.

	Recall that $r_L=u_L-\pi_L^\text{disc}u_0$ and note that for any $f: \Z/L \to \R$ with $\norm{f}_{(\lambda), L}< \infty$,
	\begin{align*}
		e^{-\lambda|x_i|}\abs{f(x_{i\pm1})}
		\le e^{\lambda/L}\norm{f}_{(\lambda),L},\qquad \text{ for all  } x_i \in \Z/L.
	\end{align*}
	Hence, it holds for all $L \in \N$ that
	\begin{align*}
		\norm{(P^L_1-I)r_L}_{(\lambda),L}
		&\le \alpha(1+e^{\lambda/L})\norm{r_L}_{(\lambda),L},\\
		\norm{(P^L_1-I)\pi_L^\text{disc}u_0}_{(\lambda),L}
		&\le \alpha\omega_{u_0,\lambda}(1/L).
	\end{align*}
	Since $u_L=r_L+\pi_L^\text{disc}u_0$, the preceding estimates imply that the following holds uniformly in $L \in \N$,
	\begin{align*}
		\sup_{t\in[0,T]}\norm{A_L(t)-\widehat A_L(t)}_{(\lambda)}
		\lesssim
		\norm{r_L}_{(\lambda),L}
		+\omega_{u_0,\lambda}(1/L),
	\end{align*}
	which proves~\eqref{ineq:initial temporal interpolation gap}.

	Finally, we show $\omega_{u_0,\lambda}(\delta)\to0$ as $\delta\downarrow0$ under the stated assumption $u_0\in\mathcal C_b(\R)$. Indeed, given $\varepsilon>0$, choose $K>0$ such that $2\norm{u_0}_{\mathcal C_b(\R)}e^{-\lambda K}<\varepsilon$. By uniform continuity of $u_0$ on $[-K-1,K+1]$, there exists $h_\varepsilon\in(0,1]$ such that $\abs{u_0(x)-u_0(y)}<\varepsilon$ whenever $x,y\in[-K-1,K+1]$ and $|x-y|\le h_\varepsilon$. For $\delta\le h_\varepsilon$, the part of~\eqref{eq:weighted_modulus_of_continuity} with $|x|\le K$ is therefore bounded by $\varepsilon$, while the part with $|x|>K$ is bounded by $2\norm{u_0}_{\mathcal C_b(\R)}e^{-\lambda K}<\varepsilon$. Hence $\omega_{u_0,\lambda}(\delta)\le\varepsilon$, which completes the proof.
\end{proof}

\begin{lemma}[{\cite{di2012hitchhiker}*{Theorem~8.2}}]\label{lem:GRR}
	Let $K \subset \R^d$ be a compact set with Lipschitz boundary, let $p \in [1,\infty)$ and $s \in (0,1)$ so that $sp>d$. Then there exists $C>0$ depending only on $s,p$ and $K$ so that, for every continuous $f:K\to\R$,
	\begin{align*}
		\norm{f}_{\mathcal C_K^{\gamma}} \le C\pr{\norm{f}_{L^p_K}^p + \int_K \int_K \frac{\abs{f(x) - f(y)}^p}{\abs{x-y}^{d+sp}} \mathrm dx \mathrm dy}^\frac{1}{p}
	\end{align*}
	where $\gamma = \pr{sp - d} / p$.
\end{lemma}
\begin{lemma}\label{lem:holder norm bound uniform}
	For $m \in \Z$, denote $I_m = [m,m+1]$ and let $p>8$  and $\kappa_p>3$ be given as in~\eqref{ineq:p_moment bound}. Denote $Z_L(t,x):= \widetilde \xi^L_t(x) - \pi_L^\text{int}P^L_{\lfloor tL^2\rfloor}u_L(x)$, where $u_L:=\xi_0^L$, for each $L \in \N$, and $(t,x) \in \R_+\times \R$.  Set $\kappa:=\frac p4\wedge\frac{1+\kappa_p}{2}>2$. Then, for every $T>0$ and $0<\gamma<(\kappa-2)/p$,
	\begin{align*}
		\sup_{L} \sup_{m \in \Z}\E \bracket{\sup_{t\in [0,T]}\norm{Z_L(t,\cdot)}_{\mathcal C_{I_m}^{\gamma}}^p} < \infty.
	\end{align*}
\end{lemma}
\begin{proof}
	Let  $\widetilde Z_L(t,x)$ be obtained through linearly interpolating in time of $Z_L$ between the points $t_i \in \N/L^2 \cup \{0\}$. Hence $\widetilde Z_L$ is a continuous random field on $\R_+\times \R$, and it is enough to prove the inequality for $\widetilde Z_L$.

	Fix an arbitrary $T>0$ and $0<\gamma<(\kappa-2)/p$,  set $\varsigma:=\gamma+2/p$, so that $2/p<\varsigma<\kappa/p$. Since $\frac p4\wedge(1+\kappa_p)\ge\kappa$, Propositions~\ref{prop:spatial_increment_moment_bound} and~\ref{prop:temporal_increment_moment_bound}, together with the triangle inequality, imply that
	\begin{align*}
		\E\bracket{\abs{Z_L(t_k,x_i)-Z_L(t_j,x_\ell)}^p}
		\lesssim
		\abs{t_k-t_j}^{\kappa}+\abs{x_i-x_\ell}^{\kappa}
	\end{align*}
	uniformly for $t_j,t_k\in L^{-2}\Z_+\cap[0,T+1]$, $m \in \Z$ and $x_i,x_\ell\in I_m\cap(\Z/L)$. Clearly, by the definition of linear interpolation, this implies that uniformly in $m$ and all sufficiently large $L$, it holds uniformly in $t,s \in [0,T]$, $m \in \Z$ and  $x,y \in I_m$ that
	\begin{align}\label{ineq:Z interpolated space time increment}
		\E\bracket{\abs{\widetilde Z_L(t,x)-\widetilde Z_L(s,y)}^p}
		\lesssim\pr{\abs{t-s}+\abs{x-y}}^{\kappa}.
	\end{align}
	Moreover, since $\widetilde Z_L(0,x)=0$ for all $x \in \R$, we see from~\eqref{ineq:Z interpolated space time increment} that
	\begin{align}\label{ineq:Z interpolated lp bound}
		\sup_L\sup_{m\in\Z}\E\bracket{
		\norm{\widetilde Z_L}_{L^p([0,T]\times I_m)}^p}
		<\infty.
	\end{align}
	Then by Lemma~\ref{lem:GRR},  it holds uniformly in $m \in \Z$ and sufficiently large $L\ge 1$ that
	\begin{align*}
		\E \bracket{\norm{\widetilde Z_L}_{\mathcal C_{[0,T]\times I_m}^{\gamma}}^p}&{\lesssim} {\E\bracket{\norm{\widetilde Z_L}_{L^p([0,T]\times I_m)}^p}+ \E \bracket{\iint_{[0,T]^2}\iint_{I_m\times I_m} \frac{\abs{\widetilde Z_L(t,x) -\widetilde Z_L(s,y)}^p}{\pr{\abs{t-s}+\abs{x-y}}^{2+\varsigma p}} \mathrm dx\mathrm dy \mathrm dt \mathrm ds}}\\
		&{\overset{\eqref{ineq:Z interpolated lp bound}}{\lesssim}} {1+ \iint_{[0,T]^2}\iint_{I_m\times I_m} \frac{\E \bracket{\abs{\widetilde Z_L(t,x) -\widetilde Z_L(s,y)}^p}}{\pr{\abs{t-s}+\abs{x-y}}^{2+\varsigma p}} \mathrm dx\mathrm dy \mathrm dt \mathrm ds}\\
		&{\overset{\eqref{ineq:Z interpolated space time increment}}{\lesssim}} {1+ \iint_{[0,T]^2}\iint_{I_m^2} \pr{\abs{t-s}+\abs{x-y}}^{\kappa - \varsigma p - 2}\mathrm dx\mathrm dy \mathrm dt \mathrm ds}\\
		&{<} {\infty},
	\end{align*}
	which is the desired result.
\end{proof}

We will need the following interpolation inequality, whose proof can be found in Appendix~\ref{app:heat-kernel-proofs}.
\begin{lemma}\label{lem:holder-lebesgue-interpolation}
	Let $K \subset \R$ be any compact interval, let $p \ge 1$ and $\gamma \in (0,1]$, then for any continuous $f:K \to \R$, the following inequality holds
	\begin{align*}
		\norm{f}_{\mathcal C_K^0} \lesssim \norm{f}_{L^p_K}^{1-\theta}\norm{f}^\theta_{\mathcal C^\gamma_K}+\norm{f}_{L^p_K},
	\end{align*}
	where $\theta = \frac{1}{\gamma p+1}$, and the implicit constant depends only on $K$, $p\ge 1$ and $\gamma \in (0,1]$.
\end{lemma}

\begin{thm}[Tightness in $\mathcal D(\R_+; \mathcal C_\R^\text{tem})$]\label{thm:tightness}
	Suppose $\widetilde\xi_0^L\to u_0$ in $\mathcal C_\R^\mathrm{tem}$ for some $u_0\in\mathcal C_b(\R)$, and write $u_L:=\xi_0^L$. Let $Y_L$ be the temporal linear interpolation of $\widetilde\xi^L$ between the points in $\Z_+/L^2$. Then $\{Y_L\}_{L\in\N}$ is tight in $\mathcal C(\R_+;\mathcal C_\R^\mathrm{tem})$.

	Let $p >8$ and $\kappa_p >3$ be as in~\eqref{ineq:p_moment bound}, set $\kappa:=\frac{p}{4}\wedge\frac{1+\kappa_p}{2}>2$. Then for every $T,\lambda>0$ and $0<\varepsilon<2(\kappa-2)$, it holds uniformly in $L \in \N$ that
	\begin{align}\label{ineq:full temporal interpolation gap}
		\E\bracket{
		\sup_{t\in[0,T]}
		\norm{\widetilde\xi_t^L-Y_L(t)}_{(\lambda)}^p}
		\lesssim
		L^{-2(\kappa-2)+\varepsilon}
		+\norm{u_L-\pi_L^\text{disc}u_0}_{(\lambda),L}^p
		+\omega_{u_0,\lambda}(1/L)^p.
	\end{align}
	As a consequence, $\{\widetilde\xi^L\}_{L\in\N}$ is tight in $\mathcal D(\R_+;\mathcal C_\R^\mathrm{tem})$, and all its limit points have continuous paths.
\end{thm}
\begin{proof}
	Let $T,\lambda >0$ be arbitrary but fixed, define for $L \in \N$ and $(t,x) \in \R_+\times \R$ that
	\begin{align*}
		A_L(t,x):=\pi_L^\text{int}P^L_{\lfloor tL^2\rfloor}u_L(x),\qquad
		Z_L(t,x):=\widetilde\xi_t^L-A_L(t,x),
	\end{align*}
	and denote their temporal linear interpolations by $\widehat A_L$ and $\widehat Z_L$, respectively. By linearity,
	\begin{align}\label{eq:Y decomposition}
		Y_L=\widehat A_L+\widehat Z_L,
		\qquad
		\widetilde\xi^L-Y_L
		=(A_L-\widehat A_L)+(Z_L-\widehat Z_L).
	\end{align}

	We first estimate the stochastic remainder. Let $t_k:=k/L^2$ for $k \in \Z_+$ and denote
	\begin{align*}
		F_{k,L}:=Z_L(t_{k+1})-Z_L(t_k).
	\end{align*}
	We see by Proposition~\ref{prop:spatial_increment_moment_bound} and Lemma~\ref{lem:GRR} with $d=1$ that, for every $0<\gamma<\frac{\kappa-1}{p}$, it holds uniformly in  $L \in \N,\, m \in \Z$ and $k \in \Z_+$ with $t_k \le T+1$ that
	\begin{align}\label{ineq:increment spatial holder Z}
		\E\bracket{\norm{F_{k,L}}_{\mathcal C_{I_m}^{\gamma}}^p}
		\lesssim 1.
	\end{align}
	On the other hand, since $t_{k+1}-t_k=L^{-2}$, Proposition~\ref{prop:temporal_increment_moment_bound} implies the following holds uniformly in $L \in \N,\, k\in \Z_+$ and $m \in \Z$ with $t_k \le T+1$,
	\begin{align}\label{ineq:increment lp Z}
		\E\bracket{\norm{F_{k,L}}_{L^p_{I_m}}^p}
		\lesssim L^{-2\kappa}.
	\end{align}

	Let $\theta_{\gamma}:=(\gamma p+1)^{-1}$, then Lemma~\ref{lem:holder-lebesgue-interpolation}, H\"older's inequality,~\eqref{ineq:increment spatial holder Z}, and~\eqref{ineq:increment lp Z} imply the following holds uniformly in $L \in \N,\, k\in \Z_+$ and $m \in \Z$ with $t_k \le T+1$,
	\begin{align*}
		\E\bracket{\norm{F_{k,L}}_{\mathcal C^0_{I_m}}^p}
		\lesssim
		\E\bracket{
		\norm{F_{k,L}}_{L^p_{I_m}}^{p(1-\theta_{\gamma})}
		\norm{F_{k,L}}_{\mathcal C_{I_m}^{\gamma}}^{p\theta_{\gamma}}}
		+\E\bracket{\norm{F_{k,L}}_{L^p_{I_m}}^p}\lesssim
		L^{-2\kappa(1-\theta_{\gamma})}
		+L^{-2\kappa}.
	\end{align*}
	For $t\in[t_k,t_{k+1}]$, the definition of temporal interpolation gives that, for all $x \in \R$,
	\begin{align*}
		\abs{Z_L(t,x)-\widehat Z_L(t,x)}
		\le\abs{F_{k,L}(x)}.
	\end{align*}
	Since $e^{-\lambda|x|}\lesssim e^{-\lambda|m|}$ uniformly in $m \in \Z$ and $x\in I_m$, we see it holds uniformly in $L \in \N$ that
	\begin{align*}
		\E\bracket{\sup_{t\in[0,T]}
		\norm{Z_L(t)-\widehat Z_L(t)}_{(\lambda)}^p}
		&\lesssim
		\sum_{k\le(T+1)L^2}\sum_{m\in\Z}
		e^{-\lambda p|m|}
		\E\bracket{\norm{F_{k,L}}_{\mathcal C^0_{I_m}}^p}\\
		&\lesssim
		L^{2-2\kappa(1-\theta_{\gamma})}
		+L^{2-2\kappa}.
	\end{align*}
	Note that $\kappa(1-\theta_{\gamma}) =\frac{\kappa\gamma p}{\gamma p+1} \longrightarrow\kappa-1$ as $\gamma\uparrow(\kappa-1)/p$. Hence, for every $\varepsilon>0$, we may choose $\gamma<(\kappa-1)/p$ sufficiently close to $(\kappa-1)/p$ so that the following holds uniformly in $L \in \N$,
	\begin{align}\label{ineq:stochastic temporal interpolation gap}
		\E\bracket{\sup_{t\in[0,T]}
		\norm{Z_L(t)-\widehat Z_L(t)}_{(\lambda)}^p}
		\lesssim
		L^{-2(\kappa-2)+\varepsilon}.
	\end{align}

	Since $\widetilde\xi_0^L=\pi_L^\text{int}u_L$ and spatial interpolation agrees with $u_L$ on $\Z/L$, we see
	\begin{align*}
		\norm{u_L-\pi_L^\text{disc}u_0}_{(\lambda),L}
		\le\norm{\widetilde\xi_0^L-u_0}_{(\lambda)}
		\longrightarrow0,\qquad \text{ as  }L \to \infty.
	\end{align*}
	Lemma~\ref{lem:initial convergence} therefore applies and gives
	\begin{align}\label{ineq:initial gap used in tightness}
		\sup_{t\in[0,T]}
		\norm{A_L(t)-\widehat A_L(t)}_{(\lambda)}
		\lesssim\norm{u_L-\pi_L^\text{disc}u_0}_{(\lambda),L}
		+\omega_{u_0,\lambda}(1/L),
	\end{align}
	uniformly in $L \in \N$. Therefore, by~\eqref{eq:Y decomposition},~\eqref{ineq:stochastic temporal interpolation gap},~\eqref{ineq:initial gap used in tightness},
	\begin{align*}
		\E\bracket{\sup_{t\in[0,T]}
		\norm{\widetilde\xi_t^L-Y_L(t)}_{(\lambda)}^p}
		\lesssim
		L^{-2(\kappa-2)+\varepsilon}+\norm{u_L-\pi_L^\text{disc}u_0}_{(\lambda),L}^p
		+\omega_{u_0,\lambda}(1/L)^p,
	\end{align*}
	which proves~\eqref{ineq:full temporal interpolation gap}.

	Finally, Lemma~\ref{lem:initial convergence} also shows that $\widehat A_L\to P_\cdot u_0$ in $\mathcal C(\R_+;\mathcal C_\R^\mathrm{tem})$, uniformly on compact time intervals. Propositions~\ref{prop:spatial_increment_moment_bound} and~\ref{prop:temporal_increment_moment_bound}, together with \cite[Lemma 6.3]{shiga1994two}, imply that $\{\widehat Z_L\}_{L\in\N}$, hence  $\{Y_L\}_{L\in\N}$, is tight in $\mathcal C(\R_+;\mathcal C_\R^\mathrm{tem})$.

	By Lemma~\ref{lem:initial convergence}, the right hand side of~\eqref{ineq:full temporal interpolation gap} converges to zero for $0<\varepsilon<2(\kappa-2)$. Hence $\widetilde\xi^L-Y_L\to0$ in probability, uniformly on compact time intervals in $\mathcal C_\R^\mathrm{tem}$. The tightness of $\{\widetilde\xi^L\}_{L\in\N}$ in $\mathcal D(\R_+;\mathcal C_\R^\mathrm{tem})$ follows.
\end{proof}

We will need the following lemma.
\begin{lemma}\label{lem:equicontinuity}
	Suppose $\widetilde\xi_0^L\to u_0$ in $\mathcal C_\R^\mathrm{tem}$ for some $u_0\in\mathcal C_b(\R)$, and set $u_L:=\xi_0^L$. Let $p>8$ be as in~\eqref{ineq:p_moment bound}. Then, for every $K\ge1$ and $T>0$,
	\begin{align*}
		\lim_{L\to\infty}\E\bracket{
		\sup_{t\in[0,T]}\sup_{\abs{x_i}\le K\atop x_i\in\Z/L}
		\abs{\xi_t^L(x_i)-\xi_t^L(x_{i+1})}^p}=0.
	\end{align*}
\end{lemma}
\begin{proof}
	Define for $t \ge 0$ and $L \in \N$ that
	\begin{align*}
		A_L(t):=\pi_L^\text{int}P^L_{\lfloor tL^2\rfloor}u_L,\qquad
		Z_L(t):=\widetilde\xi_t^L-A_L(t).
	\end{align*}
	Note that $\widetilde\xi_0^L=\pi_L^\text{int}u_L(x_i) = u_L(x_i)$ for all $x_i \in \Z/L$. Hence for every $\lambda >0$,
	\begin{align*}
		\norm{u_L-\pi_L^\text{disc}u_0}_{(\lambda),L}
		\le\norm{\widetilde\xi_0^L-u_0}_{(\lambda)}
		\longrightarrow0,\qquad \text{ as }L \to \infty.
	\end{align*}
	So Lemma~\ref{lem:initial convergence} applies. Note that for $t\in[0,T]$ and $x_i\in\Z/L$,
	\begin{align*}
		\abs{\xi_t^L(x_i)-\xi_t^L(x_{i+1})}
		&\le
		\abs{A_L(t,x_i)-A_L(t,x_{i+1})}
		+\abs{Z_L(t,x_i)-Z_L(t,x_{i+1})}.
	\end{align*}

	\noindent We first consider the initial-semigroup term.  We see for any $\lambda >0$,
	\begin{align*}
		\sup_{t\in[0,T]}
		\sup_{\abs{x_i}\le K\atop x_i\in\Z/L}
		\abs{A_L(t,x_i)-A_L(t,x_{i+1})}&\le
		2e^{\lambda(K+1)}
		\sup_{t\in[0,T]}\norm{A_L(t)-P_tu_0}_{(\lambda)}\\
		&\qquad+
		\sup_{t\in[0,T]}
		\sup_{\substack{|x-y|\le L^{-1}\\|x|,|y|\le K+1}}
		\abs{P_tu_0(x)-P_tu_0(y)}.
	\end{align*}
	The first term on the right hand side converges to zero by Lemma~\ref{lem:initial convergence}. Since $(t,x)\mapsto P_tu_0(x)$ is continuous, it is uniformly continuous on $[0,T]\times[-K-1,K+1]$, so the second term also converges to zero. Consequently,
	\begin{align*}
		\lim_{L\to \infty}\sup_{t\in[0,T]}
		\sup_{\abs{x_i}\le K\atop x_i\in\Z/L}
		\abs{A_L(t,x_i)-A_L(t,x_{i+1})}
		=0.
	\end{align*}

	It remains to control the stochastic remainder.  Let $\mathcal I_K:=\{m\in\Z:I_m\cap[-K,K+1]\neq\varnothing\}$. Every pair $x_i,x_{i+1}$ with $|x_i|\le K$ is contained in some $I_m$ with $m\in\mathcal I_K$. Therefore, by Lemma~\ref{lem:holder norm bound uniform}, there exists some $\gamma >0$ so that the following holds uniformly in $L \in \N$,
	\begin{align*}
		\E\bracket{
		\sup_{t\in[0,T]}
		\sup_{\abs{x_i}\le K\atop x_i\in\Z/L}
		\abs{Z_L(t,x_i)-Z_L(t,x_{i+1})}^p}\le L^{-\gamma p}
		\sum_{m\in\mathcal I_K}
		\E\bracket{
		\sup_{t\in[0,T]}
		\norm{Z_L(t,\cdot)}_{\mathcal C^\gamma_{I_m}}^p}
		\lesssim L^{-\gamma p}.
	\end{align*}
	Combining the preceding estimates gives the desired result.
\end{proof}

\section{Characterizing limiting points}\label{sec:limit-points}
By Theorem~\ref{thm:tightness}, the laws of $\{\widetilde\xi^L\}_L$ are tight in $\mathcal D(\R_+;\mathcal C^\mathrm{tem}_\R)$,  which is a Polish space, hence,  Prokhorov's theorem and the Skorokhod representation theorem allow us, after passing to a subsequence and realizing the processes on a common probability space, to assume that $\widetilde\xi^L\to\xi$ almost surely in $\mathcal D(\R_+;\mathcal C^\mathrm{tem}_\R)$. The limit $\xi$ has continuous paths, so this convergence is uniform on compact time intervals in the topology of $\mathcal C^\mathrm{tem}_\R$. In this section, we will show that $\xi$ solves~\eqref{eq: spde target}.

In this section, for a right continuous square integrable martingale $(M_t)_{t \ge 0}$, we shall denote $([M]_t)_{t \ge 0}$ its predictable quadratic variation. In other words, $([M]_t)_{t\ge 0}$ is the unique  predictable, right-continuous increasing process with $[M]_0 = 0$, so that $M_t^2 - [M]_t$ is a martingale in $t \ge 0$. If $M^1$ and $M^2$ are two square integrable right-continuous martingales, we define  $[M^1,M^2]$ to be the unique predictable finite-variation process starting from zero such that $M^1M^2-[M^1,M^2]$ is a martingale.

For $\varphi \in \mathcal C_c^\infty(\R)$, denote for $t \ge 0$ that
\begin{align}\label{eq:mtg phi limit}
	M_t(\varphi):=\langle \xi_t - \xi_0,\varphi\rangle - \int_0^t \frac{\alpha}{2}\langle \xi_s,\Delta \varphi\rangle+\langle b(\xi_s),\varphi\rangle \mathrm ds,
\end{align}
We first note that for any $L \ge 1$, $\widetilde \xi^L$ and $\xi$ are $[0,1]$-valued random fields on $\R_+\times \R$.  Define
\begin{align*}
	\widetilde B_L\pr{\xi_s^L}(x):=\pi_L^\mathrm{int}\pr{B_L\pr{\xi_s^L}}(x)
	=\pi_L^\mathrm{int}\pr{b_L\left(\xi_s^L+\frac{\alpha}{2}\Delta_L\xi_s^L\right)}(x),\qquad (s,x) \in \R_+\times \R.
\end{align*}
For $\varphi\in\mathcal C_c^\infty(\R)$, $L\in\N$, and $t\ge0$, define
\begin{align}\label{eq:mtg varphi prelimit}
	\widetilde M_t^L(\varphi):=  \left\langle \widetilde \xi^L_t- \widetilde \xi_0^L, \varphi \right \rangle -  \int_0^{\frac{\lfloor tL^2 \rfloor}{L^2} } \frac{\alpha}{2L^{-2}}\left \langle \pi_L^\mathrm{int}\pr{\Delta_L \xi^L_s}, \varphi\right \rangle \mathrm ds - \int_0^{\frac{\lfloor tL^2 \rfloor}{L^2}} \left \langle \widetilde B_L\pr{\xi_s^L}, \varphi\right \rangle \mathrm ds.
\end{align}
By~\eqref{eq: weak formulation prelimit}, this is a square integrable martingale in $t\ge0$.
\begin{prop}\label{prop:limiting martingale}
	For any $\varphi \in \mathcal C_c^\infty(\R)$ the process $(M_t(\varphi))_{t \ge 0}$ defined as in~\eqref{eq:mtg phi limit} is a continuous, mean zero, square integrable martingale with respect to the natural filtration of $\xi$. In particular, let $p >8$ be as in~\eqref{ineq:p_moment bound}, then  for any $T>0$
	\begin{align}\label{eq:lp convergence of mtg term}
		\lim_{L \to \infty}  \norm{\sup_{t \in [0,T]}\abs{\widetilde M^L_t(\varphi) - M_t(\varphi)}}_{L^p_\Omega} = 0.
	\end{align}
\end{prop}
\begin{proof}
	We shall fix an arbitrary $T >0$ in this proof. We note that continuity of $M_t(\varphi)$ in $t\ge 0$ is trivial since all terms on the right hand side of~\eqref{eq:mtg phi limit} are continuous in time.

	Let $\varphi \in \mathcal C_c^\infty(\R)$ and $K \ge 1$ so that $\mathrm{supp}\pr{\varphi} \subset [-K,K]$. For any $L \ge 1$, recall $\rho_L$ from~\eqref{eq:interpolation convolution identity}. Consider by Fubini's theorem and summation by parts that
	\begin{align*}
		\left \langle\pi^\mathrm{int}_L\pr{  \Delta_L \xi_s^L}, \varphi\right \rangle &= \int_\R \frac{1}{L}\sum_{x_i \in \Z/L} \Delta_L\xi^L_s(x_i) \rho_L(x - x_i) \varphi(x) \mathrm dx\\
		&= \frac{1}{L}\sum_{x_i \in \Z/L}\Delta_L\xi^L_s(x_i)  \int_\R \rho_L(x- x_i ) \varphi(x) \mathrm dx  \\
		&=\frac{1}{L} \sum_{x_i \in \Z/L} \Delta_L \xi_s^L(x_i) \int_\R \rho_L(x) \varphi(x+x_i) \mathrm dx\\
		&=\frac{1}{L}\sum_{x_i \in \Z/L} \xi_s^L(x_i) \int_\R \rho_L(x) \pr{\varphi(x+x_{i+1}) -2\varphi(x+x_i) +\varphi(x+x_{i - 1})}\mathrm dx \\
		&=\frac{1}{L}\sum_{x_i \in \Z/L} \xi_s^L(x_i) \int_\R \rho_L(x - x_i) \pr{\varphi\pr{x+\frac{1}{L}} - 2\varphi(x) + \varphi\pr{x - \frac{1}{L}} }\mathrm dx \\
		&= \frac{1}{L} \sum_{x_i \in \Z/L} \xi_s^L(x_i) \int_\R \rho_L(x-x_i) \Delta_L\varphi(x) \mathrm dx \\
		&=\int_\R \pr{\frac{1}{L} \sum_{x_i \in \Z/L} \xi_s^L(x_i) \rho_L(x - x_i)  }  \Delta_L \varphi(x) \mathrm dx   \\
		&=\left \langle \widetilde \xi_s^L, \Delta_L \varphi\right \rangle.
	\end{align*}
	Therefore, for any $t>0$
	\begin{equation}\label{ineq: martingale difference}
		\begin{aligned}
			\abs{M_t(\varphi) - \widetilde M_t^L(\varphi)} &\le  \abs{  \left\langle  \xi_t-  \xi_0, \varphi \right \rangle -\left\langle \widetilde \xi^L_t- \widetilde \xi_0^L, \varphi \right \rangle}+\int_0^{\frac{\lfloor tL^2\rfloor}{L^2}} \frac{\alpha}{2} \abs{\left \langle   \xi_s-\widetilde \xi^L_s, \Delta \varphi\right \rangle} \mathrm ds \\
			&\qquad + \int_0^{\frac{\lfloor tL^2 \rfloor}{L^2}} \abs{ \left\langle b(\xi_s) - b(\widetilde \xi_s^L), \varphi\right\rangle} \mathrm ds + E_1^L(t)+E_2^L(t)+E_3^L(t),
		\end{aligned}
	\end{equation}
	where
	\begin{align*}
		E_1^L(t) &=\int_0^{\frac{\lfloor tL^2 \rfloor}{L^2}}\frac{\alpha}{2} \abs{\left \langle \widetilde \xi^L_s, \frac{\Delta_L}{L^{-2}}\varphi - \Delta \varphi \right \rangle }\mathrm ds\\
		E_2^L (t)&=\int_0^{\frac{\lfloor tL^2 \rfloor}{L^2}} \abs{\left \langle\widetilde B_L\pr{\xi_s^L}- b\pr{\widetilde \xi^L_s}, \varphi \right \rangle} \mathrm ds\\
		E_3^L(t)&=\int_{\frac{\lfloor tL^2 \rfloor}{L^2}}^t \frac{\alpha}{2}\abs{\langle \xi_s, \Delta \varphi \rangle} + \abs{\langle b(\xi_s), \varphi \rangle}\mathrm ds.
	\end{align*}

	We see by boundedness of $\xi$, $\xi^L$ and the fact that $\varphi \in \mathcal C_c^\infty(\R)$ that
	\begin{align*}
		\limsup_{L \to \infty}\sup_{t \in [0,T]} E_3^L(t) + E_1^L(t) = 0,\qquad \text{ almost surely. }
	\end{align*}
	Let us denote for $f: [0,1] \to \R$ that
	\begin{align}\label{eq:modulus of continuity on unit interval}
		\omega_f^0(\delta):=\sup\left\{\abs{f(u)-f(v)}:u,v\in[0,1],\ \abs{u-v}\le \delta\right\},\qquad \text{ for }\delta \ge 0.
	\end{align}
	Since $b$ is continuous on $[0,1]$,  we see $\lim_{\delta\downarrow 0}\omega_b^0(\delta)=0$.

	Let $x \in \R$, then there exists $x_i = i/L \in \Z/L$ so that  $x=\tau x_i+(1-\tau)x_{i+1}$ for some $\tau\in[0,1]$. The definitions of $\widetilde B_L$ and linear interpolation give
	\begin{align*}
		\abs{\widetilde B_L\pr{\xi_s^L}(x)-b\pr{\widetilde\xi_s^L(x)}}\le \sup_{y\in[0,1]}\abs{b_L(y)-b(y)}
		+\omega_b^0\left(\max_{-1\le j\le1}
		\abs{\xi_s^L(x_{i+j+1})-\xi_s^L(x_{i+j})}\right).
	\end{align*}
	Consequently,
	\begin{align*}
		\sup_{\abs{x}\le K}
		\abs{\widetilde B_L\pr{\xi_s^L}(x)-b\pr{\widetilde\xi_s^L(x)}}\le \sup_{y\in[0,1]}\abs{b_L(y)-b(y)}
		+\omega_b^0\left(\sup_{\substack{\abs{x_i}\le K+1\\x_i\in\Z/L}}
		\abs{\xi_s^L(x_{i+1})-\xi_s^L(x_i)}\right).
	\end{align*}
	Lemma~\ref{lem:equicontinuity}, the boundedness of $\omega_b^0$, and dominated convergence therefore imply
	\begin{align}
		\limsup_{L\to\infty}\norm{\sup_{t\in[0,T]}E_2^L(t)}_{L^p_\Omega}
		&\lesssim \norm{\varphi}_{L^1_\R}\limsup_{L\to\infty}
		\left\{\sup_{y\in[0,1]}\abs{b_L(y)-b(y)}\right.\nonumber\\
		&\left.\qquad+\norm{\omega_b^0\left(\sup_{\substack{s\in[0,T],\ \abs{x_i}\le K+1\\x_i\in\Z/L}}
		\abs{\xi_s^L(x_{i+1})-\xi_s^L(x_i)}\right)}_{L^p_\Omega}\right\}=0.
	\end{align}
	By the almost-sure convergence obtained from Theorem~\ref{thm:tightness}, for every $K\ge1$, it holds that
	\begin{align*}
		\limsup_{L \to \infty}\sup_{t \in [0,T]} \sup_{\abs{x} \le K} \abs{\widetilde\xi_t^L(x) - \xi_t(x)} = 0,\qquad \text{almost surely.}
	\end{align*}
	Hence, inserting the previous estimates into~\eqref{ineq: martingale difference}, we see by the uniform boundedness of $\xi^L,\xi$ and the dominated convergence theorem, it holds that
	\begin{align*}
		\lim_{L \to \infty}\norm{\sup_{t\in[0,T]}
		\abs{\widetilde M_t^L(\varphi)-M_t\pr{\varphi}}}_{L^p_\Omega}
		\lesssim&
		\lim_{L\to \infty}\bigg(\norm{\sup_{t\in[0,T]}\sup_{\abs{x}\le K}
		\abs{\widetilde\xi_t^L(x)-\xi_t(x)}}_{L^p_\Omega}\\
		&\qquad\qquad\qquad+\norm{\omega_b^0\left(\sup_{t\in[0,T]}\sup_{\abs{x}\le K}
		\abs{\widetilde\xi_t^L(x)-\xi_t(x)}\right)}_{L^p_\Omega}\bigg)\\
		=&0,
	\end{align*}
	which proves~\eqref{eq:lp convergence of mtg term}.

	Moreover, $\widetilde M^L(\varphi)$ is adapted to the natural filtration of $\widetilde\xi^L$ and is a martingale with respect to that filtration, while $M(\varphi)$ is adapted to the natural filtration of $\xi$. Since $p>1$,~\eqref{eq:lp convergence of mtg term} and Theorem~\ref{thm:tightness} yield the joint convergence in law
	\begin{align*}
		\pr{\widetilde\xi^L,\widetilde M^L(\varphi)}
		\Longrightarrow \pr{\xi,M(\varphi)}.
	\end{align*}
	In addition, by the uniform boundedness of $\widetilde \xi^L$ and $\xi$ and of $b_L$ for large $L$ (by~\eqref{eq:coefficient convergence}),   $\{\widetilde M_t^L(\varphi)\}_L$ is uniformly integrable for every $t\ge0$. Therefore, by \cite{ethier2009markov}*{Chapter 7, Problem 7, p.~362} with $E=\mathcal C^\mathrm{tem}_\R$,  $M(\varphi)$ is a martingale with respect to the natural filtration of $\xi$. Its square integrability follows again from~\eqref{eq:lp convergence of mtg term}, and it has mean zero because $M_0(\varphi)=0$.
\end{proof}
For $L\ge 1$ and $\varphi \in \mathcal C(\R)$, we denote $\widehat \varphi_L:\Z/L \to \R$ via
\begin{align*}
	\widehat \varphi_L(x_i):= \int_\R \rho_L(x - x_i) \varphi(x) \mathrm dx,\qquad x_i \in \Z/L.
\end{align*}
Denote for $(s,x) \in \R_+\times \R$ and $x_i \in \Z/L$ that
\begin{align*}
	\Gamma_L\pr{\xi_s^L}(x_i)
	&:=\sigma_L\pr{\xi_s^L(x_i)+\frac{\alpha}{2}\Delta_L\xi_s^L(x_i)}
	-L^{-3}b_L\pr{\xi_s^L(x_i)+\frac{\alpha}{2}\Delta_L\xi_s^L(x_i)}^2,\\
	\widetilde\Gamma_L\pr{\xi_s^L}(x)
	&:=\pi_L^\mathrm{int}\pr{\Gamma_L\pr{\xi_s^L}}(x).
\end{align*}
\begin{lemma}\label{lem:quadratic variation prelimit}
	Let $L \ge 1$, $\varphi,\psi \in \mathcal C_c^\infty(\R)$ and $(\widetilde M_t^L (\varphi))_t$ be as in~\eqref{eq:mtg varphi prelimit}.  Then
	\begin{align*}
		[\widetilde M^L(\varphi),\widetilde M^L(\psi)]_t
		=\int_0^{\frac{\lfloor tL^2\rfloor}{L^2}}\frac{1}{L} \sum_{x_i \in \Z/L}
		\widehat{\varphi}_L(x_i)\widehat{\psi}_L(x_i)
		\Gamma_L\pr{\xi_{s}^L}(x_i) \mathrm ds,
		\qquad 0\le t<\infty\quad\text{a.s.}
	\end{align*}
	In addition, almost surely in $0\le t < \infty$,
	\begin{align*}
		\abs{[\widetilde M^L(\varphi),\widetilde M^L(\psi)]_t
		-\int_0^{\frac{\lfloor tL^2\rfloor}{L^2}}
		\left\langle\varphi\psi,\widetilde\Gamma_L\pr{\xi_s^L}\right\rangle\mathrm ds}
		\le t\omega_\psi\pr{ 2L^{-1}}
		\norm{\varphi}_{L^1_\R}\sup_{y\in[0,1]}\sigma_L(y),
		\qquad
	\end{align*}
\end{lemma}
\begin{proof}
	Recall $\eta^L$ in~\eqref{eq:martingale decomposition one step} for $L \ge 1$, then by~\eqref{eq: weak formulation prelimit} and Fubini's theorem, it holds for all $t \ge 0$ and $L \ge 1$ that
	\begin{align*}
		\widetilde M^L_t(\varphi)&= \sum_{k = 1}^{{\lfloor tL^2 \rfloor}} \int_\R \frac{1}{L}\sum_{x_i \in \Z/L}  \eta^L_k(x_i) \rho_L(x_i-x) \varphi(x) \mathrm dx\\
		&=\sum_{k = 1}^{{\lfloor tL^2 \rfloor}}\frac{1}{L}\sum_{x_i \in \Z/L}  \eta^L_k(x_i) \int_\R  \rho_L(x_i-x) \varphi(x) \mathrm dx\\
		&=:\sum_{k = 1}^{{\lfloor tL^2 \rfloor}}\frac{1}{L}\sum_{x_i \in \Z/L}  \eta^L_k(x_i) \widehat{\varphi}_L(x_i).
	\end{align*}
	Hence, by conditional independence, we see
	\begin{align*}
		[\widetilde M^L(\varphi),\widetilde M^L(\psi)]_t &= \sum_{k = 1}^{\lfloor tL^2\rfloor}\frac{1}{L^2}\mathbb E \bracket{\pr{ \sum_{x_i \in \Z/L} \eta^L_k(x_i)\widehat{\varphi}_L(x_i)   }\pr{ \sum_{x_i \in \Z/L} \eta^L_k(x_i)\widehat{\psi}_L(x_i)   } \mid \mathcal F_{k-1}}\\
		&=\sum_{k = 1}^{\lfloor tL^2\rfloor}\frac{1}{L^2} \sum_{x_i,x_j \in \Z/L} \widehat{\varphi}_L(x_i)\widehat{\psi}_L(x_j) \E \bracket{\eta^L_k(x_i) \eta^L_k(x_j)\mid \mathcal F_{k-1}}\\
		&=\sum_{k = 1}^{\lfloor tL^2\rfloor  } \frac{1}{L^2} \sum_{x_i \in \Z/L} \widehat \varphi_L(x_i)\widehat \psi_L(x_i) \E\bracket{\eta^L_k(x_i)^2\mid \mathcal F_{k-1}}\\
		&\overset{\eqref{eq:mtg_increment l2} }{=} \frac{1}{L^2}\sum_{k = 0}^{\lfloor tL^2\rfloor-1} \frac{1}{L}\sum_{x_i \in \Z/L} \widehat{\varphi}_L (x_i)\widehat{\psi}_L (x_i) \Gamma_L\pr{\xi_{k/L^2}^L}(x_i)\\
		&= \int_0^{\frac{\lfloor tL^2\rfloor}{L^2}} \frac{1}{L} \sum_{x_i \in \Z/L} \widehat{\varphi}_L(x_i)\widehat{\psi}_L (x_i) \Gamma_L\pr{\xi_{s}^L}(x_i) \mathrm ds,
	\end{align*}
	which proves the first part of the Lemma.

	For the second part, since $\int_\R \rho_L(x) \mathrm dx = 1$ and $\mathrm{supp}(\rho_L) \subset[-1/L,1/L]$ for all $L \ge 1$, we see for $x_i \in \Z/L$,
	\begin{align*}
		\abs{\widehat \varphi_L(x_i)\widehat{\psi}_L (x_i) - \widehat{\varphi\cdot\psi}_L(x_i)}&= \abs{\iint_{\R^2} \varphi(x)\psi(y) \rho_L(x-x_i)\rho_L(y - x_i) \mathrm dy \mathrm dx - \int_{\R} \varphi(x)\psi(x) \rho_L(x-x_i) \mathrm dx}\\
		&=\abs{\iint_{\R^2} \varphi(x) \pr{\psi(y) - \psi(x)}\rho_{L}(x-x_i) \rho_L(y - x_i) \mathrm dy \mathrm dx}\\
		&\le \iint_{\R^2} \abs{\varphi(x)} \abs{\psi(y) - \psi(x)}\rho_{L}(x-x_i) \rho_L(y - x_i) \mathrm dy \mathrm dx\\
		&= \iint_{[x_i - 1/L,x_i+1/L]^2}\abs{\varphi(x)} \abs{\psi(y) - \psi(x)}\rho_{L}(x-x_i) \rho_L(y - x_i) \mathrm dy \mathrm dx\\
		&\le \omega_\psi\pr{2L^{-1}} \int_\R \abs{\varphi(x)} \rho_L(x- x_i) \mathrm dx\\
		&= \omega_\psi\pr{2L^{-1}} \widehat{\abs{\varphi} }_L(x_i).
	\end{align*}
	Therefore,
	\begin{align*}
		\abs{\frac{1}{L}\sum_{x_i \in \Z/L} \widehat\varphi_L(x_i) \widehat{\psi}_L(x_i) \Gamma_L\pr{\xi^L_s}(x_i) - \left \langle \varphi\psi, \widetilde \Gamma_L\pr{\xi^L_s} \right \rangle}&\le \frac{1}{L} \sum_{x_i \in \Z/L } \abs{\widehat \varphi_L(x_i)\widehat{\psi}_L(x_i)  - \widehat{\varphi\cdot \psi}_L(x_i)}\Gamma_L\pr{\xi^L_s}(x_i)\\
		&\le \omega_\psi\pr{2L^{-1}} \frac{1}{L} \sum_{x_i \in \Z/L} \widehat{\abs{\varphi}}_L(x_i) \Gamma_L\pr{\xi^L_s}(x_i)\\
		&\le \omega_\psi(2L^{-1}) \norm{\varphi}_{L^1_\R}\sup_{y\in[0,1]}\sigma_L(y),
	\end{align*}
	Indeed,~\eqref{eq:R assumption} gives
	\begin{align*}
		\Gamma_L(\xi_s^L)(x_i)
		=L\operatorname{Var}_{x}\pr{\mathcal R^L(x)-x}
		\bigg|_{x=\xi_s^L(x_i)+\frac{\alpha}{2}\Delta_L\xi_s^L(x_i)},
	\end{align*}
	and hence $0\le\Gamma_L(\xi_s^L)(x_i)\le\sup_{y\in[0,1]}\sigma_L(y)$. Integrating with respect to $s$ proves the second assertion.
\end{proof}

\begin{prop}\label{prop:characterization of limiting martingale}
	Let $\varphi,\,\psi\in\mathcal C_c^\infty(\R)$. Then
	\begin{align*}
		\left[M(\varphi) ,M(\psi) \right]_t = \int_0^t \langle \sigma(\xi_s), \varphi\psi\rangle\mathrm ds,\qquad t \ge 0.
	\end{align*}
\end{prop}
\begin{proof}
	For $t\ge0$ and $\varphi,\psi\in\mathcal C_c^\infty(\R)$, set
	\begin{align*}
		A_t(\varphi,\psi):=\int_0^t\langle\sigma(\xi_s),\varphi\psi\rangle\,\mathrm ds.
	\end{align*}
	Since $A(\varphi,\psi)$ is continuous, adapted, and of finite variation, it is predictable. By uniqueness of predictable covariation, it therefore suffices to show that $M(\varphi)M(\psi)-A(\varphi,\psi)$ is a martingale. For each $L$, the corresponding compensated product
	\begin{align*}
		N_t^L:=\widetilde M_t^L(\varphi)\widetilde M_t^L(\psi)
		-\bracket{\widetilde M^L(\varphi),\widetilde M^L(\psi)}_t,\qquad t \ge 0
	\end{align*}
	is a martingale. Write $N_t:=M_t(\varphi)M_t(\psi)-A_t(\varphi,\psi)$. We will prove that, for every $T>0$, $N^L$ converges to $N$ uniformly on $[0,T]$ in $L^2(\Omega)$; equivalently,
	\begin{align*}
		\lim_{L \to \infty}\norm{\sup_{t\in[0,T]}\abs{\pr{\widetilde M^L_t(\varphi) \widetilde M^L_t(\psi) - \bracket{\widetilde M^L(\varphi),\widetilde M^L(\psi)}_t} - \pr{M_t(\varphi)M_t(\psi) -\int_0^t \langle \sigma(\xi_s), \varphi\psi\rangle \mathrm ds}}}_{L^2_\Omega} =0.
	\end{align*}
	By~\eqref{eq:lp convergence of mtg term} and H\"older's inequality, the product terms converge uniformly on $[0,T]$ in $L^2(\Omega)$. Moreover, Lemma~\ref{lem:quadratic variation prelimit} gives
	\begin{align*}
		\sup_{t\in[0,T]}\abs{\bracket{\widetilde M^L(\varphi),\widetilde M^L(\psi)}_t-A_t(\varphi,\psi)}
		&\le T\omega_\psi(2L^{-1})\norm{\varphi}_{L^1_\R}
		\sup_{y\in[0,1]}\sigma_L(y)\\
		&\qquad+\int_0^{\frac{\lceil TL^2\rceil}{L^2}}
		\left\langle\abs{\widetilde\Gamma_L(\xi_s^L)-\sigma(\xi_s)},
		\abs{\varphi\psi}\right\rangle\mathrm ds\\
		&\qquad+L^{-2}\norm{\varphi\psi}_{L^1_\R}\sup_{y\in[0,1]}\sigma(y).
	\end{align*}
	By~\eqref{eq:coefficient convergence}, $\sup_{y\in[0,1]}\sigma_L(y)$ is uniformly bounded for all sufficiently large $L$. Hence the first and third terms on the right hand side converge to zero. It therefore remains to prove
	\begin{align}\label{eq: char of limiting martingale done}
		\lim_{L \to \infty}  \norm{\int_0^{\frac{\lceil TL^2\rceil}{L^2}} \left \langle \abs{\widetilde \Gamma_L(\xi_s^L) - \sigma(\xi_s)}, \abs{\varphi\psi}\right \rangle \mathrm ds}_{L^2_\Omega} = 0.
	\end{align}
	Recall $\omega_f^0$ from~\eqref{eq:modulus of continuity on unit interval}. Since $\sigma$ is continuous on $[0,1]$, $\omega_\sigma^0(r)\to0$ as $r\downarrow0$. Suppose $\operatorname{supp}(\varphi\psi)\subset[-K,K]$ for some $K\ge1$. If $x\in[x_i,x_{i+1}]$ for some $x_i = i/L$, then the definition of $\widetilde\Gamma_L$ and spatial linear interpolation give
	\begin{align*}
		\abs{\widetilde\Gamma_L\pr{\xi_s^L}(x)-\sigma\pr{\widetilde\xi_s^L(x)}}\le \sup_{y\in[0,1]}\abs{\sigma_L(y)-\sigma(y)}
		&+L^{-3}\sup_{y\in[0,1]}\abs{b_L(y)}^2\\
		&+\omega_\sigma^0\left(\max_{-1\le j\le1}
		\abs{\xi_s^L(x_{i+j+1})-\xi_s^L(x_{i+j})}\right).
	\end{align*}
	Therefore, for all sufficiently large $L$,
	\begin{align*}
		\sup_{\substack{s\in[0,T+1]\\\abs{x}\le K}}
		\abs{\widetilde\Gamma_L\pr{\xi_s^L}(x)-\sigma\pr{\xi_s(x)}}&\le \sup_{y\in[0,1]}\abs{\sigma_L(y)-\sigma(y)}+L^{-3}\sup_{y\in[0,1]}\abs{b_L(y)}^2\\
		&\qquad\quad+\omega_\sigma^0\left(
		\sup_{\substack{s\in[0,T+1],\ \abs{x_i}\le K+1\\x_i\in\Z/L}}
		\abs{\xi_s^L(x_{i+1})-\xi_s^L(x_i)}\right)\\
		&\qquad\quad+\omega_\sigma^0\left(
		\sup_{\substack{s\in[0,T+1]\\\abs{x}\le K}}
		\abs{\widetilde\xi_s^L(x)-\xi_s(x)}\right).
	\end{align*}
	By Lemma~\ref{lem:equicontinuity}, the third term converges to zero in $L^p(\Omega)$. The almost-sure local uniform convergence $\widetilde\xi^L\to\xi$ implies that the fourth term converges to zero almost surely. The first two terms in the preceding bound converge to zero by~\eqref{eq:coefficient convergence}. Therefore, the boundedness and continuity of $\omega_\sigma^0$ imply that the right hand side converges to zero in $L^2(\Omega)$ as $L \to \infty$. Since $\lceil TL^2\rceil/L^2\le T+1$, we see~\eqref{eq: char of limiting martingale done} holds.

	Consequently, $N^L\to N$ uniformly on $[0,T]$ in $L^2(\Omega)$. Together with the almost-sure convergence $\widetilde\xi^L\to\xi$, this yields
	\begin{align*}
		(\widetilde\xi^L,N^L)\Longrightarrow(\xi,N),
	\end{align*}
	and the $L^2$ convergence implies that $\{N_t^L\}_L$ is uniformly integrable for every $t\ge0$. Moreover, $N^L$ is a martingale with respect to the natural filtration of $\widetilde\xi^L$, while $N$ is adapted to the natural filtration of $\xi$. Hence \cite{ethier2009markov}*{Chapter 7, Problem 7, p.~362} implies that $N$ is a martingale with respect to the natural filtration of $\xi$. The uniqueness of predictable covariation therefore gives $[M(\varphi),M(\psi)]=A(\varphi,\psi)$.
\end{proof}
\begin{proof}[Proof of Theorem~\ref{thm:main}]
	We recall the standard argument for completeness. Fix $n\in\N$, test functions $\varphi_1,\dots,\varphi_n\in\mathcal C_c^\infty(\R)$, and $F\in\mathcal C_c^\infty(\R^n)$. Set
	\begin{align*}
		Z_t^i:=\langle\xi_t,\varphi_i\rangle,
		\qquad Z_t:=\pr{Z_t^1,\dots,Z_t^n}.
	\end{align*}
	Recall $M$ from~\eqref{eq:mtg phi limit}, Proposition~\ref{prop:limiting martingale} gives
	\begin{align*}
		Z_t^i=Z_0^i+\int_0^t
		\frac{\alpha}{2}\langle\xi_s,\Delta\varphi_i\rangle
		+\langle b(\xi_s),\varphi_i\rangle\mathrm ds+M_t(\varphi_i).
	\end{align*}
	Note that since $\varphi_i$'s have compact support, $(Z_t^i)_{t \ge 0}$ is a continuous square integrable semi-martingale for each $i = 1,\dots, n$. By Proposition~\ref{prop:characterization of limiting martingale}, we see
	\begin{align*}
		[M(\varphi_i),M(\varphi_j)]_t
		=\int_0^t\langle\sigma(\xi_s),\varphi_i\varphi_j\rangle\,\mathrm ds.
	\end{align*}
	Therefore, It\^o's formula shows that
	\begin{align*}
		F(Z_t)-F(Z_0)
		&-\int_0^t\sum_{i=1}^n\partial_iF(Z_s)
		\left\{
		\frac{\alpha}{2}\langle\xi_s,\Delta\varphi_i\rangle
		+\langle b(\xi_s),\varphi_i\rangle
		\right\}\mathrm ds\\
		&-\frac12\int_0^t\sum_{i,j=1}^n\partial_{ij}F(Z_s)
		\langle\sigma(\xi_s),\varphi_i\varphi_j\rangle\,\mathrm ds
	\end{align*}
	is a mean zero martingale in $t\ge 0$. Hence the law of $\xi$ solves the  martingale problem associated with~\eqref{eq: spde target}. By \cite{iwata1987infinite}*{Theorem 3.3}, on an extension of the underlying probability space there exists a space--time white noise $W$ such that $\xi$ is a weak solution of~\eqref{eq: spde target}. Weak uniqueness for~\eqref{eq: spde target} therefore identifies the law of every subsequential limit with the law of its unique weak solution. Together with tightness, this proves the convergence of the full sequence.
\end{proof}
\appendix
\section{Appendix}\label{app:heat-kernel-proofs}

\begin{proof}[Proof of Corollary~\ref{cor:heat kernel on diagonal bound}]
	Suppose $n = 2k$ for some $k \ge 1$. Then, by Cauchy--Schwarz,
	\begin{align*}
		p_{2k}(x,y) &= \sum_z p_k(x,z)p_k(z,y)\\
		& \le \pr{\sum_z p_k(x,z)^2}^\frac{1}{2}\pr{\sum_z p_k(y,z)^2}^\frac{1}{2}\\
		&=\sqrt{p_{2k}(x,x)}\sqrt{p_{2k}(y,y)}\\
		&=p_{2k}(0,0)\\
		&\lesssim k^{-1/2},
	\end{align*}
	where we used Proposition~\ref{prop:lawler-242} in the last inequality.

	Suppose now $n = 2k+1$ for some $k \ge 1$. Then
	\begin{align*}
		p_{2k+1}(x,y)  = \sum_z p_{2k}(x,z)p_1(z,y) \le \sup_z p_{2k}(x,z) =p_{2k}(0,0) \lesssim k^{-1/2} \lesssim\pr{2k+1}^{-1/2}.
	\end{align*}
	The case $n=1$ is immediate.
\end{proof}

\begin{proof}[Proof of Corollary~\ref{cor:gradient l2 bound}]
	Consider $n\in \N$ that
	\begin{align*}
		\sum_z\pr{p_n(x,z) - p_n(y,z)}^2  &\le \sum_z \abs{p_n(x,z)-p_n(y,z)}\pr{p_n(x,z)+p_n(y,z)}\\
		& \le 2\sup_{u,v} p_n(u,v) \sum_z \abs{p_n(x,z)-p_n(y,z)}\\
		&\lesssim n^{-1/2}\frac{\abs{x-y}}{n^{1/2}},
	\end{align*}
	where we used Proposition~\ref{prop:lawler-242}  and Corollary~\ref{cor:heat kernel on diagonal bound} in the last line.

	On the other hand, by Proposition~\ref{prop:lawler-242} again, we see
	\begin{align*}
		\sum_z\pr{p_n(x,z) - p_n(y,z)}^2& = \sum_z\pr{p_n(x,z)^2-2p_n(x,z)p_n(y,z)+p_n(y,z)^2}\\
		&\le p_{2n}(x,x)+p_{2n}(y,y)\\
		&=2p_{2n}(0,0)\lesssim n^{-1/2}.
	\end{align*}
	Finally, the semigroup property implies that, for $n,h\ge1$ and $x \in \Z$,
	\begin{align*}
		\sum_{z\in\Z}\abs{p_{n+h}(x,z)-p_n(x,z)}
		&\le\sum_{w\in\Z}p_h(x,w)
		\sum_{z\in\Z}\abs{p_n(w,z)-p_n(x,z)}\\
		&{\overset{\eqref{ineq:discrete-heat-kernel-gradient-abs-difference}}{\lesssim}} {n^{-1/2}\sum_{w\in\Z}p_h(x,w)\abs{w-x}}\\
		&\le n^{-1/2}
		\left(\sum_{w\in\Z}p_h(x,w)\abs{w-x}^2\right)^{1/2}\\
		&=\sqrt{\frac{\alpha h}{n}},
	\end{align*}
	where we used Jensen's inequality in the third line and the fact that $\E_x[(S_h-S_0)^2]=\alpha h$ in the last line. Together with the trivial bound by $2$, this proves~\eqref{ineq:heat kernel temporal l1 difference}. This completes the proof.
\end{proof}

\begin{proof}[Proof of Lemma~\ref{lem: heat kernel temporal l2 difference bound}]
	By translation invariance we may assume $x=0$. By the semigroup property, we see it holds for all $n,k \in \N$ that
	\begin{align*}
		\sum_{y\in\Z}\bigl(p_{n+k}(0,y)-p_n(0,y)\bigr)^2
		&= \sum_y p_{n+k}(0,y)^2-2\sum_y p_{n+k}(0,y)p_n(0,y)+\sum_y p_n(0,y)^2\\
		&= p_{2n+2k}(0,0)-2p_{2n+k}(0,0)+p_{2n}(0,0).
	\end{align*}
	By Lemma~\ref{lem:local clt} with $x=0$, for all $m,h\ge 1$,
	\begin{align*}
		\abs{p_{m+h}(0,0)-p_m(0,0)}
		&\le \abs{p_{m+h}(0,0)-(2\pi\alpha(m+h))^{-1/2}}\\
		&\qquad +\abs{(2\pi\alpha(m+h))^{-1/2}-(2\pi\alpha m)^{-1/2}}\\
		&\qquad +\abs{(2\pi\alpha m)^{-1/2}-p_m(0,0)}\\
		&\lesssim (m+h)^{-3/2}+hm^{-3/2}+m^{-3/2}
		\lesssim hm^{-3/2}.
	\end{align*}
	Consequently,
	\begin{align*}
		p_{2n+2k}(0,0)-2p_{2n+k}(0,0)+p_{2n}(0,0)&\le \abs{p_{2n+2k}(0,0)-p_{2n+k}(0,0)}
		+\abs{p_{2n+k}(0,0)-p_{2n}(0,0)}\\
		&\lesssim k(2n+k)^{-3/2}+k(2n)^{-3/2}\\
		&\lesssim kn^{-3/2},
	\end{align*}
	as desired.
\end{proof}

\begin{proof}[Proof of Lemma~\ref{lem:weighted discrete semigroup stability}]
	Let $(S_k)_{k\ge0}$ be the lazy symmetric random walk introduced in Section~\ref{sec:discrete heat kernel}. Its increments $X_1,X_2,\ldots$ are independent and satisfy
	\[
		\P(X_j=1)=\P(X_j=-1)=\frac{\alpha}{2},
		\qquad
		\P(X_j=0)=1-\alpha.
	\]
	By translation invariance, for $x_i=i/L$,
	\[
		P_k^Lf(x_i)
		=\E\left[f\left(x_i+\frac{S_k}{L}\right)\right].
	\]
	Consequently,
	\begin{align*}
		e^{-\lambda|x_i|}\abs{P_k^Lf(x_i)}
		&\le
		\norm{f}_{(\lambda),L}
		\E\left[
		e^{\lambda\left(
		|x_i+S_k/L|-|x_i|
		\right)}
		\right] \\
		&\le
		\norm{f}_{(\lambda),L}
		\E\left[e^{\lambda|S_k|/L}\right].
	\end{align*}

	For every $\theta\in\R$,
	\[
		\E[e^{\theta X_1}]
		=\frac{\alpha}{2}e^\theta+(1-\alpha)
		+\frac{\alpha}{2}e^{-\theta}
		=1-\alpha+\alpha\cosh\theta.
	\]
	Since $S_k=X_1+\cdots+X_k$, independence gives
	\[
		\E[e^{\theta S_k}]
		=\bigl(1-\alpha+\alpha\cosh\theta\bigr)^k.
	\]
	Using $e^{|x|}\le e^x+e^{-x}$ and the symmetry of $S_k$, we obtain
	\begin{align*}
		\E\left[e^{\lambda|S_k|/L}\right]
		&\le
		\E\left[e^{\lambda S_k/L}\right]
		+\E\left[e^{-\lambda S_k/L}\right] \\
		&=2\bigl(1-\alpha+\alpha\cosh(\lambda/L)\bigr)^k \\
		&\le
		2\exp\left\{
		\alpha k\bigl(\cosh(\lambda/L)-1\bigr)
		\right\},
	\end{align*}
	where the last inequality follows from $1+x\le e^x$.

	For $z\ge0$,
	\[
		\cosh z-1
		=\int_0^z(z-s)\cosh s\,\mathrm ds
		\le \frac{z^2}{2}\cosh z.
	\]
	Therefore, since $L\ge1$ and $k\le TL^2$,
	\begin{align*}
		\alpha k\bigl(\cosh(\lambda/L)-1\bigr)
		&\le
		\frac{\alpha k\lambda^2}{2L^2}\cosh(\lambda/L) \\
		&\le
		\frac{\alpha T\lambda^2}{2}\cosh\lambda.
	\end{align*}
	It follows that
	\[
		\E\left[e^{\lambda|S_k|/L}\right]
		\le
		2\exp\left\{
		\frac{\alpha T\lambda^2}{2}\cosh\lambda
		\right\}.
	\]
	Taking the supremum over $x_i\in\Z/L$ completes the proof.
\end{proof}

\begin{proof}[Proof of Lemma~\ref{lem:holder-lebesgue-interpolation}]
	Write $\abs{f}_{\mathcal C_K^\gamma}:=\sup_{x\neq y\in K}\frac{|f(x)-f(y)|}{|x-y|^\gamma}$. Let $x\in K$ and $r\in(0,1]$ and $I:=K\cap[x-r,x+r]$ and note $|I|\gtrsim r$. Then
	\begin{align*}
		|f(x)|
		&\le \frac{1}{|I|}\int_I |f(y)|\,\mathrm dy + \frac{1}{|I|}\int_I |f(x)-f(y)|\,\mathrm dy\\
		&\le |I|^{-1/p}\norm{f}_{L^p_K} + r^\gamma\abs{f}_{\mathcal C_K^\gamma}\\
		&\lesssim r^{-1/p}\norm{f}_{L^p_K} + r^\gamma\abs{f}_{\mathcal C_K^\gamma},
	\end{align*}
	where the upper bound is independent of $x$. Taking supremum over $x\in K$ gives
	\begin{align*}
		\norm{f}_{\mathcal C_K} \lesssim r^{-1/p}\norm{f}_{L^p_K}+ r^\gamma\abs{f}_{\mathcal C_K^\gamma}.
	\end{align*}
	Choosing $r:=\min\left\{1,\left(\norm{f}_{L^p_K}/\abs{f}_{\mathcal C_K^\gamma}\right)^{\frac{1}{\gamma+1/p}}\right\}$ yields $\norm{f}_{\mathcal C_K}\lesssim \norm{f}_{L^p_K}^{1-\theta}\abs{f}_{\mathcal C_K^\gamma}^\theta+\norm{f}_{L^p_K}$ with $\theta=\frac{1}{\gamma p+1}$. Finally, $\abs{f}_{\mathcal C_K^\gamma}\le \norm{f}_{\mathcal C_K^\gamma}$, so the stated bound follows.
\end{proof}

\bibliographystyle{plain}
\bibliography{ref}

\begin{bibdiv}
\begin{biblist}

\bib{alsmeyer2026two}{article}{
      author={Alsmeyer, Gerold},
      author={Cordero, Fernando},
      author={Dopmeyer, Hannah},
       title={A two-size {W}right--{F}isher model: asymptotic analysis via uniform renewal theory},
        date={2026},
     journal={Stochastic Processes and their Applications},
      volume={192},
       pages={104812},
}

\bib{athreya2000uniqueness}{article}{
      author={Athreya, Siva},
      author={Tribe, Roger},
       title={Uniqueness for a class of one-dimensional stochastic {PDE}s using moment duality},
        date={2000},
     journal={The Annals of Probability},
      volume={28},
      number={4},
       pages={1711\ndash 1734},
}

\bib{baake2019modelling}{article}{
      author={Baake, Ellen},
      author={Gonz{\'a}lez~Casanova, Adri{\'a}n},
      author={Probst, Sebastian},
      author={Wakolbinger, Anton},
       title={Modelling and simulating {L}enski's long-term evolution experiment},
        date={2019},
     journal={Theoretical Population Biology},
      volume={127},
       pages={58\ndash 74},
}

\bib{boenkost2021haldaneweak}{article}{
      author={Boenkost, Florin},
      author={Gonz{\'a}lez~Casanova, Adri{\'a}n},
      author={Pokalyuk, Cornelia},
      author={Wakolbinger, Anton},
       title={Haldane's formula in {C}annings models: the case of moderately weak selection},
        date={2021},
     journal={Electronic Journal of Probability},
      volume={26},
       pages={1\ndash 36},
}

\bib{barnes2025wright}{article}{
      author={Barnes, Clayton},
      author={Mytnik, Leonid},
      author={Sun, Zhenyao},
       title={Wright--{F}isher stochastic heat equations with irregular drifts},
        date={2025},
     journal={Probability Theory and Related Fields},
}

\bib{cannings1974latent}{article}{
      author={Cannings, C.},
       title={The latent roots of certain {M}arkov chains arising in genetics: a new approach, {I}. {H}aploid models},
        date={1974},
     journal={Advances in Applied Probability},
      volume={6},
      number={2},
       pages={260\ndash 290},
}

\bib{conway1978functions}{book}{
      author={Conway, John~B.},
       title={Functions of one complex variable {I}},
     edition={2},
      series={Graduate Texts in Mathematics},
   publisher={Springer-Verlag},
     address={New York},
        date={1978},
      volume={11},
}

\bib{durrett2016genealogies}{article}{
      author={Durrett, Rick},
      author={Fan, Wai-Tong~Louis},
       title={Genealogies in expanding populations},
        date={2016},
     journal={The Annals of Applied Probability},
      volume={26},
      number={6},
       pages={3456\ndash 3490},
}

\bib{di2012hitchhiker}{article}{
      author={Di~Nezza, Eleonora},
      author={Palatucci, Giampiero},
      author={Valdinoci, Enrico},
       title={Hitchhikerʼs guide to the fractional {S}obolev spaces},
        date={2012},
     journal={Bulletin des Sciences Math{\'e}matiques},
      volume={136},
      number={5},
       pages={521\ndash 573},
}

\bib{ethier2009markov}{book}{
      author={Ethier, Stewart~N.},
      author={Kurtz, Thomas~G.},
       title={Markov processes: characterization and convergence},
   publisher={John Wiley \& Sons},
        date={2009},
}

\bib{etheridge2020rescaling}{article}{
      author={Etheridge, Alison~M.},
      author={V{\'e}ber, Amandine},
      author={Yu, Feng},
       title={Rescaling limits of the spatial {L}ambda-{F}leming-{V}iot process with selection},
        date={2020},
     journal={Electronic Journal of Probability},
      volume={25},
       pages={1\ndash 89},
}

\bib{fan2021stochastic}{article}{
      author={Fan, Wai-Tong~Louis},
       title={Stochastic {PDE}s on graphs as scaling limits of discrete interacting systems},
        date={2021},
     journal={Bernoulli},
      volume={27},
      number={3},
       pages={1899\ndash 1941},
}

\bib{gonzalez2016individual}{article}{
      author={Gonz{\'a}lez~Casanova, Adri{\'a}n},
      author={Kurt, Noemi},
      author={Wakolbinger, Anton},
      author={Yuan, Linglong},
       title={An individual-based model for the {L}enski experiment, and the deceleration of the relative fitness},
        date={2016},
     journal={Stochastic Processes and their Applications},
      volume={126},
      number={8},
       pages={2211\ndash 2252},
}

\bib{casanova2020wright}{article}{
      author={Gonz{\'a}lez~Casanova, Adri{\'a}n},
      author={Mir{\'o}~Pina, Ver{\'o}nica},
      author={Pardo, Juan~Carlos},
       title={The {W}right--{F}isher model with efficiency},
        date={2020},
     journal={Theoretical Population Biology},
      volume={132},
       pages={33\ndash 46},
}

\bib{GPP}{article}{
      author={Gonz{\'a}lez~Casanova, Adri{\'a}n},
      author={Pardo, Juan~Carlos},
      author={P{\'e}rez, Jos{\'e}~Luis},
       title={Branching processes with interactions: the subcritical cooperative regime},
        date={2021},
     journal={Advances in Applied Probability},
      volume={53},
      number={1},
       pages={251\ndash 278},
}

\bib{gillespie1974natural}{article}{
      author={Gillespie, John~H.},
       title={Natural selection for within-generation variance in offspring number},
        date={1974},
     journal={Genetics},
      volume={76},
      number={3},
       pages={601\ndash 606},
}

\bib{gillespie1975natural}{article}{
      author={Gillespie, John~H.},
       title={Natural selection for within-generation variance in offspring number {II}. {D}iscrete haploid models},
        date={1975},
     journal={Genetics},
      volume={81},
      number={2},
       pages={403\ndash 413},
}

\bib{hall2014martingale}{book}{
      author={Hall, Peter},
      author={Heyde, Christopher~C.},
       title={Martingale limit theory and its application},
   publisher={Academic Press},
        date={2014},
}

\bib{huillet2022asymptotic}{article}{
      author={Huillet, Thierry},
      author={M{\"o}hle, Martin},
       title={Asymptotic genealogies for a class of generalized {W}right--{F}isher models},
        date={2022},
     journal={Modern Stochastics: Theory and Applications},
      volume={9},
      number={1},
       pages={17\ndash 43},
}

\bib{hou2026subcritical}{misc}{
      author={Hou, Haojie},
      author={Sun, Zhenyao},
       title={On the subcritical self-catalytic branching {B}rownian motions},
        date={2026},
        note={arXiv:2501.16739},
}

\bib{iwata1987infinite}{article}{
      author={Iwata, Koichiro},
       title={An infinite dimensional stochastic differential equation with state space {$C(\mathbb R)$}},
        date={1987},
     journal={Probability Theory and Related Fields},
      volume={74},
      number={1},
       pages={141\ndash 159},
}

\bib{kingman1982coalescent}{article}{
      author={Kingman, J. F.~C.},
       title={The coalescent},
        date={1982},
     journal={Stochastic Processes and their Applications},
      volume={13},
      number={3},
       pages={235\ndash 248},
}

\bib{kliem2011convergence}{article}{
      author={Kliem, Sandra},
       title={Convergence of rescaled competing species processes to a class of {SPDE}s},
        date={2011},
     journal={Electronic Journal of Probability},
      volume={16},
       pages={618\ndash 657},
}

\bib{kimura1964stepping}{article}{
      author={Kimura, Motoo},
      author={Weiss, George~H.},
       title={The stepping stone model of population structure and the decrease of genetic correlation with distance},
        date={1964},
     journal={Genetics},
      volume={49},
      number={4},
       pages={561\ndash 576},
}

\bib{lawler2010random}{book}{
      author={Lawler, Gregory~F.},
      author={Limic, Vlada},
       title={Random walk: a modern introduction},
      series={Cambridge Studies in Advanced Mathematics},
   publisher={Cambridge University Press},
        date={2010},
      volume={123},
}

\bib{marchal2017sub}{article}{
      author={Marchal, Olivier},
      author={Arbel, Julyan},
       title={On the sub-{G}aussianity of the {B}eta and {D}irichlet distributions},
        date={2017},
     journal={Electronic Communications in Probability},
      volume={22},
      number={54},
       pages={1\ndash 14},
}

\bib{mohle2000total}{article}{
      author={M{\"o}hle, Martin},
       title={Total variation distances and rates of convergence for ancestral coalescent processes in exchangeable population models},
        date={2000},
     journal={Advances in Applied Probability},
      volume={32},
      number={4},
       pages={983\ndash 993},
}

\bib{mohle2001forward}{article}{
      author={M{\"o}hle, Martin},
       title={Forward and backward diffusion approximations for haploid exchangeable population models},
        date={2001},
     journal={Stochastic Processes and their Applications},
      volume={95},
      number={1},
       pages={133\ndash 149},
}

\bib{mohle2001classification}{article}{
      author={M{\"o}hle, Martin},
      author={Sagitov, Serik},
       title={A classification of coalescent processes for haploid exchangeable population models},
        date={2001},
     journal={The Annals of Probability},
      volume={29},
      number={4},
       pages={1547\ndash 1562},
}

\bib{mueller1995stochastic}{article}{
      author={Mueller, Carl},
      author={Tribe, Roger},
       title={Stochastic p.d.e.'s arising from the long range contact and long range voter processes},
        date={1995},
     journal={Probability Theory and Related Fields},
      volume={102},
      number={4},
       pages={519\ndash 545},
}

\bib{mueller2009some}{incollection}{
      author={Mueller, Carl},
       title={Some tools and results for parabolic stochastic partial differential equations},
        date={2009},
   booktitle={A minicourse on stochastic partial differential equations, {Lecture Notes in Math.}, vol. 1962, {Springer}, {Berlin}},
      editor={Khoshnevisan, Davar},
      editor={Rassoul-Agha, Firas},
      series={Lecture Notes in Mathematics},
      volume={1962},
   publisher={Springer},
     address={Berlin},
       pages={111\ndash 144},
}

\bib{ng2011dirichlet}{book}{
      author={Ng, Kai~Wang},
      author={Tian, Guo-Liang},
      author={Tang, Man-Lai},
       title={Dirichlet and related distributions: Theory, methods and applications},
   publisher={John Wiley \& Sons},
        date={2011},
}

\bib{pitman1999coalescents}{article}{
      author={Pitman, Jim},
       title={Coalescents with multiple collisions},
        date={1999},
     journal={The Annals of Probability},
      volume={27},
      number={4},
       pages={1870\ndash 1902},
}

\bib{sagitov1999general}{article}{
      author={Sagitov, Serik},
       title={The general coalescent with asynchronous mergers of ancestral lines},
        date={1999},
     journal={Journal of Applied Probability},
      volume={36},
      number={4},
       pages={1116\ndash 1125},
}

\bib{serfling1974probability}{article}{
      author={Serfling, R.~J.},
       title={Probability inequalities for the sum in sampling without replacement},
        date={1974},
     journal={The Annals of Statistics},
      volume={2},
      number={1},
       pages={39\ndash 48},
}

\bib{shiga1988stepping}{incollection}{
      author={Shiga, Tokuzo},
       title={Stepping stone models in population genetics and population dynamics},
        date={1988},
   booktitle={Stochastic processes in physics and engineering},
      series={Mathematics and Its Applications},
      volume={42},
   publisher={Reidel},
     address={Dordrecht},
       pages={345\ndash 355},
}

\bib{shiga1994two}{article}{
      author={Shiga, Tokuzo},
       title={Two contrasting properties of solutions for one-dimensional stochastic partial differential equations},
        date={1994},
     journal={Canadian Journal of Mathematics},
      volume={46},
      number={2},
       pages={415\ndash 437},
}

\bib{walsh1986introduction}{incollection}{
      author={Walsh, John~B.},
       title={An introduction to stochastic partial differential equations},
        date={1986},
   booktitle={{{\'E}cole d'{\'E}t{\'e} de Probabilit{\'e}s de Saint-Flour XIV---1984}, {Lecture Notes in Math.}, vol. 1180, {Springer}, {Berlin}},
      series={Lecture Notes in Math.},
      volume={1180},
   publisher={Springer},
       pages={265\ndash 439},
}

\end{biblist}
\end{bibdiv}

\end{document}